\documentclass[11pt,hidelinks]{article}

\usepackage{etoolbox} 
\usepackage{orcidlink}

\title{
Decay of connection probability in high-dimensional \\ continuum percolation}

\author{
Matthew Dickson\,\orcidlink{0000-0002-8629-4796}\thanks{Department of Mathematics,
	University of British Columbia,
	Vancouver, BC, Canada V6T 1Z2.
	\href{mailto:dickson@math.ubc.ca}{dickson@math.ubc.ca}.
	}
\and
Yucheng Liu\,\orcidlink{0000-0002-1917-8330}\thanks{Beijing International Center for Mathematical Research,
	Peking University,
	Beijing, China 100871 and Department of Mathematics,
	University of British Columbia,
	Vancouver, BC, Canada V6T 1Z2.
	\href{mailto:yliu135@pku.edu.cn}{yliu135@pku.edu.cn}.
	}
}
\date{\vspace{-5ex}} 

\usepackage{amsmath, amssymb, amscd, amsthm, amsfonts}
\usepackage{graphicx} 
\usepackage{hyperref} 

\usepackage{booktabs}
\usepackage{xcolor}
\usepackage{ dsfont, bm}
\usepackage[title]{appendix}
\usepackage{comment}
\usepackage{enumerate}
\usepackage{enumitem}
\usepackage{cite}

\usepackage[textwidth=480pt,textheight=650pt,centering]{geometry} 

\usepackage{tikz}
\usetikzlibrary{plotmarks}
\usetikzlibrary{shapes.misc}

\theoremstyle{plain}
\newtheorem{theorem}{Theorem}[section]
\newtheorem{lemma}[theorem]{Lemma}

\newtheorem{proposition}[theorem]{Proposition}

\newtheorem{definition}[theorem]{Definition}
\newtheorem{assumption}[theorem]{Assumption}

\theoremstyle{remark}
\newtheorem{remark}[theorem]{Remark}

\numberwithin{equation}{section}

\newcommand{\ie}{i.e.}
\newcommand{\eg}{e.g.}

\newcommand{\eps}{\varepsilon}

\newcommand{\N}{\mathbb{N}}
\newcommand{\Z}{\mathbb{Z}}

\newcommand{\R}{\mathbb{R}}

\renewcommand{\P}{\mathbb{P}}
\newcommand{\T}{\mathbb{T}}

\newcommand{\grad}{\nabla}
\newcommand{\inv}{^{-1}}
\renewcommand{\(}{\left(}
\renewcommand{\)}{\right)}
\newcommand{\half}{\frac{1}{2}}

\newcommand{\1}{\mathds{1}}

\newcommand{\nl}{\nonumber \\}
\providecommand{\bigabs}[1]{\big\lvert#1\big\rvert}

\providecommand{\biggabs}[1]{\bigg\lvert#1\bigg\rvert}

\providecommand{\bignorm}[1]{\big\lVert#1\big\rVert}
\providecommand{\Bignorm}[1]{\Big\lVert#1\Big\rVert}
\providecommand{\biggnorm}[1]{\bigg\lVert#1\bigg\rVert}
\providecommand{\Biggnorm}[1]{\Bigg\lVert#1\Bigg\rVert}

\newcommand{\supk}[1]{^{(#1)}}

\newcommand{\supzero}{^{(0)}}

\newcommand{\supn}{^{(n)}}

\newcommand{\supa}{^{(a)}}
\newcommand{\supb}{^{(b)}}
\newcommand{\supab}{^{(a,b)}}
\newcommand{\supphi}{^{(\phi)}}
\newcommand{\suptheta}{^{(\theta)}}
\newcommand{\supthetao}{^{(\theta),\circ}}
\newcommand{\supzeroo}{^{(0),\circ}}
\newcommand{\supphitheta}{^{(\phi,\theta)}}
\newcommand{\supphizero}{^{(\phi,0)}}
\newcommand{\phitheta}{{\phi+\theta}}
\newcommand{\ab}{{a+b}}
\newcommand{\plus}{{\lambda, +}}

\newcommand{\D}{\mathrm{d}}

\newcommand{\crit}{_{\lambda_c}}
\newcommand{\blackcirc}{\bullet}
\newcommand{\connect}{\leftrightarrow}
\newcommand{\spec}{_{\mathrm{special}}}
\newcommand{\zuo}{_{\mathrm{left}}}
\newcommand{\you}{_{\mathrm{right}}}

\providecommand{\Rd}{{\mathbb R^d}}
\providecommand{\Zd}{\mathbb Z^d}

\newcommand{\dc}{d_c}

\usepackage{mathtools}
\usepackage{caption}
\usepackage{subcaption}
\usepackage{todonotes}

\newcommand{\Lacelam}{\Pi_\lambda}

\newcommand{\tlam}{\tau_\lambda}
\newcommand{\taulam}{\tlam}
\newcommand{\taul}{\tlam}
\newcommand{\lam}{\lambda}

\newcommand{\tlamo}{\tau_\lambda^\circ}

\newcommand{\connf}{\varphi}

\newcommand{\adj}{\connf}

\DeclarePairedDelimiter\abs{\lvert}{\rvert}
\DeclarePairedDelimiter\norm{\lVert}{\rVert}
\DeclarePairedDelimiter\floor{\lfloor}{\rfloor}

\newcommand{\dd}{\mathrm{d}} 

\providecommand{\Rd}{{\mathbb R^d}}
\providecommand{\Zd}{{\mathbb Z^d}}
\newcommand{\vertiii}[1]{{\left\vert\kern-0.25ex\left\vert\kern-0.25ex\left\vert #1 
    \right\vert\kern-0.25ex\right\vert\kern-0.25ex\right\vert}}
\newcommand{\orig}{0}

\newcommand{\EndBlock}{\bar{\psi}}

\newcommand{\adjLine}{\raisebox{3pt}{\tikz[scale=0.6]{
	\draw (0,0) -- (1,0); 
	\draw (0.5,0) circle (0pt) node[above]{\scriptsize$\sim$};}}}
\newcommand{\tauLine}{\raisebox{3pt}{\tikz[scale=0.6]{
	\draw (0,0) -- (1,0); }}}
\newcommand{\tauCircleLine}{\raisebox{3pt}{\tikz[scale=0.6]{
	\draw (0,0) -- (1,0); 
	\draw (0.5,0.2) circle (2pt); }}}

\newcommand{\weightedLine}{\raisebox{2pt}{\tikz[scale=0.6]{
	\draw[thick, double=white] (0,0) -- (1,0); }}}
\newcommand{\weightedLineRed}{\raisebox{2pt}{\tikz[scale=0.6]{
	\draw[thick, double=white,red] (0,0) -- (1,0); }}}

\newcommand{\weightedCircleLine}{\raisebox{2pt}{\tikz[scale=0.6]{
	\draw[thick, double=white] (0,0) -- (1,0);
	\draw (0.5,0.2) circle (2pt); }}}

\newcommand{\Triangleab}{\raisebox{-15pt}{
    \begin{tikzpicture}[scale=0.6]
    	\draw (0,0) -- (1,0.6) -- (1,-0.6) -- cycle;
        \filldraw[fill=white] (0,0) circle (2pt);
        \filldraw[fill=white] (1,0.6) circle (2pt) node[right]{$a$};
        \filldraw[fill=white] (1,-0.6) circle (2pt) node[right]{$b$};
    \end{tikzpicture}
}}
\newcommand{\TriangleRightcd}{\raisebox{-15pt}{
    \begin{tikzpicture}[scale=0.6]
    	\draw (0,0.6) -- (1,0) -- (0,-0.6) -- cycle;
        \filldraw[fill=white] (0,0.6) circle (2pt) node[left]{$c$};
        \filldraw[fill=white] (1,0) circle (2pt) node[right]{$0$};
        \filldraw[fill=white] (0,-0.6) circle (2pt) node[left]{$d$};
    \end{tikzpicture}
}}
\newcommand{\sideTriangle}[2]{\raisebox{-9pt}{
    \begin{tikzpicture}[scale=0.6]
        \draw (0,0) -- (0.5,0.8) -- (1,0) -- cycle;
        \filldraw[fill=white] (0,0) circle (2pt) node[left]{$#1$};
        \filldraw[fill=white] (1,0) circle (2pt) node[right]{$#2$};
        \filldraw (0.5,0.8) circle (2pt); 
    \end{tikzpicture}
}}
\newcommand{\sideTriangleBlue}[2]{\raisebox{-9pt}{
    \begin{tikzpicture}[scale=0.6]
        \draw (0,0) -- (0.5,0.8) -- (1,0);
        \draw[thick, double=white,blue] (0,0) -- (1,0);
        \filldraw[fill=white] (0,0) circle (2pt) node[left]{$#1$};
        \filldraw[fill=white] (1,0) circle (2pt) node[right]{$#2$};
        \filldraw (0.5,0.8) circle (2pt); 
    \end{tikzpicture}
}}

\newcommand{\WeightedBubble}{\raisebox{-16pt}{
    \begin{tikzpicture}[scale=0.6]
        \draw[thick, red, double=white] (1,0.6) -- (2,0.6);
        \draw[thick, blue, double=white] (2,-0.6) -- (1,-0.6);
        \filldraw[fill=white] (1,0.6) circle (2pt) node[left]{$a-x$};
        \filldraw[fill=white] (1,-0.6) circle (2pt) node[left]{$b-x$};
        \filldraw[fill=white] (2,0.6) circle (2pt) node[right]{$c$};
        \filldraw[fill=white] (2,-0.6) circle (2pt) node[right]{$d$};
    \end{tikzpicture}
}}

\newcommand{\PiOneOneOne}{\raisebox{-9pt}{
    \begin{tikzpicture}[scale=0.6]
        \draw (0,0) -- (1,0.6) -- (1,-0.6) -- cycle;
        \draw (1,0.6) -- (2,0.6) --(2,-0.6) -- (1,-0.6);
        \draw (2,0.6) -- (3,0) -- (2,-0.6);
        \draw (1.5,-0.4) circle (2pt);
        \filldraw[fill=white] (0,0) circle (2pt); 
        \filldraw (1,0.6) circle (2pt);
        \filldraw (1,-0.6) circle (2pt);
        \filldraw (2,0.6) circle (2pt);
        \filldraw (2,-0.6) circle (2pt);
        \filldraw[fill=white] (3,0) circle (2pt) node[right]{$x$};
    \end{tikzpicture}
}}
\newcommand{\PiOneOneTwo}{\raisebox{-9pt}{
    \begin{tikzpicture}[scale=0.6]
        \draw (0,0) -- (1,0.6) -- (1,-0.6) -- cycle;
        \draw (1,0.6) -- (2,0) -- (1,-0.6);
        \filldraw[fill=white] (0,0) circle (2pt); 
        \filldraw (1,0.6) circle (2pt);
        \filldraw (1,-0.6) circle (2pt);
        \filldraw[fill=white] (2,0) circle (2pt) node[right]{$x$};
    \end{tikzpicture}
}}
\newcommand{\PiOneTwoOne}{\raisebox{-9pt}{
    \begin{tikzpicture}[scale=0.6]
        \draw (0,0) -- (1,0.6) -- (1,-0.6) -- cycle;
        \draw (1,0.6) -- (2,0.6) --(2,-0.6) -- (1,-0.6);
        \draw (2,0.6) -- (3,0) -- (2,-0.6);
        \draw (1.5,-0.4) circle (2pt);
        \filldraw (0,0) circle (2pt);
        \filldraw[fill=white] (1,0.6) circle (2pt);
        \filldraw (1,-0.6) circle (2pt);
        \filldraw (2,0.6) circle (2pt);
        \filldraw (2,-0.6) circle (2pt);
        \filldraw[fill=white] (3,0) circle (2pt) node[right]{$x$};
    \end{tikzpicture}
}}
\newcommand{\PiOneTwoTwo}{\raisebox{-9pt}{
    \begin{tikzpicture}[scale=0.6]
        \draw (0,0) -- (1,0.6) -- (1,-0.6) -- cycle;
        \draw (1,0.6) -- (2,0) -- (1,-0.6);
        \filldraw (0,0) circle (2pt);
        \filldraw[fill=white] (1,0.6) circle (2pt);
        \filldraw (1,-0.6) circle (2pt);
        \filldraw[fill=white] (2,0) circle (2pt) node[right]{$x$};
    \end{tikzpicture}
}}
\newcommand{\PiOneThreeOne}{\raisebox{-9pt}{
    \begin{tikzpicture}[scale=0.6]
        \draw (0,0.6) -- (1,0.6) -- (2,0) -- (1,-0.6) -- (0,-0.6) -- cycle;
        \draw (1,0.6) -- (1,-0.6);
        \draw (0,0) circle (0pt) node[rotate=90,above]{$\sim$};
        \draw (0.5,-0.4) circle (2pt);
        \filldraw[fill=white] (0,0.6) circle (2pt);
        \filldraw (1,0.6) circle (2pt);
        \filldraw[fill=white] (2,0) circle (2pt) node[right]{$x$};
        \filldraw (0,-0.6) circle (2pt);
        \filldraw (1,-0.6) circle (2pt);
    \end{tikzpicture}
}}
\newcommand{\PiOneThreeTwo}{\raisebox{-9pt}{
    \begin{tikzpicture}[scale=0.6]
        \draw (0,0.6) -- (1,0) -- (0,-0.6) -- cycle;
        \draw (0,0) circle (0pt) node[rotate=90,above]{$\sim$};
        \filldraw[fill=white] (0,0.6) circle (2pt);
        \filldraw[fill=white] (1,0) circle (2pt) node[right]{$x$};
        \filldraw (0,-0.6) circle (2pt);
    \end{tikzpicture}
}}

\newcommand{\PiOneWeightedOne}{\raisebox{-9pt}{
    \begin{tikzpicture}[scale=0.6]
        \draw[thick, red, double=white] (0,0) -- (1,0.6);
        \draw (1,0.6) -- (1,-0.6);
        \draw[thick, blue, double=white] (1,-0.6) -- (0,0);
        \draw (1,0.6) -- (2,0.6) --(2,-0.6) -- (1,-0.6);
        \draw (2,0.6) -- (3,0) -- (2,-0.6);
        \draw (1.5,-0.4) circle (2pt);
        \filldraw[fill=white] (0,0) circle (2pt); 
        \filldraw (1,0.6) circle (2pt);
        \filldraw (1,-0.6) circle (2pt);
        \filldraw (2,0.6) circle (2pt);
        \filldraw (2,-0.6) circle (2pt);
        \filldraw[fill=white] (3,0) circle (2pt) node[right]{$x$};
    \end{tikzpicture}
}}
\newcommand{\PiOneWeightedTwo}{\raisebox{-9pt}{
    \begin{tikzpicture}[scale=0.6]
        \draw (0,0) -- (1,0.6);
        \draw (1,0.6) -- (1,-0.6);
        \draw[thick, blue, double=white] (1,-0.6) -- (0,0);
        \draw[thick, red, double=white] (1,0.6) -- (2,0.6);
        \draw (2,0.6) --(2,-0.6) -- (1,-0.6);
        \draw (2,0.6) -- (3,0) -- (2,-0.6);
        \draw (1.5,-0.4) circle (2pt);
        \filldraw[fill=white] (0,0) circle (2pt); 
        \filldraw (1,0.6) circle (2pt);
        \filldraw (1,-0.6) circle (2pt);
        \filldraw (2,0.6) circle (2pt);
        \filldraw (2,-0.6) circle (2pt);
        \filldraw[fill=white] (3,0) circle (2pt) node[right]{$x$};
    \end{tikzpicture}
}}
\newcommand{\PiOneWeightedThree}{\raisebox{-9pt}{
    \begin{tikzpicture}[scale=0.6]
        \draw (0,0) -- (1,0.6);
        \draw (1,0.6) -- (1,-0.6);
        \draw[thick, blue, double=white] (1,-0.6) -- (0,0);
        \draw (1,0.6) -- (2,0.6) --(2,-0.6) -- (1,-0.6);
        \draw[thick, red, double=white] (2,0.6) -- (3,0);
        \draw (3,0) -- (2,-0.6);
        \draw (1.5,-0.4) circle (2pt);
        \filldraw[fill=white] (0,0) circle (2pt); 
        \filldraw (1,0.6) circle (2pt);
        \filldraw (1,-0.6) circle (2pt);
        \filldraw (2,0.6) circle (2pt);
        \filldraw (2,-0.6) circle (2pt);
        \filldraw[fill=white] (3,0) circle (2pt) node[right]{$x$};
    \end{tikzpicture}
}}
\newcommand{\PiOneCaseOne}{\raisebox{-9pt}{
    \begin{tikzpicture}[scale=0.6]
        \draw (0,0) -- (1,0.6);
        \draw (1,0.6) -- (1,-0.6);
        \draw (1,-0.6) -- (0,0);
        \draw[thick, red, double=white] (1,0.6) -- (2,0.6);
        \draw (2,0.6) --(2,-0.6);
        \draw[thick, blue, double=white] (2,-0.6) -- (1,-0.6);
        \draw (2,0.6) -- (3,0) -- (2,-0.6);
        \filldraw[fill=white] (0,0) circle (2pt); 
        \filldraw (1,0.6) circle (2pt);
        \filldraw (1,-0.6) circle (2pt);
        \filldraw (2,0.6) circle (2pt);
        \filldraw (2,-0.6) circle (2pt);
        \filldraw[fill=white] (3,0) circle (2pt) node[right]{$x$};
    \end{tikzpicture}
}}
\newcommand{\PiOneCaseOneRight}{\raisebox{-15pt}{
    \begin{tikzpicture}[scale=0.6]
        \draw[thick, red, double=white] (1,0.6) -- (2,0.6);
        \draw (2,0.6) --(2,-0.6);
        \draw[thick, blue, double=white] (2,-0.6) -- (1,-0.6);
        \draw (2,0.6) -- (3,0) -- (2,-0.6);
        \filldraw[fill=white] (1,0.6) circle (2pt) node[left]{$a$};
        \filldraw[fill=white] (1,-0.6) circle (2pt) node[left]{$b$};
        \filldraw (2,0.6) circle (2pt);
        \filldraw (2,-0.6) circle (2pt);
        \filldraw[fill=white] (3,0) circle (2pt) node[right]{$x$};
    \end{tikzpicture}
}}
\newcommand{\PiOneCaseOneRightShift}{\raisebox{-16pt}{
    \begin{tikzpicture}[scale=0.6]
        \draw[thick, red, double=white] (1,0.6) -- (2,0.6);
        \draw (2,0.6) --(2,-0.6);
        \draw[thick, blue, double=white] (2,-0.6) -- (1,-0.6);
        \draw (2,0.6) -- (3,0) -- (2,-0.6);
        \filldraw[fill=white] (1,0.6) circle (2pt) node[left]{$a-x$};
        \filldraw[fill=white] (1,-0.6) circle (2pt) node[left]{$b-x$};
        \filldraw (2,0.6) circle (2pt);
        \filldraw (2,-0.6) circle (2pt);
        \filldraw[fill=white] (3,0) circle (2pt) node[right]{$0$};
    \end{tikzpicture}
}}
\newcommand{\PiOneCaseOneSecond}{\raisebox{-9pt}{
    \begin{tikzpicture}[scale=0.6]
        \draw (0,0) -- (1,0.6) -- (1,-0.6) -- cycle;
        \draw (1,0.6) -- (2,0.6) --(2,-0.6);
        \draw[thick, blue, double=white] (2,-0.6) -- (1,-0.6);
        \draw[thick, red, double=white] (2,0.6) -- (3,0);
        \draw (3,0) -- (2,-0.6);
        \draw (1.5,-0.4) circle (2pt);
        \filldraw[fill=white] (0,0) circle (2pt); 
        \filldraw (1,0.6) circle (2pt);
        \filldraw (1,-0.6) circle (2pt);
        \filldraw (2,0.6) circle (2pt);
        \filldraw (2,-0.6) circle (2pt);
        \filldraw[fill=white] (3,0) circle (2pt) node[right]{$x$};
    \end{tikzpicture}
}}
\newcommand{\PiOneCaseOneSecondLeft}{\raisebox{-15pt}{
    \begin{tikzpicture}[scale=0.6]
        \draw (0,0) -- (1,0.6) -- (1,-0.6) -- cycle;
        \draw (1,0.6) -- (2,0.6) --(2,-0.6);
        \draw[thick, blue, double=white] (2,-0.6) -- (1,-0.6);
        \draw (1.5,-0.4) circle (2pt);
        \filldraw[fill=white] (0,0) circle (2pt); 
        \filldraw (1,0.6) circle (2pt);
        \filldraw (1,-0.6) circle (2pt);
        \filldraw[fill=white] (2,0.6) circle (2pt) node[right]{$c$};
        \filldraw[fill=white] (2,-0.6) circle (2pt) node[right]{$d$};
    \end{tikzpicture}
}}
\newcommand{\PiOneCaseOneSecondRight}{\raisebox{-15pt}{
    \begin{tikzpicture}[scale=0.6]
        \draw[thick, red, double=white] (2,0.6) -- (3,0);
        \draw (3,0) -- (2,-0.6);
        \filldraw[fill=white] (2,0.6) circle (2pt) node[left]{$c$};
        \filldraw[fill=white] (2,-0.6) circle (2pt) node[left]{$d$};
        \filldraw[fill=white] (3,0) circle (2pt) node[right]{$x$};
    \end{tikzpicture}
}}
\newcommand{\PiOneCaseOneSecondMid}{\raisebox{-15pt}{
    \begin{tikzpicture}[scale=0.6]
        \draw (1,0.6) -- (2,0.6) --(2,-0.6);
        \draw[thick, blue, double=white] (2,-0.6) -- (1,-0.6);
        \draw (1.5,-0.4) circle (2pt);
        \filldraw[fill=white] (1,0.6) circle (2pt) node[left]{$a$};
        \filldraw[fill=white] (1,-0.6) circle (2pt) node[left]{$b$};
        \filldraw (2,0.6) circle (2pt);
        \filldraw (2,-0.6) circle (2pt);
    \end{tikzpicture}
}}

\newcommand{\PiOneCaseTwo}{\raisebox{-9pt}{
    \begin{tikzpicture}[scale=0.6]
        \draw (0,0) -- (1,0.6) -- (1.5,-0.6);
        \draw[thick, blue, double=white] (0,0) -- (1.5,-0.6);
        \draw[thick, red, double=white] (1,0.6) -- (2,0.6);
        \draw (2,0.6) -- (3,0) -- (1.5,-0.6) -- cycle;
        \filldraw[fill=white] (0,0) circle (2pt); 
        \filldraw (1,0.6) circle (2pt);
        \filldraw (1.5,-0.6) circle (2pt);
        \filldraw (2,0.6) circle (2pt);
        \filldraw[fill=white] (3,0) circle (2pt) node[right]{$x$};
    \end{tikzpicture}
}}

\newcommand{\PiNCaseOneJOneStart}[3]{\raisebox{-15pt}{
    \begin{tikzpicture}[scale=0.6]
    \draw[] (0,0) -- (1,0.6) -- (1,-0.6) -- (0,0);
    \draw[thick, blue, double=white] (1,-0.6) -- (0,0);
    \draw[thick, red, double=white] (0,0) -- (1,0.6);
    \draw[] (1,0.6) -- (2,0.6);
    \draw[] (1,-0.6) -- (2,-0.6);
	\filldraw[fill=white] (1.5,-0.4) circle (2pt);
    \filldraw (1,-0.6) circle (2pt);
    \filldraw (1,0.6) circle (2pt);
    \filldraw[fill=white] (0,0) circle (2pt) node[left]{$#1$};
    \filldraw[fill=white] (2,0.6) circle (2pt) node[right]{$#2$};
    \filldraw[fill=white] (2,-0.6) circle (2pt) node[right]{$#3$};
    \end{tikzpicture}
}}

\newcommand{\PiNCaseTwoJOnePtThreeVTwoSwapL}[4]{\raisebox{-15pt}{
    \begin{tikzpicture}[scale=0.6]
    \draw[] (1,0.6) -- (1,-0.6) -- (0,-0.6);
    \draw[] (1,0.6) -- (2,0.6) -- (2,-0.6) -- (1,-0.6);
    \draw[thick, blue, double=white] (1,0.6) -- (2,0.6);
    \draw[thick, red, double=white] (0,-0.6) -- (1,-0.6);
	\filldraw[fill=white] (2.5,-0.4) circle (2pt);
    \filldraw[fill=white] (1,0.6) circle (2pt) node[left]{$#1$};
    \filldraw[fill=white] (0,-0.6) circle (2pt) node[left]{$#2$};
    \filldraw (1,-0.6) circle (2pt);
    \filldraw[fill=white] (3,0.6) circle (2pt) node[right]{$#3$};
    \filldraw[fill=white] (3,-0.6) circle (2pt) node[right]{$#4$};
    \draw[] (2,0.6) -- (3,0.6);
    \draw[] (2,-0.6) -- (3,-0.6);
    \filldraw (2,0.6) circle (2pt);
    \filldraw (2,-0.6) circle (2pt);
    \end{tikzpicture}
}}

\newcommand{\PiNCaseTwoJOnePtThreeVTwoSwapR}[4]{\raisebox{-15pt}{
    \begin{tikzpicture}[scale=0.6]
    \draw[] (1,0.6) -- (1,-0.6) -- (0,-0.6);
    \draw[] (1,0.6) -- (2,0.6) -- (2,-0.6) -- (1,-0.6);
    \draw[] (1,0.6) -- (2,0.6);
    \draw[thick, red, double=white] (0,-0.6) -- (1,-0.6);
	\filldraw[fill=white] (2.5,-0.4) circle (2pt);
    \filldraw[fill=white] (1,0.6) circle (2pt) node[left]{$#1$};
    \filldraw[fill=white] (0,-0.6) circle (2pt) node[left]{$#2$};
    \filldraw (1,-0.6) circle (2pt);
    \filldraw[fill=white] (3,0.6) circle (2pt) node[right]{$#3$};
    \filldraw[fill=white] (3,-0.6) circle (2pt) node[right]{$#4$};
    \draw[thick, blue, double=white] (2,0.6) -- (3,0.6);
    \draw[] (2,-0.6) -- (3,-0.6);
    \filldraw (2,0.6) circle (2pt);
    \filldraw (2,-0.6) circle (2pt);
    \end{tikzpicture}
}}

\newcommand{\PiNCaseTwoJOnePtFiveVTwo}[4]{\raisebox{-15pt}{
    \begin{tikzpicture}[scale=0.6]
    \draw[] (0,0.6) -- (1,0.6) -- (1,-0.6) -- (0,-0.6);
    \draw[] (1,0.6) -- (2,0.6) -- (2,-0.6) -- (1,-0.6);
    \draw[thick, red, double=white] (0,0.6) -- (1,0.6);
    \draw[thick, blue, double=white] (0,-0.6) -- (1,-0.6);
	\filldraw[fill=white] (2.5,-0.4) circle (2pt);
    \filldraw[fill=white] (0,0.6) circle (2pt) node[left]{$#1$};
    \filldraw[fill=white] (0,-0.6) circle (2pt) node[left]{$#2$};
    \filldraw (1,0.6) circle (2pt);
    \filldraw (1,-0.6) circle (2pt);
    \filldraw[fill=white] (3,0.6) circle (2pt) node[right]{$#3$};
    \filldraw[fill=white] (3,-0.6) circle (2pt) node[right]{$#4$};
    \draw[] (2,0.6) -- (3,0.6);
    \draw[] (2,-0.6) -- (3,-0.6);
    \filldraw (2,0.6) circle (2pt);
    \filldraw (2,-0.6) circle (2pt);
    \end{tikzpicture}
}}

\newcommand{\PiNCaseTwoJOnePtSixVTwo}[4]{\raisebox{-15pt}{
    \begin{tikzpicture}[scale=0.6]
    \draw[] (0,0.6) -- (1,0.6) -- (1,-0.6) -- (0,-0.6);
    \draw[] (1,0.6) -- (2,0.6) -- (2,-0.6) -- (1,-0.6);
    \draw[thick, red, double=white] (1,0.6) -- (2,0.6);
    \draw[thick, blue, double=white] (0,-0.6) -- (1,-0.6);
	\filldraw[fill=white] (2.5,-0.4) circle (2pt);
	\filldraw[fill=white] (0.5,0.8) circle (2pt);
    \filldraw[fill=white] (0,0.6) circle (2pt) node[left]{$#1$};
    \filldraw[fill=white] (0,-0.6) circle (2pt) node[left]{$#2$};
    \filldraw (1,0.6) circle (2pt);
    \filldraw (1,-0.6) circle (2pt);
    \filldraw[fill=white] (3,0.6) circle (2pt) node[right]{$#3$};
    \filldraw[fill=white] (3,-0.6) circle (2pt) node[right]{$#4$};
    \draw[] (2,0.6) -- (3,0.6);
    \draw[] (2,-0.6) -- (3,-0.6);
    \filldraw (2,0.6) circle (2pt);
    \filldraw (2,-0.6) circle (2pt);
    \end{tikzpicture}
}}

\newcommand{\PiNCaseTwoJOnePtSevenVTwo}[4]{\raisebox{-15pt}{
    \begin{tikzpicture}[scale=0.6]
    \draw[] (0,0.6) -- (1,0.6) -- (1,-0.6) -- (0,-0.6);
    \draw[] (1,0.6) -- (2,0.6) -- (2,-0.6) -- (1,-0.6);
    \draw[thick, red, double=white] (0,0.6) -- (1,0.6);
    \draw[thick, blue, double=white] (1,-0.6) -- (2,-0.6);
	\filldraw[fill=white] (2.5,-0.4) circle (2pt);
    \filldraw[fill=white] (0,0.6) circle (2pt) node[left]{$#1$};
    \filldraw[fill=white] (0,-0.6) circle (2pt) node[left]{$#2$};
    \filldraw (1,0.6) circle (2pt);
    \filldraw (1,-0.6) circle (2pt);
    \filldraw[fill=white] (3,0.6) circle (2pt) node[right]{$#3$};
    \filldraw[fill=white] (3,-0.6) circle (2pt) node[right]{$#4$};
    \draw[] (2,0.6) -- (3,0.6);
    \draw[] (2,-0.6) -- (3,-0.6);
    \filldraw (2,0.6) circle (2pt);
    \filldraw (2,-0.6) circle (2pt);
    \end{tikzpicture}
}}

\newcommand{\PiNCaseTwoJOnePtEightVTwo}[4]{\raisebox{-15pt}{
    \begin{tikzpicture}[scale=0.6]
    \draw[] (0,0.6) -- (1,0.6) -- (1,-0.6) -- (0,-0.6);
    \draw[] (1,0.6) -- (2,0.6) -- (2,-0.6) -- (1,-0.6);
    \draw[thick, red, double=white] (1,0.6) -- (2,0.6);
    \draw[thick, blue, double=white] (1,-0.6) -- (2,-0.6);
	\filldraw[fill=white] (0.5,0.8) circle (2pt);
	\filldraw[fill=white] (2.5,-0.4) circle (2pt);
    \filldraw[fill=white] (0,0.6) circle (2pt) node[left]{$#1$};
    \filldraw[fill=white] (0,-0.6) circle (2pt) node[left]{$#2$};
    \filldraw (1,0.6) circle (2pt);
    \filldraw (1,-0.6) circle (2pt);
    \filldraw[fill=white] (3,0.6) circle (2pt) node[right]{$#3$};
    \filldraw[fill=white] (3,-0.6) circle (2pt) node[right]{$#4$};
    \draw[] (2,0.6) -- (3,0.6);
    \draw[] (2,-0.6) -- (3,-0.6);
    \filldraw (2,0.6) circle (2pt);
    \filldraw (2,-0.6) circle (2pt);
    \end{tikzpicture}
}}

\newcommand{\PiNDiagramFive}[4]{\raisebox{-15pt}{
    \begin{tikzpicture}[scale=0.6]
    \filldraw (-0.3,0) circle (0pt);
    \draw (1.5,0.6) -- (1,-0.6);
    \draw (0,-0.6) -- (1,-0.6) -- (2,-0.6) -- (1.5,0.6);
    \draw[thick, red, double=white] (0,-0.6) -- (1,-0.6);
    \draw (1.5,0.6) -- (3,0.6);
    \draw[thick, blue, double=white] (1.5,0.6) -- (3,0.6);
    \filldraw (1,-0.6) circle (2pt);
    \filldraw[fill=white] (1.5,0.6) circle (2pt) node[left]{$#1$};
    \filldraw[fill=white] (0,-0.6) circle (2pt) node[left]{$#2$};
    \filldraw[fill=white] (3,0.6) circle (2pt) node[right]{$#3$};
    \filldraw[fill=white] (2,-0.6) circle (2pt) node[right]{$#4$};
    \end{tikzpicture}
}}

\newcommand{\PiNDiagramSix}[4]{\raisebox{-15pt}{
    \begin{tikzpicture}[scale=0.6]
    \filldraw (-0.3,0) circle (0pt);
    \draw (1.5,0.6) -- (1,-0.6);
    \draw (0,-0.6) -- (1,-0.6) -- (2,-0.6) -- (1.5,0.6);
    \draw (1.5,0.6) -- (3,0.6);
    \draw[thick, blue, double=white] (1.5,0.6) -- (3,0.6);
    \draw[thick, red, double=white] (1,-0.6) -- (2,-0.6);
    \filldraw (1,-0.6) circle (2pt);
    \filldraw[fill=white] (1.5,0.6) circle (2pt) node[left]{$#1$};
    \filldraw[fill=white] (0,-0.6) circle (2pt) node[left]{$#2$};
    \filldraw[fill=white] (3,0.6) circle (2pt) node[right]{$#3$};
    \filldraw[fill=white] (2,-0.6) circle (2pt) node[right]{$#4$};
    \end{tikzpicture}
}}

\newcommand{\PiNDiagramTen}[4]{\raisebox{-15pt}{
    \begin{tikzpicture}[scale=0.6]
    \draw (1.5,0.6) -- (1.5,-0.1) -- (1,-0.6);
    \draw (0,-0.6) -- (1,-0.6) -- (2,-0.6) -- (1.5,-0.1);
    \draw (1.5,0.6) -- (3,0.6);
    \draw[thick, blue, double=white] (1.5,0.6) -- (3,0.6);
    \draw[thick, red, double=white] (1,-0.6) -- (2,-0.6);
    \filldraw (1,-0.6) circle (2pt);
    \filldraw (1.5,-0.1) circle (2pt);
    \filldraw[fill=white] (1.5,0.6) circle (2pt) node[left]{$#1$};
    \filldraw[fill=white] (0,-0.6) circle (2pt) node[left]{$#2$};
    \filldraw[fill=white] (3,0.6) circle (2pt) node[right]{$#3$};
    \filldraw[fill=white] (2,-0.6) circle (2pt) node[right]{$#4$};
    \end{tikzpicture}
}}

\newcommand{\PiNCaseTwoJTwoPtTwelveSwap}[4]{\raisebox{-15pt}{
    \begin{tikzpicture}[scale=0.6]
    \draw (0,0.6) -- (1.5,0.6) -- (1.5,-0.1) -- (1,-0.6);
    \draw (0,-0.6) -- (1,-0.6) -- (2,-0.6) -- (1.5,-0.1);
    \draw[thick, blue, double=white] (0,0.6) -- (1.5,0.6);
    \draw[thick, red, double=white] (1,-0.6) -- (2,-0.6);
    \draw (1.5,0.6) -- (3,0.6);
    \filldraw (1,-0.6) circle (2pt);
    \filldraw (1.5,-0.1) circle (2pt);
    \filldraw (1.5,0.6) circle (2pt);
    \filldraw[fill=white] (0,0.6) circle (2pt) node[left]{$#1$};
    \filldraw[fill=white] (0,-0.6) circle (2pt) node[left]{$#2$};
    \filldraw[fill=white] (3,0.6) circle (2pt) node[right]{$#3$};
    \filldraw[fill=white] (2,-0.6) circle (2pt) node[right]{$#4$};
    \end{tikzpicture}
}}

\newcommand{\PiNDiagramEighteen}[4]{\raisebox{-15pt}{
    \begin{tikzpicture}[scale=0.6]
    \draw (0,0.6) -- (1.5,0.6) -- (1.5,-0.1) -- (1,-0.6);
    \draw (0,-0.6) -- (1,-0.6) -- (2,-0.6) -- (1.5,-0.1);
    \draw (1.5,0.6) -- (3,0.6);
    \draw[thick, blue, double=white] (1.5,0.6) -- (3,0.6);
    \draw[thick, red, double=white] (1,-0.6) -- (2,-0.6);
    \filldraw (1,-0.6) circle (2pt);
    \filldraw (1.5,-0.1) circle (2pt);
    \filldraw (1.5,0.6) circle (2pt);
    \filldraw[fill=white] (0,0.6) circle (2pt) node[left]{$#1$};
    \filldraw[fill=white] (0,-0.6) circle (2pt) node[left]{$#2$};
    \filldraw[fill=white] (3,0.6) circle (2pt) node[right]{$#3$};
    \filldraw[fill=white] (2,-0.6) circle (2pt) node[right]{$#4$};
    \end{tikzpicture}
}}

\newcommand{\PiNDiagramSeven}[4]{\raisebox{-15pt}{
    \begin{tikzpicture}[scale=0.6]
    \filldraw (-0.3,0) circle (0pt);
    \draw (1.5,0.6) -- (1,-0.6);
    \draw (0,-0.6) -- (1,-0.6) -- (2,-0.6) -- (1.5,0.6);
    \draw (1.5,0.6) -- (3,0.6);
    \draw (2,-0.6) -- (3,-0.6);
    \draw[thick, red, double=white] (0,-0.6) -- (1,-0.6);
    \draw[thick, blue, double=white] (1.5,0.6) -- (3,0.6);
    \filldraw (1,-0.6) circle (2pt);
    \filldraw (2,-0.6) circle (2pt);
    \filldraw[fill=white] (1.5,0.6) circle (2pt) node[left]{$#1$};
    \filldraw[fill=white] (0,-0.6) circle (2pt) node[left]{$#2$};
    \filldraw[fill=white] (3,0.6) circle (2pt) node[right]{$#3$};
    \filldraw[fill=white] (3,-0.6) circle (2pt) node[right]{$#4$};
    \end{tikzpicture}
}}

\newcommand{\PiNDiagramEleven}[4]{\raisebox{-15pt}{
    \begin{tikzpicture}[scale=0.6]
    \draw (1.5,0.6) -- (1.5,-0.1) -- (1,-0.6);
    \draw (0,-0.6) -- (1,-0.6) -- (2,-0.6) -- (1.5,-0.1);
    \draw (1.5,0.6) -- (3,0.6);
    \draw (2,-0.6) -- (3,-0.6);
    \draw[thick, red, double=white] (0,-0.6) -- (1,-0.6);
    \draw[thick, blue, double=white] (1.5,0.6) -- (3,0.6);
    \filldraw (1,-0.6) circle (2pt);
    \filldraw (2,-0.6) circle (2pt);
    \filldraw (1.5,-0.1) circle (2pt);
	\filldraw[fill=white] (2.5,-0.4) circle (2pt);
    \filldraw[fill=white] (1.5,0.6) circle (2pt) node[left]{$#1$};
    \filldraw[fill=white] (0,-0.6) circle (2pt) node[left]{$#2$};
    \filldraw[fill=white] (3,0.6) circle (2pt) node[right]{$#3$};
    \filldraw[fill=white] (3,-0.6) circle (2pt) node[right]{$#4$};
    \end{tikzpicture}
}}

\newcommand{\PiNDiagramTwelve}[4]{\raisebox{-15pt}{
    \begin{tikzpicture}[scale=0.6]
    \draw (1.5,0.6) -- (1.5,-0.1) -- (1,-0.6);
    \draw (0,-0.6) -- (1,-0.6) -- (2,-0.6) -- (1.5,-0.1);
    \draw (1.5,0.6) -- (3,0.6);
    \draw (2,-0.6) -- (3,-0.6);
    \draw[thick, red, double=white] (1,-0.6) -- (2,-0.6);
    \draw[thick, blue, double=white] (1.5,0.6) -- (3,0.6);
    \filldraw (1,-0.6) circle (2pt);
    \filldraw (2,-0.6) circle (2pt);
    \filldraw (1.5,-0.1) circle (2pt);
    \filldraw[fill=white] (1.5,0.6) circle (2pt) node[left]{$#1$};
    \filldraw[fill=white] (0,-0.6) circle (2pt) node[left]{$#2$};
    \filldraw[fill=white] (3,0.6) circle (2pt) node[right]{$#3$};
    \filldraw[fill=white] (3,-0.6) circle (2pt) node[right]{$#4$};
    \end{tikzpicture}
}}

\newcommand{\PiNCaseTwoJTwoPtNineVOne}[4]{\raisebox{-15pt}{
    \begin{tikzpicture}[scale=0.6]
    \draw (0,0.6) -- (1.5,0.6) -- (1,-0.6);
    \draw (0,-0.6) -- (1,-0.6) -- (2,-0.6) -- (1.5,0.6);
    \draw[thick, red, double=white] (0,0.6) -- (1.5,0.6);
    \draw[thick, blue, double=white] (0,-0.6) -- (1,-0.6);
    \draw (1.5,0.6) -- (3,0.6);
    \draw (2,-0.6) -- (3,-0.6);
    \filldraw (1,-0.6) circle (2pt);
    \filldraw (2,-0.6) circle (2pt);
    \filldraw (1.5,0.6) circle (2pt);
	\filldraw[fill=white] (2.5,-0.4) circle (2pt);
    \filldraw[fill=white] (0,0.6) circle (2pt) node[left]{$#1$};
    \filldraw[fill=white] (0,-0.6) circle (2pt) node[left]{$#2$};
    \filldraw[fill=white] (3,0.6) circle (2pt) node[right]{$#3$};
    \filldraw[fill=white] (3,-0.6) circle (2pt) node[right]{$#4$};
    \end{tikzpicture}
}}

\newcommand{\PiNDiagramFifteen}[4]{\raisebox{-15pt}{
    \begin{tikzpicture}[scale=0.6]
    \draw (0,0.6) -- (1.5,0.6) -- (1,-0.6);
    \draw (0,-0.6) -- (1,-0.6) -- (2,-0.6) -- (1.5,0.6);
    \draw (1.5,0.6) -- (3,0.6);
    \draw (2,-0.6) -- (3,-0.6);
    \draw[thick, blue, double=white] (1.5,0.6) -- (3,0.6);
    \draw[thick, red, double=white] (0,-0.6) -- (1,-0.6);
    \filldraw (1,-0.6) circle (2pt);
    \filldraw (2,-0.6) circle (2pt);
    \filldraw (1.5,0.6) circle (2pt);
	\filldraw[fill=white] (2.5,-0.4) circle (2pt);
    \filldraw[fill=white] (0,0.6) circle (2pt) node[left]{$#1$};
    \filldraw[fill=white] (0,-0.6) circle (2pt) node[left]{$#2$};
    \filldraw[fill=white] (3,0.6) circle (2pt) node[right]{$#3$};
    \filldraw[fill=white] (3,-0.6) circle (2pt) node[right]{$#4$};
    \end{tikzpicture}
}}

\newcommand{\PiNCaseTwoJTwoPtTenVOne}[4]{\raisebox{-15pt}{
    \begin{tikzpicture}[scale=0.6]
    \draw (0,0.6) -- (1.5,0.6) -- (1,-0.6);
    \draw (0,-0.6) -- (1,-0.6) -- (2,-0.6) -- (1.5,0.6);
    \draw[thick, red, double=white] (0,0.6) -- (1.5,0.6);
    \draw[thick, blue, double=white] (1,-0.6) -- (2,-0.6);
    \draw (1.5,0.6) -- (3,0.6);
    \draw (2,-0.6) -- (3,-0.6);
    \filldraw (1,-0.6) circle (2pt);
    \filldraw (2,-0.6) circle (2pt);
    \filldraw (1.5,0.6) circle (2pt);
	\filldraw[fill=white] (2.5,-0.4) circle (2pt);
    \filldraw[fill=white] (0,0.6) circle (2pt) node[left]{$#1$};
    \filldraw[fill=white] (0,-0.6) circle (2pt) node[left]{$#2$};
    \filldraw[fill=white] (3,0.6) circle (2pt) node[right]{$#3$};
    \filldraw[fill=white] (3,-0.6) circle (2pt) node[right]{$#4$};
    \end{tikzpicture}
}}

\newcommand{\PiNCaseTwoJTwoPtTenVOneSwap}[4]{\raisebox{-15pt}{
    \begin{tikzpicture}[scale=0.6]
    \draw (0,0.6) -- (1.5,0.6) -- (1,-0.6);
    \draw (0,-0.6) -- (1,-0.6) -- (2,-0.6) -- (1.5,0.6);
    \draw[thick, blue, double=white] (0,0.6) -- (1.5,0.6);
    \draw[thick, red, double=white] (1,-0.6) -- (2,-0.6);
    \draw (1.5,0.6) -- (3,0.6);
    \draw (2,-0.6) -- (3,-0.6);
    \filldraw (1,-0.6) circle (2pt);
    \filldraw (2,-0.6) circle (2pt);
    \filldraw (1.5,0.6) circle (2pt);
	\filldraw[fill=white] (2.5,-0.4) circle (2pt);
    \filldraw[fill=white] (0,0.6) circle (2pt) node[left]{$#1$};
    \filldraw[fill=white] (0,-0.6) circle (2pt) node[left]{$#2$};
    \filldraw[fill=white] (3,0.6) circle (2pt) node[right]{$#3$};
    \filldraw[fill=white] (3,-0.6) circle (2pt) node[right]{$#4$};
    \end{tikzpicture}
}}

\newcommand{\PiNDiagramSixteen}[4]{\raisebox{-15pt}{
    \begin{tikzpicture}[scale=0.6]
    \draw (0,0.6) -- (1.5,0.6) -- (1,-0.6);
    \draw (0,-0.6) -- (1,-0.6) -- (2,-0.6) -- (1.5,0.6);
    \draw (1.5,0.6) -- (3,0.6);
    \draw (2,-0.6) -- (3,-0.6);
    \draw[thick, blue, double=white] (1.5,0.6) -- (3,0.6);
    \draw[thick, red, double=white] (1,-0.6) -- (2,-0.6);
    \filldraw (1,-0.6) circle (2pt);
    \filldraw (2,-0.6) circle (2pt);
    \filldraw (1.5,0.6) circle (2pt);
	\filldraw[fill=white] (0.75,0.8) circle (2pt);
    \filldraw[fill=white] (0,0.6) circle (2pt) node[left]{$#1$};
    \filldraw[fill=white] (0,-0.6) circle (2pt) node[left]{$#2$};
    \filldraw[fill=white] (3,0.6) circle (2pt) node[right]{$#3$};
    \filldraw[fill=white] (3,-0.6) circle (2pt) node[right]{$#4$};
    \end{tikzpicture}
}}

\newcommand{\PiNDiagramSixteenFlipped}[4]{\raisebox{-15pt}{
    \begin{tikzpicture}[scale=0.6]
    \draw (0,0.6) -- (1.5,0.6) -- (1,-0.6);
    \draw (0,-0.6) -- (1,-0.6) -- (2,-0.6) -- (1.5,0.6);
    \draw (1.5,0.6) -- (3,0.6);
    \draw (2,-0.6) -- (3,-0.6);
    \draw[thick, blue, double=white] (1.5,0.6) -- (3,0.6);
    \draw[thick, red, double=white] (1,-0.6) -- (2,-0.6);
    \filldraw (1,-0.6) circle (2pt);
    \filldraw (2,-0.6) circle (2pt);
    \filldraw (1.5,0.6) circle (2pt);
	\filldraw[fill=white] (2.5,-0.4) circle (2pt);
    \filldraw[fill=white] (0,0.6) circle (2pt) node[left]{$#1$};
    \filldraw[fill=white] (0,-0.6) circle (2pt) node[left]{$#2$};
    \filldraw[fill=white] (3,0.6) circle (2pt) node[right]{$#3$};
    \filldraw[fill=white] (3,-0.6) circle (2pt) node[right]{$#4$};
    \end{tikzpicture}
}}

\newcommand{\PiNDiagramThree}[4]{\raisebox{-15pt}{
    \begin{tikzpicture}[scale=0.6]
    \draw (0,0.6) -- (1.5,0.6) -- (1,-0.6);
    \draw (0,-0.6) -- (1,-0.6) -- (2,-0.6) -- (1.5,0.6);
    \draw (1.5,0.6) -- (3,0.6);
    \draw (2,-0.6) -- (3,-0.6);
    \draw[thick, red, double=white] (0,0.6) -- (1.5,0.6);
    \draw[thick, blue, double=white] (2,-0.6) -- (3,-0.6);
    \filldraw (1,-0.6) circle (2pt);
    \filldraw (2,-0.6) circle (2pt);
    \filldraw (1.5,0.6) circle (2pt);
    \filldraw[fill=white] (0,0.6) circle (2pt) node[left]{$#1$};
    \filldraw[fill=white] (0,-0.6) circle (2pt) node[left]{$#2$};
    \filldraw[fill=white] (3,0.6) circle (2pt) node[right]{$#3$};
    \filldraw[fill=white] (3,-0.6) circle (2pt) node[right]{$#4$};
    \end{tikzpicture}
}}

\newcommand{\PiNCaseTwoJTwoPtElevenVOne}[4]{\raisebox{-15pt}{
    \begin{tikzpicture}[scale=0.6]
    \draw (0,0.6) -- (1.5,0.6) -- (1.5,-0.1) -- (1,-0.6);
    \draw (0,-0.6) -- (1,-0.6) -- (2,-0.6) -- (1.5,-0.1);
    \draw[thick, red, double=white] (0,0.6) -- (1.5,0.6);
    \draw[thick, blue, double=white] (0,-0.6) -- (1,-0.6);
    \draw (1.5,0.6) -- (3,0.6);
    \draw (2,-0.6) -- (3,-0.6);
    \filldraw (1,-0.6) circle (2pt);
    \filldraw (2,-0.6) circle (2pt);
    \filldraw (1.5,-0.1) circle (2pt);
    \filldraw (1.5,0.6) circle (2pt);
	\filldraw[fill=white] (2.5,-0.4) circle (2pt);
    \filldraw[fill=white] (0,0.6) circle (2pt) node[left]{$#1$};
    \filldraw[fill=white] (0,-0.6) circle (2pt) node[left]{$#2$};
    \filldraw[fill=white] (3,0.6) circle (2pt) node[right]{$#3$};
    \filldraw[fill=white] (3,-0.6) circle (2pt) node[right]{$#4$};
    \end{tikzpicture}
}}

\newcommand{\PiNDiagramNineteen}[4]{\raisebox{-15pt}{
    \begin{tikzpicture}[scale=0.6]
    \draw (0,0.6) -- (1.5,0.6) -- (1.5,-0.1) -- (1,-0.6);
    \draw (0,-0.6) -- (1,-0.6) -- (2,-0.6) -- (1.5,-0.1);
    \draw (1.5,0.6) -- (3,0.6);
    \draw (2,-0.6) -- (3,-0.6);
    \draw[thick, blue, double=white] (1.5,0.6) -- (3,0.6);
    \draw[thick, red, double=white] (0,-0.6) -- (1,-0.6);
    \filldraw (1,-0.6) circle (2pt);
    \filldraw (2,-0.6) circle (2pt);
    \filldraw (1.5,-0.1) circle (2pt);
    \filldraw (1.5,0.6) circle (2pt);
	\filldraw[fill=white] (2.5,-0.4) circle (2pt);
    \filldraw[fill=white] (0,0.6) circle (2pt) node[left]{$#1$};
    \filldraw[fill=white] (0,-0.6) circle (2pt) node[left]{$#2$};
    \filldraw[fill=white] (3,0.6) circle (2pt) node[right]{$#3$};
    \filldraw[fill=white] (3,-0.6) circle (2pt) node[right]{$#4$};
    \end{tikzpicture}
}}

\newcommand{\PiNDiagramFour}[4]{\raisebox{-15pt}{
    \begin{tikzpicture}[scale=0.6]
    \draw (0,0.6) -- (1.5,0.6) -- (1.5,-0.1) -- (1,-0.6);
    \draw (0,-0.6) -- (1,-0.6) -- (2,-0.6) -- (1.5,-0.1);
    \draw (1.5,0.6) -- (3,0.6);
    \draw (2,-0.6) -- (3,-0.6);
    \draw[thick, red, double=white] (0,0.6) -- (1.5,0.6);
    \draw[thick, blue, double=white] (2,-0.6) -- (3,-0.6);
    \filldraw (1,-0.6) circle (2pt);
    \filldraw (2,-0.6) circle (2pt);
    \filldraw (1.5,-0.1) circle (2pt);
    \filldraw (1.5,0.6) circle (2pt);
    \filldraw[fill=white] (0,0.6) circle (2pt) node[left]{$#1$};
    \filldraw[fill=white] (0,-0.6) circle (2pt) node[left]{$#2$};
    \filldraw[fill=white] (3,0.6) circle (2pt) node[right]{$#3$};
    \filldraw[fill=white] (3,-0.6) circle (2pt) node[right]{$#4$};
    \end{tikzpicture}
}}

\newcommand{\PiNCaseTwoJTwoPtTwelveVOne}[4]{\raisebox{-15pt}{
    \begin{tikzpicture}[scale=0.6]
    \draw (0,0.6) -- (1.5,0.6) -- (1.5,-0.1) -- (1,-0.6);
    \draw (0,-0.6) -- (1,-0.6) -- (2,-0.6) -- (1.5,-0.1);
    \draw[thick, red, double=white] (0,0.6) -- (1.5,0.6);
    \draw[thick, blue, double=white] (1,-0.6) -- (2,-0.6);
    \draw (1.5,0.6) -- (3,0.6);
    \draw (2,-0.6) -- (3,-0.6);
    \filldraw (1,-0.6) circle (2pt);
    \filldraw (2,-0.6) circle (2pt);
    \filldraw (1.5,-0.1) circle (2pt);
    \filldraw (1.5,0.6) circle (2pt);
	\filldraw[fill=white] (2.5,-0.4) circle (2pt);
    \filldraw[fill=white] (0,0.6) circle (2pt) node[left]{$#1$};
    \filldraw[fill=white] (0,-0.6) circle (2pt) node[left]{$#2$};
    \filldraw[fill=white] (3,0.6) circle (2pt) node[right]{$#3$};
    \filldraw[fill=white] (3,-0.6) circle (2pt) node[right]{$#4$};
    \end{tikzpicture}
}}

\newcommand{\PiNCaseTwoJTwoPtTwelveVOneSwap}[4]{\raisebox{-15pt}{
    \begin{tikzpicture}[scale=0.6]
    \draw (0,0.6) -- (1.5,0.6) -- (1.5,-0.1) -- (1,-0.6);
    \draw (0,-0.6) -- (1,-0.6) -- (2,-0.6) -- (1.5,-0.1);
    \draw[thick, blue, double=white] (0,0.6) -- (1.5,0.6);
    \draw[thick, red, double=white] (1,-0.6) -- (2,-0.6);
    \draw (1.5,0.6) -- (3,0.6);
    \draw (2,-0.6) -- (3,-0.6);
    \filldraw (1,-0.6) circle (2pt);
    \filldraw (2,-0.6) circle (2pt);
    \filldraw (1.5,-0.1) circle (2pt);
    \filldraw (1.5,0.6) circle (2pt);
    \filldraw[fill=white] (0,0.6) circle (2pt) node[left]{$#1$};
    \filldraw[fill=white] (0,-0.6) circle (2pt) node[left]{$#2$};
    \filldraw[fill=white] (3,0.6) circle (2pt) node[right]{$#3$};
    \filldraw[fill=white] (3,-0.6) circle (2pt) node[right]{$#4$};
    \end{tikzpicture}
}}

\newcommand{\PiNDiagramTwenty}[4]{\raisebox{-15pt}{
    \begin{tikzpicture}[scale=0.6]
    \draw (0,0.6) -- (1.5,0.6) -- (1.5,-0.1) -- (1,-0.6);
    \draw (0,-0.6) -- (1,-0.6) -- (2,-0.6) -- (1.5,-0.1);
    \draw (1.5,0.6) -- (3,0.6);
    \draw (2,-0.6) -- (3,-0.6);
    \draw[thick, blue, double=white] (1.5,0.6) -- (3,0.6);
    \draw[thick, red, double=white] (1,-0.6) -- (2,-0.6);
    \filldraw (1,-0.6) circle (2pt);
    \filldraw (2,-0.6) circle (2pt);
    \filldraw (1.5,-0.1) circle (2pt);
    \filldraw (1.5,0.6) circle (2pt);
    \filldraw[fill=white] (0,0.6) circle (2pt) node[left]{$#1$};
    \filldraw[fill=white] (0,-0.6) circle (2pt) node[left]{$#2$};
    \filldraw[fill=white] (3,0.6) circle (2pt) node[right]{$#3$};
    \filldraw[fill=white] (3,-0.6) circle (2pt) node[right]{$#4$};
    \end{tikzpicture}
}}

\newcommand{\PiNDiagramTwentyOne}[4]{\raisebox{-15pt}{
    \begin{tikzpicture}[scale=0.6]
    \draw[] (0,0.6) -- (1,0.6) -- (1,-0.6) -- (0,-0.6);
    \draw (1,0.6) -- (2,0.6);
    \draw (1,-0.6) -- (2,-0.6);
    \draw[thick, red, double=white] (0,0.6) -- (1,0.6);
    \draw[thick, blue, double=white] (0,-0.6) -- (1,-0.6);
    \filldraw (1,0.6) circle (2pt);
    \filldraw (1,-0.6) circle (2pt);
	\filldraw[fill=white] (1.5,-0.4) circle (2pt);
    \filldraw[fill=white] (0,0.6) circle (2pt) node[left]{$#1$};
    \filldraw[fill=white] (0,-0.6) circle (2pt) node[left]{$#2$};
    \filldraw[fill=white] (2,0.6) circle (2pt) node[right]{$#3$};
    \filldraw[fill=white] (2,-0.6) circle (2pt) node[right]{$#4$};
    \end{tikzpicture}
}}

\newcommand{\PiNDiagramTwentyOneFlipped}[4]{\raisebox{-15pt}{
    \begin{tikzpicture}[scale=0.6]
    \draw[] (0,0.6) -- (1,0.6) -- (1,-0.6) -- (0,-0.6);
    \draw (1,0.6) -- (2,0.6);
    \draw (1,-0.6) -- (2,-0.6);
    \draw[thick, red, double=white] (0,0.6) -- (1,0.6);
    \draw[thick, blue, double=white] (0,-0.6) -- (1,-0.6);
    \filldraw (1,0.6) circle (2pt);
    \filldraw (1,-0.6) circle (2pt);
	\filldraw[fill=white] (1.5,0.8) circle (2pt);
    \filldraw[fill=white] (0,0.6) circle (2pt) node[left]{$#1$};
    \filldraw[fill=white] (0,-0.6) circle (2pt) node[left]{$#2$};
    \filldraw[fill=white] (2,0.6) circle (2pt) node[right]{$#3$};
    \filldraw[fill=white] (2,-0.6) circle (2pt) node[right]{$#4$};
    \end{tikzpicture}
}}

\newcommand{\PiNDiagramTwentyTwo}[4]{\raisebox{-15pt}{
    \begin{tikzpicture}[scale=0.6]
    \draw[] (0,0.6) -- (1,0.6) -- (1,-0.6) -- (0,-0.6);
    \draw (1,0.6) -- (2,0.6);
    \draw (1,-0.6) -- (2,-0.6);
    \draw[thick, red, double=white] (0,0.6) -- (1,0.6);
    \draw[thick, blue, double=white] (1,-0.6) -- (2,-0.6);
    \filldraw (1,0.6) circle (2pt);
    \filldraw (1,-0.6) circle (2pt);
    \filldraw[fill=white] (0,0.6) circle (2pt) node[left]{$#1$};
    \filldraw[fill=white] (0,-0.6) circle (2pt) node[left]{$#2$};
    \filldraw[fill=white] (2,0.6) circle (2pt) node[right]{$#3$};
    \filldraw[fill=white] (2,-0.6) circle (2pt) node[right]{$#4$};
    \end{tikzpicture}
}}

\newcommand{\PiNDiagramTwentyTwoFlipped}[4]{\raisebox{-15pt}{
    \begin{tikzpicture}[scale=0.6]
    \draw[] (0,0.6) -- (1,0.6) -- (1,-0.6) -- (0,-0.6);
    \draw (1,0.6) -- (2,0.6);
    \draw (1,-0.6) -- (2,-0.6);
    \draw[thick, red, double=white] (0,0.6) -- (1,0.6);
    \draw[thick, blue, double=white] (1,-0.6) -- (2,-0.6);
    \filldraw (1,0.6) circle (2pt);
    \filldraw (1,-0.6) circle (2pt);
	\filldraw[fill=white] (1.5,0.8) circle (2pt);
    \filldraw[fill=white] (0,0.6) circle (2pt) node[left]{$#1$};
    \filldraw[fill=white] (0,-0.6) circle (2pt) node[left]{$#2$};
    \filldraw[fill=white] (2,0.6) circle (2pt) node[right]{$#3$};
    \filldraw[fill=white] (2,-0.6) circle (2pt) node[right]{$#4$};
    \end{tikzpicture}
}}

\newcommand{\PiNDiagramTwentySix}[4]{\raisebox{-15pt}{
    \begin{tikzpicture}[scale=0.6]
    \draw[] (1,0.6) -- (1,-0.6) -- (0,-0.6);
    \draw (1,0.6) -- (2,0.6);
    \draw (1,-0.6) -- (2,-0.6);
    \draw[thick, red, double=white] (0,-0.6) -- (1,-0.6);
    \draw[thick, blue, double=white] (1,0.6) -- (2,0.6);
    \filldraw (1,-0.6) circle (2pt);
    \filldraw[fill=white] (1,0.6) circle (2pt) node[left]{$#1$};
    \filldraw[fill=white] (0,-0.6) circle (2pt) node[left]{$#2$};
    \filldraw[fill=white] (2,0.6) circle (2pt) node[right]{$#3$};
    \filldraw[fill=white] (2,-0.6) circle (2pt) node[right]{$#4$};
    \end{tikzpicture}
}}

\newcommand{\PiNDiagramTwentySeven}[4]{\raisebox{-15pt}{
    \begin{tikzpicture}[scale=0.6]
    \draw[] (1,0.6) -- (1,-0.6) -- (0,-0.6);
    \draw (1,0.6) -- (2,0.6);
    \draw[thick, red, double=white] (0,-0.6) -- (1,-0.6);
    \draw[thick, blue, double=white] (1,0.6) -- (2,0.6);
    \filldraw[fill=white] (1,0.6) circle (2pt) node[left]{$#1$};
    \filldraw[fill=white] (0,-0.6) circle (2pt) node[left]{$#2$};
    \filldraw[fill=white] (2,0.6) circle (2pt) node[right]{$#3$};
    \filldraw[fill=white] (1,-0.6) circle (2pt) node[right]{$#4$};
    \end{tikzpicture}
}}

\newcommand{\PiNCaseThreeJOnePtThreeSwap}[3]{\raisebox{-15pt}{
    \begin{tikzpicture}[scale=0.6]
    \draw[] (1,0.6) -- (1,-0.6) -- (0,-0.6);
    \draw[] (1,0.6) -- (2,0) -- (1,-0.6);
    \draw[thick, blue, double=white] (1,0.6) -- (2,0);
    \draw[thick, red, double=white] (0,-0.6) -- (1,-0.6);
    \filldraw[fill=white] (1,0.6) circle (2pt) node[left]{$#1$};
    \filldraw[fill=white] (0,-0.6) circle (2pt) node[left]{$#2$};
    \filldraw (1,-0.6) circle (2pt);
    \filldraw[fill=white] (2,0) circle (2pt) node[right]{$#3$};
    \end{tikzpicture}
}}

\newcommand{\PiNCaseThreeJOnePtFive}[3]{\raisebox{-15pt}{
    \begin{tikzpicture}[scale=0.6]
    \draw[] (0,0.6) -- (1,0.6) -- (1,-0.6) -- (0,-0.6);
    \draw[] (1,0.6) -- (2,0) -- (1,-0.6);
    \draw[thick, red, double=white] (0,0.6) -- (1,0.6);
    \draw[thick, blue, double=white] (0,-0.6) -- (1,-0.6);
    \filldraw[fill=white] (0,0.6) circle (2pt) node[left]{$#1$};
    \filldraw[fill=white] (0,-0.6) circle (2pt) node[left]{$#2$};
    \filldraw (1,0.6) circle (2pt);
    \filldraw (1,-0.6) circle (2pt);
    \filldraw[fill=white] (2,0) circle (2pt) node[right]{$#3$};
    \end{tikzpicture}
}}

\newcommand{\PiNCaseThreeJOnePtSix}[3]{\raisebox{-15pt}{
    \begin{tikzpicture}[scale=0.6]
    \draw[] (0,0.6) -- (1,0.6) -- (1,-0.6) -- (0,-0.6);
    \draw[] (1,0.6) -- (2,0) -- (1,-0.6);
    \draw[thick, red, double=white] (1,0.6) -- (2,0);
    \draw[thick, blue, double=white] (0,-0.6) -- (1,-0.6);
	\filldraw[fill=white] (0.5,0.8) circle (2pt);
    \filldraw[fill=white] (0,0.6) circle (2pt) node[left]{$#1$};
    \filldraw[fill=white] (0,-0.6) circle (2pt) node[left]{$#2$};
    \filldraw (1,0.6) circle (2pt);
    \filldraw (1,-0.6) circle (2pt);
    \filldraw[fill=white] (2,0) circle (2pt) node[right]{$#3$};
    \end{tikzpicture}
}}

\newcommand{\PiNCaseThreeJOnePtSixFlipped}[3]{\raisebox{-15pt}{
    \begin{tikzpicture}[scale=0.6]
    \draw[] (0,0.6) -- (1,0.6) -- (1,-0.6) -- (0,-0.6);
    \draw[] (1,0.6) -- (2,0) -- (1,-0.6);
    \draw[thick, red, double=white] (1,0.6) -- (2,0);
    \draw[thick, blue, double=white] (0,-0.6) -- (1,-0.6);
    \filldraw[fill=white] (0,0.6) circle (2pt) node[left]{$#1$};
    \filldraw[fill=white] (0,-0.6) circle (2pt) node[left]{$#2$};
    \filldraw (1,0.6) circle (2pt);
    \filldraw (1,-0.6) circle (2pt);
    \filldraw[fill=white] (2,0) circle (2pt) node[right]{$#3$};
    \end{tikzpicture}
}}

\newcommand{\PiNCaseThreeJOnePtSeven}[3]{\raisebox{-15pt}{
    \begin{tikzpicture}[scale=0.6]
    \draw[] (0,0.6) -- (1,0.6) -- (1,-0.6) -- (0,-0.6);
    \draw[] (1,0.6) -- (2,0) -- (1,-0.6);
    \draw[thick, red, double=white] (0,0.6) -- (1,0.6);
    \draw[thick, blue, double=white] (1,-0.6) -- (2,0);
    \filldraw[fill=white] (0,0.6) circle (2pt) node[left]{$#1$};
    \filldraw[fill=white] (0,-0.6) circle (2pt) node[left]{$#2$};
    \filldraw (1,0.6) circle (2pt);
    \filldraw (1,-0.6) circle (2pt);
    \filldraw[fill=white] (2,0) circle (2pt) node[right]{$#3$};
    \end{tikzpicture}
}}

\newcommand{\PiNCaseThreeJOnePtEight}[3]{\raisebox{-15pt}{
    \begin{tikzpicture}[scale=0.6]
    \draw[] (0,0.6) -- (1,0.6) -- (1,-0.6) -- (0,-0.6);
    \draw[] (1,0.6) -- (2,0) -- (1,-0.6);
    \draw[thick, red, double=white] (1,0.6) -- (2,0);
    \draw[thick, blue, double=white] (1,-0.6) -- (2,0);
	\filldraw[fill=white] (0.5,0.8) circle (2pt);
    \filldraw[fill=white] (0,0.6) circle (2pt) node[left]{$#1$};
    \filldraw[fill=white] (0,-0.6) circle (2pt) node[left]{$#2$};
    \filldraw (1,0.6) circle (2pt);
    \filldraw (1,-0.6) circle (2pt);
    \filldraw[fill=white] (2,0) circle (2pt) node[right]{$#3$};
    \end{tikzpicture}
}}

\newcommand{\PiNCaseThreeJOnePtEightFlipped}[3]{\raisebox{-15pt}{
    \begin{tikzpicture}[scale=0.6]
    \draw[] (0,0.6) -- (1,0.6) -- (1,-0.6) -- (0,-0.6);
    \draw[] (1,0.6) -- (2,0) -- (1,-0.6);
    \draw[thick, red, double=white] (1,0.6) -- (2,0);
    \draw[thick, blue, double=white] (1,-0.6) -- (2,0);
	\filldraw[fill=white] (0.5,-0.4) circle (2pt);
    \filldraw[fill=white] (0,0.6) circle (2pt) node[left]{$#1$};
    \filldraw[fill=white] (0,-0.6) circle (2pt) node[left]{$#2$};
    \filldraw (1,0.6) circle (2pt);
    \filldraw (1,-0.6) circle (2pt);
    \filldraw[fill=white] (2,0) circle (2pt) node[right]{$#3$};
    \end{tikzpicture}
}}

\newcommand{\PiNCaseTwoGeneralPartOne}[3]{\raisebox{-15pt}{
    \begin{tikzpicture}[scale=0.6]
    \draw[] (0,0) -- (1,0.6) -- (1,-0.6) -- (0,0);
    \draw[thick, red, double=white] (0,0) -- (1,0.6);
    \draw[] (1,0.6) -- (2,0.6);
    \draw[] (1,-0.6) -- (2,-0.6);
    \draw[thick, blue, double=white] (1,-0.6) -- (2,-0.6);
    \filldraw (1,-0.6) circle (2pt);
    \filldraw (1,0.6) circle (2pt);
    \filldraw[fill=white] (0,0) circle (2pt) node[left]{$#1$};
    \filldraw[fill=white] (2,0.6) circle (2pt) node[right]{$#2$};
    \filldraw[fill=white] (2,-0.6) circle (2pt) node[right]{$#3$};
    \end{tikzpicture}
}}

\newcommand{\PiNCaseTwoMidPartThree}[4]{\raisebox{-15pt}{
    \begin{tikzpicture}[scale=0.6]
    \draw[] (0,0.6) -- (1,0.6) -- (1,-0.6) -- (0,-0.6);
    \draw[] (1,0.6) -- (2,0.6) -- (2,-0.6) -- (1,-0.6);
    \draw[] (2,0.6) -- (3,0.6);
    \draw[] (2,-0.6) -- (3,-0.6);
    \draw[thick, red, double=white] (0,0.6) -- (1,0.6);
    \draw[thick, blue, double=white] (2,-0.6) -- (3,-0.6);
    \draw[] (2,0.6) -- (3,0.6);
    \filldraw (2,0.6) circle (2pt);
    \filldraw (2,-0.6) circle (2pt);
    \filldraw[fill=white] (0,0.6) circle (2pt) node[left]{$#1$};
    \filldraw[fill=white] (0,-0.6) circle (2pt) node[left]{$#2$};
    \filldraw (1,0.6) circle (2pt);
    \filldraw (1,-0.6) circle (2pt);
    \filldraw[fill=white] (3,0.6) circle (2pt) node[right]{$#3$};
    \filldraw[fill=white] (3,-0.6) circle (2pt) node[right]{$#4$};
    \end{tikzpicture}
}}

\newcommand{\PiNCaseTwoMidPartFour}[4]{\raisebox{-15pt}{
    \begin{tikzpicture}[scale=0.6]
    \draw[] (0,0.6) -- (1,0.6) -- (1,-0.6) -- (0,-0.6);
    \draw[] (1,0.6) -- (2,0.6) -- (2,-0.6) -- (1,-0.6);
    \draw[] (2,0.6) -- (3,0.6);
    \draw[] (2,-0.6) -- (3,-0.6);
    \draw[thick, red, double=white] (1,0.6) -- (2,0.6);
    \draw[thick, blue, double=white] (2,-0.6) -- (3,-0.6);
    \draw[] (2,0.6) -- (3,0.6);
    \filldraw (2,0.6) circle (2pt);
    \filldraw (2,-0.6) circle (2pt);
	\filldraw[fill=white] (0.5,0.8) circle (2pt);
    \filldraw[fill=white] (0,0.6) circle (2pt) node[left]{$#1$};
    \filldraw[fill=white] (0,-0.6) circle (2pt) node[left]{$#2$};
    \filldraw (1,0.6) circle (2pt);
    \filldraw (1,-0.6) circle (2pt);
    \filldraw[fill=white] (3,0.6) circle (2pt) node[right]{$#3$};
    \filldraw[fill=white] (3,-0.6) circle (2pt) node[right]{$#4$};
    \end{tikzpicture}
}}

\newcommand{\PiNCaseOnePhi}[2]{\raisebox{-12pt}{
    \begin{tikzpicture}[scale=0.8]
    \fill[fill=gray!50] (1,0.6) to [out=260,in=100] (1,-0.6) to [out=170,in=10] (0,-0.6) to [out=140,in=340] (-1,0) to [out=20,in=220] (0,0.6) to [out=350,in=190] (1,0.6);
    \fill[fill=gray!50] (2,0.6) to [out=280,in=80] (2,-0.6) to [out=10,in=170] (4,-0.6) to [out=40,in=200] (5,0) to [out=160,in=320] (4,0.6) to [out=190,in=350] (2,0.6);
    \draw[] (1,0.6) -- (2,0.6);
    \draw[] (1,-0.6) -- (2,-0.6);
    \draw[thick, red, double=white] (-1,0) to [out=20,in=220] (0,0.6) to [out=350,in=190] (1,0.6);
    \draw (1.5,-0.6) node[above]{$\circ$};
    \filldraw (2,-0.6) circle (2pt);
    \filldraw (2,0.6) circle (2pt);
    \node at (0.5,0) {$\psi_0$};
    \node at (3,0) {$\bar\psi^{n-1}\bar\psi_0$};
    \filldraw[fill=white] (-1,0) circle (2pt) node[left]{$#1$};
    \filldraw (1,0.6) circle (2pt);
    \filldraw (1,-0.6) circle (2pt);
    \filldraw[fill=white] (5,0) circle (2pt) node[right]{$#2$};
    \end{tikzpicture}
}}

\newcommand{\PiNCaseOnePhiPtOne}[3]{\raisebox{-17.5pt}{
    \begin{tikzpicture}[scale=0.8]
    \fill[fill=gray!50] (1,0.6) to [out=260,in=100] (1,-0.6) to [out=170,in=10] (0,-0.6) to [out=140,in=340] (-1,0) to [out=20,in=220] (0,0.6) to [out=350,in=190] (1,0.6);
    \draw[] (1,0.6) -- (2,0.6);
    \draw[] (1,-0.6) -- (2,-0.6);
    \draw[thick, red, double=white] (-1,0) to [out=20,in=220] (0,0.6) to [out=350,in=190] (1,0.6);
    \node at (0.5,0) {$\psi_0$};
    \draw (1.5,-0.6) node[above]{$\circ$};
    \filldraw (1,-0.6) circle (2pt);
    \filldraw (1,0.6) circle (2pt);
    \filldraw[fill=white] (-1,0) circle (2pt) node[left]{$#1$};
    \filldraw[fill=white] (2,0.6) circle (2pt) node[right]{$#2$};
    \filldraw[fill=white] (2,-0.6) circle (2pt) node[right]{$#3$};
    \end{tikzpicture}
}}

\newcommand{\PiNCaseTwoPhi}{\raisebox{-12pt}{
    \begin{tikzpicture}[scale=0.8]
    \fill[fill=gray!50] (2,0.6) to [out=260,in=100] (2,-0.6) to [out=170,in=10] (0,-0.6) to [out=140,in=340] (-1,0) to [out=20,in=220] (0,0.6) to [out=350,in=190] (2,0.6);
    \fill[fill=gray!50] (2,0.6) to [out=280,in=80] (2,-0.6) to [out=10,in=170] (3,-0.6) to [out=100,in=260] (3,0.6) to [out=190,in=350] (2,0.6);
    \fill[fill=gray!50] (4,0.6) to [out=280,in=80] (4,-0.6) to [out=10,in=170] (6,-0.6) to [out=40,in=200] (7,0) to [out=160,in=320] (6,0.6) to [out=190,in=350] (4,0.6);
    \draw[thick, red, double=white] (2,0.6) to [out=350,in=190] (3,0.6);
    \draw (3,0.6) -- (4,0.6);
    \draw (3,-0.6) -- (4,-0.6);
    \filldraw (4,0.6) circle (2pt);
    \filldraw (4,-0.6) circle (2pt);
    \draw (3.5,-0.6) node[above]{$\circ$};
    \node at (1,0) {$\psi^{j-1}\psi_0$};
    \node at (5.25,0) {$\bar\psi^{n-j-1}\bar\psi_0$};
    \node at (2.5,0) {$\psi$};
    \filldraw[fill=white] (-1,0) circle (2pt) node[left]{$0$};
    \filldraw (2,0.6) circle (2pt);
    \filldraw (2,-0.6) circle (2pt);
    \filldraw (3,0.6) circle (2pt);
    \filldraw (3,-0.6) circle (2pt);
    \filldraw[fill=white] (7,0) circle (2pt) node[right]{$x$};
    \end{tikzpicture}
}}

\newcommand{\PiNCaseTwoPhiPtTwo}[3]{\raisebox{-17.5pt}{
    \begin{tikzpicture}[scale=0.8]
    \fill[fill=gray!50] (2,0.6) to [out=280,in=80] (2,-0.6) to [out=10,in=170] (3,-0.6) to [out=100,in=260] (3,0.6) to [out=190,in=350] (2,0.6);
    \fill[fill=gray!50] (4,0.6) to [out=280,in=80] (4,-0.6) to [out=10,in=170] (6,-0.6) to [out=40,in=200] (7,0) to [out=160,in=320] (6,0.6) to [out=190,in=350] (4,0.6);
    \draw[thick, red, double=white] (2,0.6) to [out=350,in=190] (3,0.6);
    \draw (3,0.6) -- (4,0.6);
    \draw (3,-0.6) -- (4,-0.6);
    \filldraw (4,0.6) circle (2pt);
    \filldraw (4,-0.6) circle (2pt);
    \draw (3.5,-0.6) node[above]{$\circ$};
    \node at (5.25,0) {$\bar\psi^{n-j-1}\bar\psi_0$};
    \node at (2.5,0) {$\psi$};
    \filldraw[fill=white] (2,0.6) circle (2pt) node[left]{$#1$};
    \filldraw[fill=white] (2,-0.6) circle (2pt) node[left]{$#2$};
    \filldraw (3,0.6) circle (2pt);
    \filldraw (3,-0.6) circle (2pt);
    \filldraw[fill=white] (7,0) circle (2pt) node[right]{$#3$};
    \end{tikzpicture}
}}

\newcommand{\PiNCaseTwoPhiPtThree}[4]{\raisebox{-17.5pt}{
    \begin{tikzpicture}[scale=0.8]
    \fill[fill=gray!50] (2,0.6) to [out=280,in=80] (2,-0.6) to [out=10,in=170] (3,-0.6) to [out=100,in=260] (3,0.6) to [out=190,in=350] (2,0.6);
    \draw[thick, red, double=white] (2,0.6) to [out=350,in=190] (3,0.6);
    \node at (2.5,0) {$\psi$};
    \draw (3,0.6) -- (4,0.6);
    \draw (3,-0.6) -- (4,-0.6);
    \filldraw (3,0.6) circle (2pt);
    \filldraw (3,-0.6) circle (2pt);
    \draw (3.5,-0.6) node[above]{$\circ$};
    \filldraw[fill=white] (2,0.6) circle (2pt) node[left]{$#1$};
    \filldraw[fill=white] (2,-0.6) circle (2pt) node[left]{$#2$};
    \filldraw[fill=white] (4,0.6) circle (2pt) node[right]{$#3$};
    \filldraw[fill=white] (4,-0.6) circle (2pt) node[right]{$#4$};
    \end{tikzpicture}
}}

\newcommand{\PiNCaseTwoPhiPtFour}[3]{\raisebox{-17.5pt}{
    \begin{tikzpicture}[scale=0.8]
    \fill[fill=gray!50] (3,0.6) to [out=280,in=80] (3,-0.6) to [out=10,in=170] (5,-0.6) to [out=40,in=200] (6,0) to [out=160,in=320] (5,0.6) to [out=190,in=350] (3,0.6);
    \node at (4.25,0) {$\bar\psi^{n-j-1}\bar\psi_0$};
    \filldraw[fill=white] (3,0.6) circle (2pt) node[left]{$#1$};
    \filldraw[fill=white] (3,-0.6) circle (2pt) node[left]{$#2$};
    \filldraw[fill=white] (6,0) circle (2pt) node[right]{$#3$};
    \end{tikzpicture}
}}

\newcommand{\PiNCaseThreePhi}{\raisebox{-12pt}{
    \begin{tikzpicture}[scale=0.8]
    \fill[fill=gray!50] (2,0.6) to [out=260,in=100] (2,-0.6) to [out=170,in=10] (0,-0.6) to [out=140,in=340] (-1,0) to [out=20,in=220] (0,0.6) to [out=350,in=190] (2,0.6);
    \fill[fill=gray!50] (2,0.6) to [out=350,in=190] (3,0.6) to [out=320,in=160] (4,0) to [out=200,in=40] (3,-0.6) to [out=170,in=10] (2,-0.6) to [out=80,in=280] (2,0.6);
    \draw[thick, red, double=white] (2,0.6) to [out=350,in=190] (3,0.6) to [out=320,in=160] (4,0);
    \node at (1,0) {$\psi^{n-1}\psi_0$};
    \node at (2.5,0) {$\psi_n$};
    \filldraw[fill=white] (-1,0) circle (2pt) node[left]{$0$};
    \filldraw (2,0.6) circle (2pt);
    \filldraw (2,-0.6) circle (2pt);
    \filldraw[fill=white] (4,0) circle (2pt) node[right]{$x$};
    \end{tikzpicture}
}}

\newcommand{\PiNCaseThreePhiPartTwo}[3]{\raisebox{-17.5pt}{
    \begin{tikzpicture}[scale=0.8]
    \fill[fill=gray!50] (2,0.6) to [out=350,in=190] (3,0.6) to [out=320,in=160] (4,0) to [out=200,in=40] (3,-0.6) to [out=170,in=10] (2,-0.6) to [out=80,in=280] (2,0.6);
    \draw[thick, red, double=white] (2,0.6) to [out=350,in=190] (3,0.6) to [out=320,in=160] (4,0);
    \node at (2.5,0) {$\psi_n$};
    \filldraw[fill=white] (2,0.6) circle (2pt) node[left]{$#1$};
    \filldraw[fill=white] (2,-0.6) circle (2pt) node[left]{$#2$};
    \filldraw[fill=white] (4,0) circle (2pt) node[right]{$#3$};
    \end{tikzpicture}
}}

\newcommand{\PiNCaseOnePhiPtOneDouble}[3]{\raisebox{-17.5pt}{
    \begin{tikzpicture}[scale=0.8]
    \fill[fill=gray!50] (1,0.6) to [out=260,in=100] (1,-0.6) to [out=170,in=10] (0,-0.6) to [out=140,in=340] (-1,0) to [out=20,in=220] (0,0.6) to [out=350,in=190] (1,0.6);
    \draw[] (1,0.6) -- (2,0.6);
    \draw[thick, blue, double=white] (2,-0.6) -- (1,-0.6) to [out=170,in=10] (0,-0.6) to [out=140,in=340] (-1,0);
    \draw[thick, red, double=white] (-1,0) to [out=20,in=220] (0,0.6) to [out=350,in=190] (1,0.6);
    \node at (0.5,0) {$\psi_0$};
    \draw (1.5,-0.6) node[above]{$\circ$};
    \filldraw (1,-0.6) circle (2pt);
    \filldraw (1,0.6) circle (2pt);
    \filldraw[fill=white] (-1,0) circle (2pt) node[left]{$#1$};
    \filldraw[fill=white] (2,0.6) circle (2pt) node[right]{$#2$};
    \filldraw[fill=white] (2,-0.6) circle (2pt) node[right]{$#3$};
    \end{tikzpicture}
}}

\newcommand{\PiNCaseTwoPhiPtThreeDouble}[4]{\raisebox{-17.5pt}{
    \begin{tikzpicture}[scale=0.8]
    \fill[fill=gray!50] (2,0.6) to [out=280,in=80] (2,-0.6) to [out=10,in=170] (3,-0.6) to [out=100,in=260] (3,0.6) to [out=190,in=350] (2,0.6);
    \draw (3,0.6) -- (4,0.6);
    \draw (3,-0.6) -- (4,-0.6);
    \draw[thick, blue, double=white] (4,-0.6) -- (3,-0.6) to [out=170,in=10] (2,-0.6);
    \draw[thick, red, double=white] (2,0.6) to [out=350,in=190] (3,0.6);
    \node at (2.5,0) {$\psi$};
    \filldraw (3,0.6) circle (2pt);
    \filldraw (3,-0.6) circle (2pt);
    \draw (3.5,-0.6) node[above]{$\circ$};
    \filldraw[fill=white] (2,0.6) circle (2pt) node[left]{$#1$};
    \filldraw[fill=white] (2,-0.6) circle (2pt) node[left]{$#2$};
    \filldraw[fill=white] (4,0.6) circle (2pt) node[right]{$#3$};
    \filldraw[fill=white] (4,-0.6) circle (2pt) node[right]{$#4$};
    \end{tikzpicture}
}}

\newcommand{\PiNCaseTwoPhiPtThreeDoubleVTwo}[4]{\raisebox{-17.5pt}{
    \begin{tikzpicture}[scale=0.8]
    \fill[fill=gray!50] (2,0.6) to [out=280,in=80] (2,-0.6) to [out=10,in=170] (3,-0.6) to [out=100,in=260] (3,0.6) to [out=190,in=350] (2,0.6);
    \draw (3,-0.6) -- (4,-0.6);
    \draw[thick, red, double=white] (3,-0.6) to [out=170,in=10] (2,-0.6);
    \draw[thick, blue, double=white] (2,0.6) to [out=350,in=190] (3,0.6) --  (4,0.6);
    \node at (2.5,0) {$\psi$};
    \filldraw (3,0.6) circle (2pt);
    \filldraw (3,-0.6) circle (2pt);
    \draw (3.5,-0.6) node[above]{$\circ$};
    \filldraw[fill=white] (2,0.6) circle (2pt) node[left]{$#1$};
    \filldraw[fill=white] (2,-0.6) circle (2pt) node[left]{$#2$};
    \filldraw[fill=white] (4,0.6) circle (2pt) node[right]{$#3$};
    \filldraw[fill=white] (4,-0.6) circle (2pt) node[right]{$#4$};
    \end{tikzpicture}
}}

\newcommand{\PiNCaseThreePhiPartTwoDouble}[3]{\raisebox{-17.5pt}{
    \begin{tikzpicture}[scale=0.8]
    \fill[fill=gray!50] (2,0.6) to [out=350,in=190] (3,0.6) to [out=320,in=160] (4,0) to [out=200,in=40] (3,-0.6) to [out=170,in=10] (2,-0.6) to [out=80,in=280] (2,0.6);
    \draw[thick, blue, double=white] (4,0) to [out=200,in=40] (3,-0.6) to [out=170,in=10] (2,-0.6);
    \draw[thick, red, double=white] (2,0.6) to [out=350,in=190] (3,0.6) to [out=320,in=160] (4,0);
    \node at (2.5,0) {$\psi_n$};
    \filldraw[fill=white] (2,0.6) circle (2pt) node[left]{$#1$};
    \filldraw[fill=white] (2,-0.6) circle (2pt) node[left]{$#2$};
    \filldraw[fill=white] (4,0) circle (2pt) node[right]{$#3$};
    \end{tikzpicture}
}}

\newcommand{\PiNCaseThreePhiPartTwoDoubleVTwo}[3]{\raisebox{-17.5pt}{
    \begin{tikzpicture}[scale=0.8]
    \fill[fill=gray!50] (2,0.6) to [out=350,in=190] (3,0.6) to [out=320,in=160] (4,0) to [out=200,in=40] (3,-0.6) to [out=170,in=10] (2,-0.6) to [out=80,in=280] (2,0.6);
    \draw[thick, red, double=white] (4,0) to [out=200,in=40] (3,-0.6) to [out=170,in=10] (2,-0.6);
    \draw[thick, blue, double=white] (2,0.6) to [out=350,in=190] (3,0.6) to [out=320,in=160] (4,0);
    \node at (2.5,0) {$\psi_n$};
    \filldraw[fill=white] (2,0.6) circle (2pt) node[left]{$#1$};
    \filldraw[fill=white] (2,-0.6) circle (2pt) node[left]{$#2$};
    \filldraw[fill=white] (4,0) circle (2pt) node[right]{$#3$};
    \end{tikzpicture}
}}

\newcommand{\PiNCaseTwoGeneralPtOneBlueM}[4]{\raisebox{-17.5pt}{
    \begin{tikzpicture}[scale=0.8]
    \fill[fill=gray!50] (2,0.6) to [out=260,in=100] (2,-0.6) to [out=170,in=10] (0,-0.6) to [out=140,in=340] (-1,0) to [out=20,in=220] (0,0.6) to [out=350,in=190] (2,0.6);
    \draw[thick, blue, double=white] (-1,0) to [out=340,in=140] (0,-0.6) to [out=10,in=170] (2,-0.6);
    \node at (1,0) {$\psi^{#4}\psi_0$};
    \filldraw[fill=white] (-1,0) circle (2pt) node[left]{$#1$};
    \filldraw[fill=white] (2,0.6) circle (2pt) node[right]{$#2$};
    \filldraw[fill=white] (2,-0.6) circle (2pt) node[right]{$#3$};
    \end{tikzpicture}
}}

\newcommand{\PiNCaseOneGeneralPtTwoBlueM}[4]{\raisebox{-17.5pt}{
    \begin{tikzpicture}[scale=0.8]
    \fill[fill=gray!50] (1,0.6) to [out=280,in=80] (1,-0.6) to [out=10,in=170] (3,-0.6) to [out=40,in=200] (4,0) to [out=160,in=320] (3,0.6) to [out=190,in=350] (1,0.6);
\draw[thick, blue, double=white]  (1,-0.6) to [out=10,in=170] (3,-0.6) to [out=40,in=200] (4,0);
    \node at (2,0) {$\bar\psi^{#4}\bar\psi_0$};
    \filldraw[fill=white] (1,0.6) circle (2pt) node[left]{$#1$};
    \filldraw[fill=white] (1,-0.6) circle (2pt) node[left]{$#2$};
    \filldraw[fill=white] (4,0) circle (2pt) node[right]{$#3$};
    \end{tikzpicture}
}}

\newcommand{\LeftBlock}[4]{\raisebox{-17.5pt}{
    \begin{tikzpicture}[scale=0.8]
    \fill[fill=gray!50] (2,0.6) to [out=260,in=100] (2,-0.6) to [out=170,in=10] (0,-0.6) to [out=140,in=340] (-1,0) to [out=20,in=220] (0,0.6) to [out=350,in=190] (2,0.6);
    \node at (1,0) {$\psi^{#4}\psi_0$};
    \filldraw[fill=white] (-1,0) circle (2pt) node[left]{$#1$};
    \filldraw[fill=white] (2,0.6) circle (2pt) node[right]{$#2$};
    \filldraw[fill=white] (2,-0.6) circle (2pt) node[right]{$#3$};
    \end{tikzpicture}
}}

\newcommand{\RightBlockBar}[4]{\raisebox{-17.5pt}{
    \begin{tikzpicture}[scale=0.8]
    \fill[fill=gray!50] (1,0.6) to [out=280,in=80] (1,-0.6) to [out=10,in=170] (3,-0.6) to [out=40,in=200] (4,0) to [out=160,in=320] (3,0.6) to [out=190,in=350] (1,0.6);
    \node at (2,0) {$\bar\psi^{#4}\bar\psi_0$};
    \filldraw[fill=white] (1,0.6) circle (2pt) node[left]{$#1$};
    \filldraw[fill=white] (1,-0.6) circle (2pt) node[left]{$#2$};
    \filldraw[fill=white] (4,0) circle (2pt) node[right]{$#3$};
    \end{tikzpicture}
}}

\begin{document}
\maketitle

\begin{abstract}
We study a percolation model on $\mathbb{R}^d$ called the random connection model.
For $d$ large, we use the lace expansion to prove that the critical two-point connection probability decays like $|x|^{-(d-2)}$ as $|x| \to \infty$, with possible anisotropic decay. 
Our proof also applies to nearest-neighbour Bernoulli percolation on $\mathbb{Z}^d$ in $d \ge 11$ and simplifies considerably the proof given by Hara in 2008.
The method is based on the recent deconvolution strategy of Liu and Slade and uses an $L^p$ version of Hara's induction argument.
\end{abstract}

%


\section{Introduction and results}

The random connection model is a model of random graphs on $\Rd$,
with vertices sampled according to a Poisson process with intensity $\lambda > 0$, and with edges drawn independently according to an integrable, even adjacency function $\adj: \Rd \to [0,1]$.
A classical example is the Gilbert disk model \cite{Gilb61} or the Boolean model, for which $\adj(x) = \1\{\abs x \le R\}$ with some fixed parameter $R > 0$; this connects every pair of vertices within distance $R$ of each other.
In dimensions $d \ge 2$, 
the general model undergoes a percolation phase transition at a critical intensity $\lambda_c \in (0,\infty)$ \cite{Penr91}.
We are interested in the case where $d$ is sufficiently large (conjectured for all $d$ above the upper critical dimension $\dc = 6$), where the model exhibits \emph{mean-field} behaviour.

While quantities like the critical intensity $\lambda_c$ depend on the fine details of the percolation model (e.g., the exact form of the adjacency function $\adj$),
the critical and near-critical behaviours of the model are expected to be more universal. 
For example, the expected cluster size is believed to diverge polynomially as $\lambda \to \lambda_c^-$, with \emph{critical exponent} $\gamma$.
Similarly, the percolation probability is believed to decay polynomially as $\lambda \to \lambda_c^+$ with exponent $\beta$,
and the tail of the cluster size at $\lambda_c$ is believed to decay polynomially with exponent $1/\delta$. 
We refer to \cite{HH17book} for more details on these critical exponents for Bernoulli percolation.
The critical exponents are usually easy to compute in non-interacting analogues of the percolation models (such as spatial branching processes), and percolation models are said to exhibit mean-field behaviour if their critical exponents are the same as those in the non-interacting analogues. 
For the random connection model, 
it is proved in \cite{HHLM26, CD24} using the \emph{lace expansion} that when $d$ is sufficiently large, there is no infinite cluster at $\lambda_c$, the triangle condition holds, and the critical exponents take the mean-field values $\gamma=1$, $\beta=1$, and $\delta=2$.
In this paper, we study the critical exponent $\eta$, which describes the decay of the critical two-point function:
\begin{equation} \label{eq:eta_intro}
\P\crit(0 \connect x) \approx \frac 1 {\abs x^{d-2+\eta}} 
	\qquad (\abs x \to \infty).
\end{equation}
We prove that $\eta$ takes the mean-field value $\eta = 0$, in the sense that the ratio of the two sides of \eqref{eq:eta_intro} converges, 
in the same high-dimensional setting as \cite{HHLM26, CD24}.

For high-dimensional Bernoulli percolation on $\Zd$, the analogous statement that \eqref{eq:eta_intro} holds with $\eta = 0$ in the asymptotic sense is a celebrated result of Hara, van der Hofstad, Slade \cite{HHS03} and Hara \cite{Hara08}.
It is used, e.g., to construct the incipient infinite cluster (IIC) \cite{HJ04}. 
Also, the slighter weaker statement that \eqref{eq:eta_intro} holds with $\eta=0$, in the sense that the ratio of the two sides remains bounded away from $0$ and infinity, is used to establish one-arm exponents \cite{KN09,KN11} and study percolation in the half-space \cite{CH20} and on a large discrete torus \cite{HMS23}. 
(This form of $\eta = 0$ can also be proved for ``spread-out'' models using an alternative method \cite{DP25b}.)
We therefore expect our result to be useful for studying the IIC and the random connection model in a half-space or on a continuum torus.

Our proof uses the lace expansion, which provides a convolution equation for the two-point function. 
The lace expansion was originally introduced for the self-avoiding walk \cite{BS85,HS92a}, and it has been extended to many settings, including Bernoulli percolation \cite{HS90a,FH17}, the Ising model \cite{Saka07}, and the random connection model \cite{HHLM26}.
To derive the asymptotic behaviour of the critical two-point function, we use a deconvolution theorem on $\Rd$ proved in \cite{Liu25Rd}, which is based on and has weaker hypotheses than the $\Zd$ deconvolution theorem of Liu and Slade \cite{LS24a}. 
In particular, it is realised in \cite{Liu25Rd} (first for subcritical models in \cite{Liu25_subcritical}) that it is sufficient to estimate $L^p$ moments for the lace function $\Pi(x)$ 
(see \eqref{eq:Sigma} for precise conditions) 
to derive asymptotics for the two-point function.
Compared to the decay hypotheses on $\Pi(x)$ used in \cite{LS24a,Hara08}, moments of $\Pi(x)$ are easier to estimate, and we prove the required conditions using an induction argument similar to the one used by Hara in \cite[Proof of (1.47)]{Hara08}. 
This strategy allows us to bypass the second, more difficult diagrammatic estimate of \cite[Lemma~1.7]{Hara08}, also simplifying the proof for percolation on $\Zd$.

\medskip\noindent
{\bf Notation.}  We write
$a \vee b = \max \{ a , b \}$ and $a \wedge b = \min \{ a , b \}$.
We write
$f = O(g)$ or $f \lesssim g$ to mean there exists a constant $C> 0$ such that $\abs {f(x)} \le C\abs {g(x)}$, and write $f= o(g)$ to mean $\lim f/g = 0$.
We write $\| \cdot \|_\infty$ for the sup norm of a function, 
write $\norm f_p = (\int_\Rd \abs{ f(x)}^p \dd x)^{1/p} $ for $1 \le p < \infty$, 
and write
\begin{equation}
\norm f_{ L^{p_1} \cap L^{p_2} } = \max \bigl\{ 
	\norm f_{p_1} , \norm f_{p_2}  \big\} 
\end{equation}
for $p_1,p_2 \in [1, \infty]$.
We use the (inverse) Fourier transform
\begin{equation}
\hat f(k) = \int_\Rd f(x) e^{ik\cdot x}  \dd x,   
\qquad
g(x) = \int_\Rd \hat g(k) e^{-ik\cdot x}  \frac{ \dd k }{ (2\pi)^d }
\qquad (k,x \in \Rd),
\end{equation}
which is defined when $f$ or $\hat g$ is in $L^1(\Rd)$, respectively.

\subsection{The model and main result}

The random connection model is defined as follows.
Given $\lambda>0$, let $\eta$ be distributed as a Poisson point process on $\R^d$ with intensity $\lam\dd x$, where $\dd x$ is the Lebesgue measure. We will call $\eta$ our vertex set. 
Let $\adj: \Rd \to [0,1]$ be an integrable, even function (\ie, $\adj(-x) = \adj(x)$ for all $x\in \Rd$), which we call the \emph{adjacency function}.
For all pairs of distinct vertices $x \ne y\in\eta$, we independently draw an edge between them with probability $\adj(x-y)$.
This constructs a random graph $\xi$, whose measure we denote by $\P_\lam$.
Two vertices $x,y\in\eta$ are said to be \emph{adjacent}, denoted $x \sim y$, if there is an edge between them.
They are said to be \emph{connected in $\xi$}, denoted $x \leftrightarrow y$ in $\xi$, if either $x=y$ or if there exists a finite sequence of vertices $x=u_0,u_1,\ldots,u_k,u_{k+1}=y$ in $\eta$ such that $u_i\sim u_{i+1}$ for all $0\leq i\leq k$.

Given $x,y\in\R^d$, it is convenient to consider the Palm version of the random graph, defined on the vertex set $\eta^{x,y}=\eta\cup\{x,y\}$, with edges drawn independently between $x,y$, and all vertices of $\eta$ according to $\adj$ as above. The resulting Palm random graph is denoted $\xi^{x,y}$.
The \emph{two-point function} $\tlam\colon \R^d\to [0,1]$ of the model is defined to be the connection probability
\begin{equation}
\tlam(x) = \mathbb{P}_\lambda(\orig\leftrightarrow x\text{ in }\xi^{\orig,x})
\end{equation}
in the Palm random graph. 
For a more detailed construction of the model, we refer the reader to \cite{HHLM26}.

We now restrict our attention to the mean-field dimensions $d > 6$ and introduce more assumptions on the adjacency function $\adj$. The conditions are slightly technical, but two standard examples of adjacency functions that obey Assumption~\ref{ass:adj} are the disk model
\begin{equation} \label{eq:disk_model}
\adj(x) = \1\{\abs x \le R\}
\qquad \text{with} \qquad
R = R_d = \pi^{-1/2} \Gamma(\tfrac d 2 + 1)^{1/d},
\end{equation}
and the Gaussian model
\begin{equation}  \label{eq:gaussian_model}
\adj(x) = \frac 1  { (2\pi)^{d/2} } \exp\Bigl\{ - \frac 1 {2} \abs x^2 \Big\} ,
\end{equation}
where $| \cdot |$ is the Euclidean norm on $\Rd$. 
Note that $\connf$ depends on the dimension $d$, but we will suppress this dimension dependence in our notation.

\begin{assumption} \label{ass:adj}
We assume $\adj: \Rd \to [0,1]$ is an integrable, even function
with $\int_\Rd \adj(x)\D x = 1$ that satisfies the following properties:
\begin{enumerate}

\item[(i)]
For each $d$, we have
\begin{gather}
\label{eq:adj_hyp}
\abs x^2 \adj(x) \in L^1 \cap L^2(\R^d), 
\\
\label{eq:adj_2+eps}
\abs x^{2+\eps} \adj(x) \in L^1(\R^d) 
\qquad \text{ for some $\eps>0$} ,
\\
\label{eq:adj_d-2}
\abs x^{d-2} \adj(x) \in L^{p} \cap L^2 \cap L^\infty (\R^d)
\qquad \text{for some $1 \le p < \frac d {4}$} , 
\end{gather}
and $\adj(x) = o(\abs x^{-(d-2)})$ as $\abs x\to\infty$;

\item[(ii)] \emph{(Infrared bound)}
There is a constant $K_{\rm IR, \adj} > 0$ (independent of $d$) such that 
\begin{equation}
\hat \adj (0) - \hat \adj(k) \ge K_{\rm IR, \adj} (\abs k^2 \wedge 1)
\qquad (k\in \Rd);
\end{equation}

\item[(iii)] \emph{(Lace expansion hypothesis\footnote{
This is hypothesis (H1.2) of \cite{HHLM26}, 
since for $m > 3$ we can bound $\norm{ \adj^{*m} }_\infty \le \norm{ \adj^{*3} }_\infty \norm \adj_1^{m-3} \le g(d)$,
and since we can always make $g(d) \ge \zeta^d$ for some $\zeta\in(0,1)$ by taking the maximum.
})}
There exists a function $g: \N \to [0, \infty)$,
satisfying $g(d) \to 0$ as $d\to \infty$,
such that
\begin{equation}
\norm{ \adj * \adj * \adj }_\infty \le g(d) 
\end{equation}
and
\begin{equation}
\bignorm{ (\adj  * \adj)(x) \1\{\abs x \ge \eps_d\} }_\infty \le g(d) 
\end{equation}
with some $\eps_d \in [0, R_d)$, where $R_d = \pi^{-1/2} \Gamma(\frac d 2 + 1)^{1/d}$ is the radius of the ball with unit volume in $\Rd$.
\end{enumerate}
\end{assumption}

The condition that $\adj$ integrates to $1$ is for convenience and can always be achieved by rescaling space and intensity (see \cite[Section~5.1]{HHLM26}), so there is no loss of generality in fixing the radius in \eqref{eq:disk_model} or the variance of the Gaussian distribution in \eqref{eq:gaussian_model}.
Parts~(ii) and (iii) of Assumption~\ref{ass:adj} are used in \cite{HHLM26} to derive the lace expansion when $d$ is sufficiently large.
Part~(i) is additional but can be readily verified, \eg, when $\adj$ decays like $\adj(x) \le C_d (1 + \abs x)^{-(d+2+\rho)}$ with some $\rho > 0 \vee \frac{d-8}2$.

\begin{remark}
Our assumption \eqref{eq:adj_d-2} on $\abs x^{d-2} \adj(x)$ might seem restrictive, 
but in fact standard methods \cite{Uchi98,LL10} to derive the $\abs x^{-(d-2)}$ asymptotics for the \emph{random walk} two-point function, with step distribution $\adj$, already require the stronger condition that $\abs x^{d-2} \adj(x) \in L^1$.
If one is willing to assume $\adj(x) \le C_d (1 + \abs x)^{-(d+2+\rho)}$ with some $\rho > 0$, Hara's Gaussian lemma \cite{Hara08} can be used to derive $\abs x^{-(d-2)}$ asymptotics without the $(d-2)$-th moment.
\end{remark}

The following is our main result, which is a precise version of $\eta=0$ in \eqref{eq:eta_intro}.

\begin{theorem} \label{thm:main}
Let $\adj$ obey Assumption~\ref{ass:adj}
and let $a_d = \frac{ \Gamma(\frac{ d-2 } 2 ) }{ 2\pi^{d/2}}$.
If $d > d_0$ with $d_0 \ge 8$ sufficiently large, 
then there is a positive-definite diagonal matrix $\Sigma$ such that
\begin{equation} \label{eq:tau_asymp}
\tau\crit (x) \sim \frac{ a_d }{ \lambda_c \sqrt{\det \Sigma} } \frac 1 { ( x \cdot \Sigma\inv x )^{(d-2)/2} } 
	\qquad (\abs x \to \infty).
\end{equation}
The matrix $\Sigma$ is given explicitly in terms of lace expansion quantities in \eqref{eq:Sigma}.
\end{theorem}

Our method is not perturbative in itself, and it can be applied as long as the lace expansion converges. 
The restriction of Theorem~\ref{thm:main} to $d_0 \ge 8$, rather than the optimal $d_0 \ge \dc = 6$, is due to a use of the square diagram in Section~\ref{sec:simple} that is not completely necessary but simplifies the proof.
We do not aim to optimise the proof, 
as the lace expansion for the random connection model is only known to converge when $d_0$ is sufficiently large.
For percolation on $\Zd$, \cite{FH17} shows that we can take $d_0 = 10$, which the current proof covers.

\begin{remark}
Rather than ``nearest-neighbour'' models which we considered in this paper, one can introduce an additional large ``spread-out'' parameter to show the convergence of the lace expansion in any dimension $d > \dc = 6$.
Our method can then be used to prove \eqref{eq:tau_asymp} for sufficiently spread-out random connection model in $d > 8$. 
For the spread-out model in dimension $d = 7 \text{ or }8$, 
an alternative perturbative method based on convolution estimates \cite{HHS03, LS26b, BKM24} is available.
\end{remark}

\subsection{Strategy of proof}
\label{sec:proof_main}

We now present the framework of the proof. 
By the lace expansion of \cite{HHLM26}, under Assumption~\ref{ass:adj} there exists $d_0 \ge 6$ large such that the following holds for all dimensions $d > d_0$:
there is a family of functions $\Pi_\lambda: \Rd \to \R$ for which the lace expansion equation (Ornstein--Zerneke equation)
\begin{equation} \label{eq:OZ-intro}
\tau_\lambda = (\adj + \Pi_\lambda) + \lambda (\adj + \Pi_\lambda) * \tau_\lambda 
\end{equation}
holds for all $x\in \Rd$ and $\lambda \le \lambda_c$.
We rearrange this slightly to put it into the deconvolution framework of \cite{Liu25Rd}.
Let $J_\lambda = \lambda (\adj + \Pi_\lambda)$, 
and let $\delta$ denote the Dirac delta function (convolution identity).
Equation \eqref{eq:OZ-intro} can then be rewritten as
\begin{equation} \label{eq:OZ}
( \delta - J_\lambda ) * \lambda \tau_\lambda = J_\lambda .
\end{equation}
To get the asymptotics of $\tau\crit(x)$ using \cite[Theorem~1.3]{Liu25Rd}, we need to verify certain moment conditions on $J\crit = \lam_c (\adj + \Pi\crit)$.
The necessary conditions on $\adj$ have been put as part of Assumption~\ref{ass:adj}. For $\Pi\crit$, the conditions can be stated in terms of the following definition. Our goal is to prove that the $(d-2)$-th moment of $\Pi_\lam$ is good.

\begin{definition} \label{def:good}
Let $a \ge 0$.
If $a \le 2$, we say that the $a$-th moment of $\Pi_\lam$ is \emph{good} if
\begin{equation} 
\sup_{\lambda \in [\half \lambda_c, \lambda_c ) }
\bignorm{ \abs x^a \Pi_\lambda(x) }
	_{L^1 \cap L^\infty } 
	< \infty .
\end{equation}
If $a \in (2, d+2]$, we write 
\begin{equation}
p_a^* = \frac d { d - a + 2 } \in (1, \infty], 
\end{equation}
and we say that the $a$-th moment of $\Pi_\lam$ is \emph{good} if there exists some $p_a \in [1, p^*_a)$ for which
\begin{equation} \label{eq:Pi_good_moment}
\sup_{\lambda \in [\half \lambda_c, \lambda_c ) }
\bignorm{ \abs x^a \Pi_\lambda(x) }
	_{L^{p_a} \cap L^2 \cap L^\infty } 
	< \infty .
\end{equation}
\end{definition}

Definition~\ref{def:good} has been designed to allow an induction argument on $a$. We will use the usual convergence argument of the lace expansion to show that the second moment of $\Pi_\lam$ is good. Then we inductively increase $a$, until we reach the required $(d-2)$-th moment.
Hara used a similar strategy in \cite[Proof of (1.47)]{Hara08} on the $L^1$ norm of $\abs x^a \Pi(x)$. 
The use of $L^p$ norms is an innovation of this paper.
When $a=d-2$, we do not expect \eqref{eq:Pi_good_moment} to be finite with $p_a = 1$.

\begin{proposition} \label{prop:climb}
Let $d > d_0 \vee 8$, 
let $\phi \in (0, d-2)$ be an integer, and let $\theta = 2$. 
If $\phi + \theta < \frac 3 2 d -  4$, then
\begin{equation}
\text{The $\phi$-th moment of $\Pi_\lambda$ is good}
\quad \implies \quad
\text{The $(\phi+\theta)$-th moment of $\Pi_\lambda$ is good.}
\end{equation}
Moreover, if $\phi < \half d - 2$, 
then the $(\phi+\theta)$-th moment of $\Pi_\lambda$ is good with $p_\phitheta = 1$.
\end{proposition}

The following lemma is a slight extension of results of \cite{HHLM26} that we will prove in Section~\ref{sec:lace}.
It implies that the $\phi$-th moment of $\Pi_\lam$ is good when $\phi \in [0,2]$ and provides the base case of the induction.

\begin{lemma} \label{lem:Pi_base}
Let $d > d_0$. Then
\begin{equation} \label{eq:Pi_base}
\sup_{\lambda \in [\half \lambda_c, \lambda_c ] }
\max\Bigl\{
\norm{ \Pi_\lambda(x) }_{L^1 \cap L^\infty }  ,\,
\bignorm{ \abs x^2 \Pi_\lambda(x) }_{L^1 \cap L^\infty }  ,\,
\norm{ \tau_\lambda }_{L^2} 
\Bigl\} < \infty .
\end{equation}
We also have $\hat J\crit(0) = 1$ and the following \emph{infrared bound}:
there is a constant $K_{\rm IR} > 0$ such that
\begin{equation} \label{eq:J_infrared}
\hat J_\lambda (0) \le 1,  \qquad
\hat J_\lambda(0) - \hat J_\lambda(k) \ge K_{\rm IR} (\abs k^2 \wedge 1)
\end{equation}
uniformly in $k \in \Rd$ and in $\lambda \in [\half \lambda_c, \lambda_c]$.
\end{lemma}

Using the lemma, we can also get a Fourier integral representation for $\tau\crit(x)$, as follows. 
We first use the lace expansion equation \eqref{eq:OZ} to get, for $H_\lam = \lam \tau_\lam - J_\lam$, that
\begin{equation}
(\delta - J_\lam) * H_\lam 
= J_\lam - (J_\lam - J_\lam * J_\lam)
= J_\lam * J_\lam .
\end{equation}
We then take the $L^2$ Fourier transform, solve for $\hat H_\lam$, and then take the inverse Fourier transform.
Indeed, by Lemma~\ref{lem:Pi_base}, 
we have $J_\lam = \lam(\adj + \Pi_\lam) \in L^1 \cap L^\infty \subset L^2$
and $H_\lam = \lam \tau_\lam - J_\lam \in L^2$, so it is valid to take the Fourier transform. 
At $\lambda = \lambda_c$, this produces the Fourier integral
\begin{equation} \label{eq:Hint}
\lambda_c \tau\crit(x) = J\crit(x) + 
\int_\Rd \frac{ \hat J\crit(k)^2 }{ 1 - \hat J\crit(k) }  e^{-ik\cdot x} \frac{ \D k }{ (2\pi)^d } ,
\end{equation}
which is well-defined thanks to the infrared bound \eqref{eq:J_infrared}.

\begin{proof}[Proof of Theorem~\ref{thm:main}]
We take the $d_0$ produced by the lace expansion of \cite{HHLM26}, and we increase it if necessary to make $d_0 \ge 8$.
We use the Fourier integral representation \eqref{eq:Hint}. 
By Theorem~1.3 of \cite{Liu25Rd} (with $J = g = J\crit$ and $G = \lambda_c \tau\crit)$, if we can verify
\begin{enumerate} [label=(\roman*)]
\item
$J\crit(x), \abs x^2 J\crit(x) \in L^1 \cap L^2 (\Rd)$,

\item
$\abs x^{2+\eps} J\crit(x) \in L^1(\Rd)$ for some $\eps > 0$,

\item
$\abs x^{d-2} J\crit(x) \in L^p \cap L^2(\Rd)$ for some $1 \le p < d/4$,

\item
$\hat J\crit(0) = 1$ and the infrared bound \eqref{eq:J_infrared} at $\lam = \lambda_c$, and

\item
$J\crit(x) = o(\abs x^{-(d-2)})$ as $\abs x\to \infty$,
\end{enumerate}
then the asymptotic behaviour of the Fourier integral is given by
\begin{equation} \label{eq:Sigma}
\lambda_c\tau\crit (x) \sim \frac{ a_d \hat J\crit(0)}{ \sqrt{\det \Sigma} } \frac 1 { ( x \cdot \Sigma\inv x )^{(d-2)/2} } 
	\qquad (\abs x \to \infty),
\end{equation}
where $\Sigma$ is the diagonal matrix $\Sigma = \mathrm{diag} ( \int_\Rd x_i^2 J\crit(x) \D x : 1 \le i \le d )$ which is positive definite by the infrared bound \eqref{eq:J_infrared}. 
This rearranges to the desired result. 
Hypotheses (i) and (iv) directly follow from \eqref{eq:adj_hyp} and Lemma~\ref{lem:Pi_base}.
We verify (ii),(iii), and (v) using Proposition~\ref{prop:climb}, as follows.

Recall $\theta = 2$. Let
\begin{equation}
\phi_0 = \begin{cases} 
2	&(d \text{ is odd}) \\
1	&(d \text{ is even}),
\end{cases}
\qquad \phi_{i+1} = \phi_i + \theta
\quad (i \ge 0) .
\end{equation}
Since Lemma~\ref{lem:Pi_base} implies that the $\phi_0$-th moment of $\Pi_\lam$ is good, we can use Proposition~\ref{prop:climb} repeatedly, as long as $\phi_i < d-2$ and $\phi_i + \theta < \frac 3 2 d - 4 = (d-2) + (\half d - 2 )$.
The first use of the proposition gives
\begin{equation} 
\sup_{\lambda \in [\half \lambda_c, \lambda_c ) }
\bignorm{ \abs x^{\phi_0 + 2} \Pi_\lambda(x) }
	_{L^1 \cap L^\infty } 
	< \infty ,
\end{equation}
since $\phi_0 \le 2 < \half d - 2 $ in dimensions $d > 8$.
As $\Pi_\lambda \to \Pi\crit$ pointwise as $\lambda \to \lambda_c^-$ by \cite[Corollary~6.1]{HHLM26}, Fatou's lemma implies that
\begin{equation} 
\bignorm{ \abs x^{\phi_0 + 2} \Pi\crit(x) }_{L^1 \cap L^\infty}
\le \sup_{\lambda \in [\half \lambda_c, \lambda_c ) }
\bignorm{ \abs x^{\phi_0 + 2} \Pi_\lambda(x) }
	_{L^1 \cap L^\infty } 
	< \infty .
\end{equation}
The $L^1$ part of the estimate, together with \eqref{eq:adj_2+eps} in Assumption~\ref{ass:adj}, gives hypothesis~(ii).

By the choice of $\phi_0$, the last use of Proposition~\ref{prop:climb} always yields that the $(d-1)$-th moment of $\Pi_\lam$ is good.
By Definition~\ref{def:good}, this means that there exists some $p'  \in  [1, d/3)$ for which 
\begin{equation} 
\bignorm{ \abs x^{d-1} \Pi\crit(x) }_{L^{p'} \cap L^2 \cap L^\infty}
\le \sup_{\lambda \in [\half \lambda_c, \lambda_c ) }
\bignorm{ \abs x^{d - 1} \Pi_\lambda(x) }
	_{L^{p'} \cap L^2 \cap L^\infty}
	< \infty .
\end{equation}
The $L^\infty$ part of the estimate, together with Assumption~\ref{ass:adj}(i), gives
\begin{equation}
J\crit(x) = \lambda_c ( \adj(x) + \Pi\crit(x) )
= o \bigg(\frac 1 {\abs x^{d-2}}\bigg) + O\bigg(\frac 1 {\abs x^{d-1}}\bigg)
= o \bigg(\frac 1 {\abs x^{d-2}}\bigg) ,
\end{equation}
which is hypothesis~(v).
Also, using $\abs x^2 \Pi\crit(x) \in L^1 \cap L^\infty$ from Lemma~\ref{lem:Pi_base} and the decomposition
\begin{equation}
\abs x^{d-2} \Pi\crit(x) =  \big( \abs x^{ 2 } \Pi\crit(x) \big)^{ \frac{1}{d-3} } 
	\big( \abs x^{ d-1 } \Pi\crit(x) \big)^{ \frac{d-4}{d-3} } ,
\end{equation}
H\"older's inequality implies that
(using $\frac 1 {p'} > \frac 3 d = 1 - \frac{d-3}d$)
\begin{equation}
\abs x^{d-2} \Pi\crit(x) \in L^p \cap L^2 \cap L^\infty
\quad \text{with}\quad
\frac 1 p = \frac{ 1 }{ d-3 } + \frac{ d-4 }{(d-3) p' }
> \frac 4 d.
\end{equation}
Together with \eqref{eq:adj_d-2} in Assumption~\ref{ass:adj}, this gives hypothesis~(iii) with the smaller $p$.

This verifies all of (i)--(v) and completes the proof.
\end{proof}

\subsection{Organisation}

We have reduced the proof of Theorem~\ref{thm:main} to proving Proposition~\ref{prop:climb} and Lemma~\ref{lem:Pi_base}.
In Section~\ref{sec:lace}, 
we state our main $L^p$ diagrammatic estimate on the lace function $\Pi_\lam$ (Proposition~\ref{prop:diagram}), and we use it in conjunction with previous lace expansion results to prove Lemma~\ref{lem:Pi_base}.
The proof of Proposition~\ref{prop:diagram}, which differs from standard diagrammatic estimates only in the use of a few $L^p$ diagrams, is presented in Section~\ref{sec:diagram} and Appendix~\ref{sec:DiagramBounds}.
In Section~\ref{sec:simple},
we prove Proposition~\ref{prop:climb} using Proposition~\ref{prop:diagram}, by bounding the hypothesised $L^p$ diagrams using basic Fourier analysis; this part is based on the method of \cite{LS24a} and its extension to $\Rd$ by \cite{Liu25Rd}.

\section{Diagrammatic estimates}
\label{sec:lace}

\subsection{Diagrams}

We first define the quantities involved in the upper bound on $\Pi_\lam$.
For this, it is convenient to take $\lambda \in [\half \lam_c, \lam_c)$, so that all diagrams are finite. All estimates will be uniform in $\lambda$, so they also hold at $\lambda_c$ by a limiting argument, as in the proof of Theorem~\ref{thm:main}.

We recall, from \cite[Observation~4.3]{HHLM26}, that an upper bound of $\tau_\lam(x)$ is given by
\begin{equation} \label{eq:def_tilde_tau} 
\tilde \tau_\lambda(x) = \adj(x) + \tau_\plus(x),
\quad \text{where} \quad
\tau_\plus(x) = \lambda( \adj * \tau_\lambda )(x) .
\end{equation} 
We use superscripts on a function to denote multiplying (weighting) the function by a power of $\abs x$, e.g., $\tilde\tau_\lambda\supa(x) = \abs x^a \tilde\tau_\lambda(x)$. 
For any $a,b \ge 0$, any exponent $p \in [1,\infty]$, and any $u,v\in \Rd$, we define
\begin{align}
\label{eq:def_Wp}
W_p \supk{a,b} (u)  
	&= \bignorm{  \tilde \tau_\lambda\supa (x) 
		\tilde \tau_\lambda\supb ( u + x) }_{ L^p_x } ,
	\\
T \supb (u) 
	&= \lambda^2 ( \tau_\lambda\supb * \tlam * \tlam )(u) ,
	\\
\label{eq:def_Y}
Y \supa (x,y) &= \int_{\R^{3d}}
	\tau_\lam (z_1) \tau_\lam \supa (z_2 - z_1) \tau_\lam (x - z_2)
	 \tau_\lam (z_3 - z_1)    \nl
	&\qquad\qquad\qquad
	 \times \tau_\lam (z_2 - z_3) \tau_\lam (y - z_3)
	\lambda^3 \dd z_1 \dd z_2 \dd z_3  , 
	\\
\label{eq:def_H}
H_p \supab (u,v) &=  \biggnorm{ \int_\Rd
	Y \supa (x,y)
	\tau_\lam (y - u) \tau_\lam \supb ( v + x - y) 
	\lambda \dd y  }_{L^p_x} ,
\end{align}
where $L^p_x$ indicates the $L^p$ norm with respect to $x$.
All of $W_p \supab, T\supb, Y\supa, H_p \supab$ depend on $\lambda$ implicitly, and we write
\begin{equation}
\begin{alignedat}2
\bar E\supa &= 
\sup_{\lambda \in [\half \lambda_c, \lambda_c ) }
\bignorm{ \tilde \tau_\lam\supa }_\infty  , 
\qquad
\bar W_p\supab &&= 
\sup_{\lambda \in [\half \lambda_c, \lambda_c ) }
\bignorm{ W_p \supab }_\infty  , 
\\
\bar T\supb &= 
\sup_{\lambda \in [\half \lambda_c, \lambda_c ) }
\bignorm{ T \supb }_\infty  , 
\qquad
\bar H_p\supab &&= 
\sup_{\lambda \in [\half \lambda_c, \lambda_c ) }
\bignorm{ H_p \supab(u,v) }_{L^\infty_u L^\infty_v} 
\end{alignedat}
\end{equation}
for their supremum.
When $p=1$, 
\begin{equation} \label{eq:bubble}
W_1 \supk{a,b} (u)  
	= \bignorm{  \tilde \tau_\lambda\supa (x) 
		\tilde \tau_\lambda\supb (u+x) }_{ L^1_x } 
= ( \tilde \tau_\lambda\supa * \tilde \tau_\lambda\supb )(u) 
\end{equation}
reduces to the usual (weighted) open bubble diagram.
Similarly, $H_1\supab$ reduces to the usual (weighted) martini diagram.
The $p$-version of these quantities allows us to estimate the $L^p$ norm of $\Pi_\lam$.

\begin{proposition} \label{prop:diagram}
Let $d > d_0$, $\phi, \theta\ge 0$, and $p\in [1,\infty]$.
If $\norm{ \abs x^\theta \adj(x) }_1,
\bar E \supphi, \bar E \suptheta, 
\bar W_p \supphitheta, 
\bar T \suptheta, 
\bar H_p \supphitheta$, and $\bar H_p \supk{\phi,0}$ are all finite, 
then
\begin{equation}
\sup_{\lambda \in [\half \lambda_c, \lambda_c ) }
\bignorm{ \abs x^{\phi+\theta} \Pi_\lam(x) }_{p} 
< \infty.
\end{equation}
\end{proposition}

The $p=1$, $\theta = 0$ case of the proposition is quite standard, and it is essentially proved in Section~7.3 of \cite{HHLM26}.
The $p=1$, $\theta > 0$ case is first proved by \cite[Lemma~1.8]{Hara08} for percolation on $\Zd$.
We develop the method further to cover all $p \in [1,\infty]$, in Section~\ref{sec:diagram}.

The following elementary lemma is useful in verifying the hypothesis about $\bar H_\infty \supphitheta$ when $p=\infty$. 

\begin{lemma} \label{lem:H_infty}
Let $d\ge1$, $\phi,\theta \ge 0$, and $p=\infty$. Then
\begin{equation}
\bar H_\infty \supphitheta
\le \lambda_c^2 \norm{\tau\crit}_2^2 \bar E\supphi \bar T\supk 0
	\times
	\sup_{\lambda \in [\half \lambda_c, \lambda_c ) }
	\bignorm{ \tlam\suptheta * \tlam }_\infty.
\end{equation}
\end{lemma}

\begin{proof}
We write $a = \phi$ and $b = \theta$.
We first bound $Y \supa(x,y)$ in \eqref{eq:def_Y}.
In the definition of $Y \supa$, we bound $ \tau_\lam \supa (z_2 - z_1)$ by its sup norm, and we write the remaining integral over $z_2$ as a convolution, to get
\begin{align}
Y \supa (x,y) &\le  \bignorm{  \tau_\lam \supa }_\infty
	\int_{\R^{2d}}
	[ \lambda ( \tau_\lam * \tau_\lam ) (x - z_3) ]
	\tau_\lam (z_1)   \tau_\lam (z_3 - z_1)    \tau_\lam (y - z_3)
	\lambda^2 dz_1 dz_3   \nl
&\le  \bignorm{  \tau_\lam \supa }_\infty
	\lambda \norm{  \tau_\lam * \tau_\lam }_\infty
	\lam^2 (\tlam*\tlam*\tlam)(y) \nl
&\le \lambda \bar E \supa  \norm{  \tau_\lambda }_2^2 \bar T  \supk 0 .
\end{align}
Using this bound in the definition of $H_\infty \supab$ in \eqref{eq:def_H}, we get
\begin{align}
H_\infty \supab (u,v) 
&\le  \lambda \norm{  \tau_\lambda }_2^2  \bar E \supa \bar T  \supk 0
\biggnorm{ \int_\Rd
	\tau_\lam (y - u) \tau_\lam \supb ( v + x - y) 
	\lambda \dd y  }_{L^\infty_x} 	\nl
&= \lambda \norm{  \tau_\lambda }_2^2  \bar E \supa \bar T  \supk 0
	\lambda \bignorm{ (\tlam\supb * \tlam)(v+x-u) }_{L^\infty_x}.
\end{align}
Taking the supremum over $u,v,\lambda$ then gives the desired result.
\end{proof}

\subsection{Proof of Lemma~\ref{lem:Pi_base}}

We begin by establishing the finiteness of a few simple diagrams.
This slightly extends the results of \cite{HHLM26}.

\begin{lemma} \label{lem:base}
Let $d > d_0$.
Then $\bar E \supk 0, \bar E \supk 2, \bar T\supk0$, and $\norm{\tau\crit}_2$ are all finite.
\end{lemma}

\begin{proof}
The hypothesis $d > d_0$ is exactly the convergence of the lace expansion.
Finiteness of the triangle diagram $\bar T\supk0$ is the main result of \cite{HHLM26}.
For $\bar E \supk 0$, we use the definition of $\tilde \tau_\lambda$ in \eqref{eq:def_tilde_tau} and use $\norm \adj_\infty \le 1$ to get
\begin{equation}
\bar E \supk 0  = \norm{ \tilde \tau\crit }_\infty
	\le \norm \adj_\infty + \lambda_c \norm \adj_1 \norm{ \tau\crit}_\infty
	\le 1 + \lambda_c .
\end{equation}

For $\norm{\tau\crit}_2$ and $\bar E \supk 2$, we need to use more detailed estimates from \cite{HHLM26}.
By Proposition~5.10 there,
the convergence of the lace expansion gives a uniform bound $f(\lambda) \le 2$, $\lambda \in [0,\lambda_c)$, on the bootstrap function $f(\lambda)$.
Using $f(\lam) \le 2$, 
the $m=0$, $n=2$ case of \cite[Lemma~5.4]{HHLM26} gives
\begin{equation}
\lambda \norm{ \tau_\lam }_2^2
=  \lambda (\tau_\lam * \tau_\lam)(0)
\lesssim 1
\end{equation}
uniformly in $\lambda \in [0, \lambda_c )$.
Monotone convergence then shows $\lambda_c \norm{\tau\crit}_2^2 < \infty$. 

For $\bar E \supk2$, 
we have $\norm{ \abs x^2 \adj(x) }_\infty < \infty$ by $\norm \adj_\infty \le 1$ and the $L^\infty$ part of \eqref{eq:adj_d-2} in Assumption~\ref{ass:adj}, so it suffices to obtain a uniform estimate on $\norm{ \abs x^2 \tau_\plus(x)}_\infty$.
We claim
\begin{equation} \label{eq:base_claim}
\norm{ ( 1 - \cos(k\cdot x) ) \tau_\plus(x) }_{\infty}  
	\lesssim  \hat \adj(0) - \hat \adj(k) 
	\lesssim \abs k^2
\end{equation}
uniformly in $k\in \Rd$ and in $\lambda \in [0, \lambda_c )$.
Assume this for a moment, 
we fix a direction $\hat k = k / \abs k$ and use the pointwise convergence
\begin{equation} \label{eq:k_limit}
\lim_{\abs k \to 0} \frac { 1 - \cos(k \cdot x) } {\abs k^2}
= \half \abs x^2 ( \hat k \cdot \hat x )^2 ,
	\qquad \hat x = \frac x {\abs x} ,
\end{equation}
to get that
\begin{equation}
\biggnorm{ \half \abs x^2 ( \hat k \cdot \hat x )^2  \tau_\plus(x) }_\infty
\le \liminf_{\abs k \to \infty} \biggnorm{ \frac { 1 - \cos(k \cdot x) } {\abs k^2} \tau_\plus(x) }_\infty
\lesssim 1  .
\end{equation}
We then sum this over all coordinate directions $\hat k = e_j$
and use $\sum^d_{j=1} (e_j \cdot \hat x)^2 = \abs{\hat x}^2 =1$ to get $\norm{ \abs x^2 \tau_\plus(x) }_\infty \lesssim 2d$ uniformly in $\lambda \in [0,\lam_c)$, as desired.

It remains to prove \eqref{eq:base_claim}.
The second inequality follows from Taylor expansion and $\norm{ \abs x^2 \adj(x) }_1 < \infty$ from \eqref{eq:adj_hyp} in Assumption~\ref{ass:adj}.
For the first inequality,
we first use $\tau_\lam \le \tilde \tau_\lam$ from \eqref{eq:def_tilde_tau}, and then we distribute the $1-\cos(k\cdot x)$ factor into the convolutions using the cosine-spliting lemma \cite[Lemma~4.1]{HHLM26}
\begin{equation}
1 - \cos(k\cdot x)
\le N \sum_{j=1}^N ( 1 - \cos (k \cdot x_j) ),
	\qquad x = \sum_{j=1}^N x_j 
\end{equation}
with $N = 2$ or $3$. 
This produces the pointwise bound
\begin{align}
( 1 - \cos(k\cdot x) )  \lambda (\adj * \tau_\lam)(x)
&\le ( 1 - \cos(k\cdot x) ) ( \lambda \adj * \adj
	+ \lambda^2 \adj * \adj  * \tau_\lam ) (x)		\nl
&\le 4 \lam \adj_k * \adj + 6 \lam^2 \adj_k * \adj * \tau_\lam
	+ 3\lam^2 \tau_{\lam,k} * \adj * \adj ,
\end{align}
where $\adj_k(x) = ( 1 - \cos(k\cdot x) ) \adj(x)$
and $\tau_{\lam,k}(x) = ( 1 - \cos(k\cdot x) ) \tau_\lam(x)$.
The first term on the right-hand side is estimated in \cite[Lemma~5.7]{HHLM26}, while the second and the third term are estimated in \cite[Lemma~5.4]{HHLM26} (where they are denoted $6\widetilde W_\lam\supk{1,1}$ and $3W_\lam \supk{2,0}$ respectively).
Each of them is bounded by a constant multiple of $\hat \adj(0) - \hat \adj(k)$, and we obtain \eqref{eq:base_claim}.
This concludes the proof.
\end{proof}

\begin{proof}[Proof of Lemma~\ref{lem:Pi_base} assuming Proposition~\ref{prop:diagram}]
The boundedness of $\norm{\tau_\lam}_2 \le \norm{\tau\crit}_2$ has been established in Lemma~\ref{lem:base}.

The infrared bound on $\hat J_\lam = \lambda (\hat \adj + \hat \Pi_\lam)$ is essentially \cite[Theorem~1.2]{HHLM26}. In the proof of that theorem
(which denotes our $\lam\inv \hat J_\lam$ by $\hat a$),
a limiting argument is used to establish $\hat J_\lam(0) \le 1$, $\hat J\crit(0) = 1$, and
\begin{equation}
\abs{ \hat \Pi_\lam(0) - \hat \Pi_\lam(k) }
\le O(\beta) [ \hat \adj(0) - \hat \adj(k)] 
\end{equation}
uniformly in $\lam \le \lam_c$, 
where $\beta$ is the small parameter of the lace expansion.
By restricting to $\lam \in [\half \lam_c, \lam_c]$ and by using the small $\beta$, we have
\begin{equation}
\hat J_\lambda(0) - \hat J_\lambda(k)
= \lambda [ \hat \adj(0) - \hat \adj(k) + \hat \Pi_\lambda(0) - \hat \Pi_\lambda(k) ]
\ge \frac 1 4 \lambda_c [ \hat \adj(0) - \hat \adj(k)] ,
\end{equation}
and we get the desired infrared bound from that of $\hat \adj$ in Assumption~\ref{ass:adj}(ii).

For the estimates on $\Pi_\lam$,
we recall from \cite[Corollary~5.3]{HHLM26} that
\begin{equation}
\norm{ \Pi_\lambda }_{L^1  \cap L^\infty} \lesssim 1,
\qquad
\norm{ ( 1 - \cos(k\cdot x) ) \Pi_\lambda(x) }_{1}  
	\lesssim \beta [ \hat \adj(0) - \hat \adj(k) ]
	\lesssim \beta \abs k^2
\end{equation}
uniformly in $k\in \Rd$ and in $\lambda \in [0, \lambda_c )$.
Since $\Pi_\lambda \to \Pi\crit$ pointwise as $\lambda \to \lambda_c^-$ by \cite[Corollary~6.1]{HHLM26}, these uniform estimates extend to $\lambda = \lambda_c$ by Fatou's lemma,
so we get the $\norm{ \Pi_\lambda }_{L^1  \cap L^\infty}$ part of \eqref{eq:Pi_base}.
Also, the $\abs k\to 0$ limiting argument in \eqref{eq:k_limit} applies to give $\norm{ \abs x^2 \Pi_\lam(x) }_1 \lesssim \beta$, so we only need to estimate $\norm{ \abs x^2 \Pi_\lam(x) }_\infty$.

For this, we use Proposition~\ref{prop:diagram} with $\phi = 2$, $\theta=0$, and $p=\infty$. The proposition gives a uniform estimate in $\lambda \in [\half \lambda_c, \lambda_c)$, which then extends to $\lambda_c$ by Fatou's lemma.
To verify its hypothesis, 
since $\theta=0$ and $\norm \adj_1 = 1$,
we only need to show that
$\bar E \supk 2, \bar E \supk 0, 
\bar W_\infty \supk{2,0}, 
\bar T \supk 0$, 
and $\bar H_\infty \supk{2,0}$ are finite.
Since $\bar W_\infty \supk{2,0} 
	\le \bar E \supk 2 \bar E \supk 0$ and
\begin{equation}
\bar H_\infty \supk{2,0} 
	\le  ( \lambda_c \norm{\tau\crit}_2^2 )^2 \bar E\supk 2 \bar T \supk 0
\end{equation}
by Lemma~\ref{lem:H_infty},
we can conclude using Lemma~\ref{lem:base}.
\end{proof}

\section{Proof of Proposition~\ref{prop:climb}}
\label{sec:simple}

In this section, 
we use Proposition~\ref{prop:diagram} to prove Proposition~\ref{prop:climb}.
The main work is to use the good $\phi$-th moment of $\Pi_\lam$ to estimate the diagrams in the hypothesis of Proposition~\ref{prop:diagram}.
This is stated as Lemma~\ref{lem:simple}, whose proof uses basic Fourier analysis and constitutes most of the section.

\subsection{Bound on diagrams}
\begin{lemma} \label{lem:simple}
Let $d > d_0 \vee 8$ and $\phi \in [0, d-2]$.
Suppose the $\phi$-th moment of $\Pi_\lambda$ is good.
Then, for any integers $a,b \in [0, \phi \vee 2]$:

\begin{enumerate}[label=(\roman*)]
\item
\emph{(Edge)}
If $a < d-2$ and $s \in [2,\infty]$ are such that $s\inv < (d-2-a)/d$, 
then
\begin{equation}
\bar E \supa_s:=  \sup_{\lambda \in [\half \lambda_c, \lambda_c ) }
	\bignorm{ \tilde\tau_\lambda\supa }_{L^s \cap L^\infty} 
	< \infty .
\end{equation}
In particular, $\bar E \supa = \bar E_\infty \supa < \infty$.

\item 
\emph{(Bubble)}
Let  $a, b < d-2$.
If both $a,b < \half d - 2$, then
\begin{equation} \label{eq:bubble_small}
\sup_{ p \in [1, \infty] }
\bar W_p \supab < \infty .
\end{equation}
If $2 < a + b < \frac 3 2 d -4$,
then there exists $ p_\ab \in [1, p^*_\ab) $ for which
\begin{equation} \label{eq:bubble_large}
\sup_{ p \in [p_\ab \wedge 2, \infty] }
\bar W_p \supab < \infty .
\end{equation}

\item 
\emph{(Triangle)}
For $b=2$, $\bar T \supb < \infty$.

\item 
\emph{(Martini)}
If $a < d-2$ and $b = 0$ or $2$, then
\begin{equation}
\sup_{ p \in [1, \infty] }
\bar H_p \supab < \infty .
\end{equation}
\end{enumerate}
\end{lemma}

\begin{remark}
The restriction on integers $a,b$ is due to an integer number of derivatives used in the proof. 
It might be possible to use fractional derivatives (like in \cite[Section~2.3]{LS24a} or \cite[Section~4.3]{Hara08}) to extend to non-integer $a,b$, but the current lemma is sufficient for our purposes.
Also, the restriction on $d>8$ is mainly due to a use of the square diagram in our bound on $\bar H_1 \supab$. (One can use $\bar T \supb$ with $b < 1$ instead of $\bar T \supk 2$ in dimension $d = 7\text{ or }8$.)
The proof of Proposition~\ref{prop:climb} uses only $\bar H_{p_\ab}\supab$ (like $\bar W_{p_\ab}\supab$ in \eqref{eq:bubble_large}), which should be finite in dimension $7$ or $8$. 
\end{remark}

\begin{proof}[Proof of Proposition~\ref{prop:climb}]
Let $d > d_0 \vee 8$, $\theta = 2$, and $\phi \in (0,d-2)$ be an integer. 
We assume the $\phi$-th moment of $\Pi_\lam$ is good, which allows us to use Lemma~\ref{lem:simple}.
To prove that the $(\phi+\theta)$-th moment of $\Pi_\lam$ is good, we use Proposition~\ref{prop:diagram} with $p = p_\phitheta$, $2$, and $\infty$, where $p_\phitheta \in [1, p_\phitheta^*)$.
For the hypothesis, we need to verify that $\norm{ \abs x^\theta \adj(x) }_1,
\bar E \supphi, \bar E \suptheta, \bar W_p \supphitheta, \bar T \suptheta, \bar H_p \supphitheta$, and $\bar H_p \supk{\phi,0}$ are all finite with those values of $p$.

Since $\theta = 2$, we have 
$\abs x^\theta \adj(x) \in L^1$ by \eqref{eq:adj_hyp} in Assumption~\ref{ass:adj}, 
$\bar E \suptheta < \infty$ by Lemma~\ref{lem:base}, and
$\bar T \suptheta < \infty$ by Lemma~\ref{lem:simple}(iii). 
Also, 
since $\phi < d-2$ is an integer, Lemma~\ref{lem:simple}(i) with $a = \phi$ and $s = \infty$ gives $\bar E \supphi < \infty$,
and Lemma~\ref{lem:simple}(iv) gives both $\bar H_p \supphitheta, \bar H_p \supk{\phi,0} < \infty$.
Finally, for $\bar W_p \supphitheta$, 
we use Lemma~\ref{lem:simple}(ii) with $a = \phi$ and $b = \theta = 2$. 
When $\phi < \half d - 2$ we have \eqref{eq:bubble_small}, so we can take $p_\phitheta = 1$.
When $\phi + \theta < \frac 3 2 d - 4$ only, the lemma gives \eqref{eq:bubble_large} with some unspecified $p_\phitheta \in [1, p_\phitheta^*)$ that is still sufficient for the proof.
\end{proof}

In the remainder of Section~\ref{sec:simple}, we prove Lemma~\ref{lem:simple}.
We will use the inverse Fourier transform to estimate the diagrams in Lemma~\ref{lem:simple}(i) and~(iii).
The diagrams in Lemma~\ref{lem:simple}(ii) and~(iv) then follow from classical inequalities in the physical space.

\subsection{Fourier transform estimates}

In this subsection, we derive Fourier integral representations that will be useful to bound $\bar E_s \supa$ and $\bar T\supb$. 
We start from $\tilde \tau_\lam$.
Recall the notation $J_\lambda = \lambda (\adj + \Pi_\lambda)$
and the lace expansion equation \eqref{eq:OZ-intro}
\begin{equation}
\lambda \tau_\lambda = J_\lam + J_\lam * \lambda \taulam.
\end{equation}
Since $\tau_\plus = \lambda \adj * \taulam$, it then satisfies
\begin{equation}
\tau_\plus = \adj * (J_\lam + J_\lam * \lambda \taulam)
= \adj * J_\lam + J_\lam * \tau_\plus ,
\end{equation}
which can be rearranged as 
\begin{equation}
(\delta - J_\lambda) *  \tau_\plus = \adj * J_\lam.
\end{equation}
By taking the Fourier and inverse Fourier transforms, we obtain
\begin{equation} \label{eq:tau_int}
\tilde \tau_\lambda(x) = \adj(x) +  \int_{\R^d} \hat \tau_\plus(k)  e^{-ik\cdot x} \frac{ \D k }{ (2\pi)^d }, 
\qquad
\hat \tau_\plus(k) = \frac{ \hat \adj \hat J_\lambda }{ 1 - \hat J_\lambda }(k) ,
\end{equation}
which allows us to estimate $\tilde \tau_\lam$ by estimating $\adj$ and $\hat \tau_\plus$.

For weighted versions of $\tilde \tau_\lam$, we first use the elementary inequality
\begin{equation}
\abs x^a
\le ( \sqrt d \norm x_\infty)^a
\le d^{a/2} \sum_{j=1}^d \abs{x_j}^a
	\qquad (a\ge 0)
\end{equation}
to distribute $\abs x^a$ into each coordinates.
Then, since the Fourier transform of $x_j^a \tau_\plus(x)$ is $(-i \grad_j)^a \hat \tau_\plus$, we can bound
\begin{align} \label{eq:tau_fourier}
\abs x^a \tilde \tau_\lam(x)
&\le  \abs x^a \adj(x) +
	d^{a/2} \sum_{j=1}^d \abs{ x_j^a \tau_\plus(x) }  		\nl
&= \abs x^a \adj(x) +  d^{a/2} \sum_{j=1}^d    \biggabs{
	\int_\Rd (-i \grad_j)^a \hat \tau_\plus(k)  e^{-ik\cdot x} \frac{ \D k }{ (2\pi)^d }     	} .
\end{align}
Similarly, when $b \ge 0$ is an even integer, using $\abs {y_j}^b = y_j^b$, a similar Fourier integral upper bound can be obtained via
\begin{align} \label{eq:triangle_fourier}
\lambda^2 ( \tau_\plus \supb * \taul * \taul )(x)
&\le d^{b/2} \sum_{j=1}^d  \Bigl( [ y_j^b \tau_\plus(y) ] * \lambda\taul * \lambda\taul \Big)(x)  	\nl
&= d^{b/2} \sum_{j=1}^d  \int_\Rd 
	\Big[ (-i \grad_j)^b \hat \tau_\plus(k)  \Big]
	\Big[ \frac {\hat J_\lambda}{ 1 - \hat J_\lambda}(k) \Big]^2
	e^{-ik\cdot x} \frac{ \D k }{ (2\pi)^d } .
\end{align}

We next estimate derivatives of $\hat \adj$ and $\hat J_\lambda$;
these can be used to get derivatives of $\hat \tau_\plus$ via the product and quotient rules of derivatives.
Since $J_\lam = \lambda( \adj + \Pi_\lam)$, using Lemma~\ref{lem:Pi_base} and \eqref{eq:adj_hyp} in Assumption~\ref{ass:adj}, the $L^1$ Fourier transform immediately gives that $\hat J_\lam$ is twice (classically) differentiable, with 
\begin{equation} \label{eq:Jhat_small}
\sup_{\lambda \in [\half \lambda_c, \lambda_c ) }
	\bignorm{ \grad^\gamma \hat J_\lam }_{L^2 \cap L^\infty} 
	< \infty 
\qquad (\abs \gamma \le 2).
\end{equation}
The following lemma extends the above estimate to higher-order \emph{weak} derivatives. 
The weak derivative is defined for locally integrable functions via the usual integration by parts formula;
we refer to \cite[Chapter~5]{Evan10} for a detailed introduction.
For us, the important point is that the weak derivative allows us to make sense of the right-hand sides of \eqref{eq:tau_fourier} and \eqref{eq:triangle_fourier}, under a weak \emph{a priori} estimate on $\Pi_\lam(x)$.
The estimates on $\grad^\gamma \hat \adj$ and $\grad^\gamma \hat J_\lam$ in the following lemma are not optimal but are sufficient for our purpose; stronger estimates can be found in \cite[Section~2]{Liu25Rd}.

\begin{lemma} \label{lem:fourier}
Let $d > d_0$ and $\phi \in [0, d-2]$.
Suppose the $\phi$-th moment of $\Pi_\lambda$ is good.
Then $\hat J_\lambda$ is $\floor \phi$ times weakly differentiable on $\Rd$, and:

\begin{enumerate}[label=(\roman*)]
\item
Let $\gamma$ be a multi-index with $\abs \gamma \le \phi \vee 2$.
On the ball $B = \{ \abs k \le 1 \}$,
if $q\in [1,\infty)$ is chosen so that $q \inv > \abs \gamma / d$,
then
\begin{equation}
\sup_{\lambda \in [\half \lambda_c, \lambda_c ) }
\max\biggl\{
\bignorm{ \grad^\gamma \hat \varphi }_{L^q(B)} 
,\,
\bignorm{ \grad^\gamma \hat J_\lam }_{L^q(B)} 
,\,
\1_{\gamma\ne0}
\biggnorm{ \frac{  \grad^{\gamma} \hat J_\lam  }{ 1 - \hat J_\lam } }_{L^q(B)}
\bigg\} < \infty .
\end{equation}

\item
Let $n \ge 2$
and $\gamma_i$ be multi-indices such that $\sum_{i=1}^n \abs{ \gamma_i } \le \phi \vee 2$. 
If $r\in [1,\infty]$ is chosen so that 
$r\inv \ge (  -2 + \sum_{i=1}^n \abs{ \gamma_i }  ) / d  $, then
\begin{equation}
\sup_{\lambda \in [\half \lambda_c, \lambda_c ) }
\biggnorm{
\grad^{\gamma_1} \hat \adj
\prod_{i=2}^n \grad^{\gamma_i} \hat J_\lam 
}_{L^r(\R^d) } < \infty .
\end{equation}
\end{enumerate}
\end{lemma}

\begin{proof}
Using Definition~\ref{def:good} and Lemma~\ref{lem:Pi_base}, the $\phi$-th moment of $\Pi_\lambda$ being good implies that $(1+\abs x^{\floor \phi}) \Pi_\lambda(x)$ belongs to $L^2(\Rd)$, with the norm uniform in $\lambda$.
Taking the $L^2$ Fourier transform, we immediately get that $\hat \Pi_\lam$ is $\floor \phi$ times weakly differentiable on $\Rd$, by Theorem~8 in \cite[Section~5.8.5]{Evan10}.
Since $\phi \le d-2$, we have $(1+\abs x^{\floor \phi}) \adj(x) \in L^2(\Rd)$ too by Assumption~\ref{ass:adj}(i). It follows that $\hat \adj$ and $\hat J_\lam = \lambda(\hat \adj + \hat \Pi_\lam)$ are $\floor \phi$ times weakly differentiable.
The rest of the claims are consequences of the moment estimates and the boundedness of the $L^p$ Fourier transform (the Hausdorff--Young inequality), by the following argument from \cite[Section~2]{Liu25Rd}.

\smallskip\noindent
(i) We first estimate $\grad^\gamma \hat \Pi_\lam$.
If $\abs \gamma \le 2$, we use \eqref{eq:Pi_base} and the $L^1$ Fourier transform to get $\grad^\gamma \hat \Pi_\lam \in L^\infty(\Rd) \subset L^q(B)$ for any $q\in [1, \infty]$, with the norm uniform in $\lambda$.
If $ \abs \gamma > 2$, 
we write $a = \phi$, $b = \abs \gamma$, and decompose
\begin{equation}
\abs x^b \Pi_\lam(x) =  \big( \abs x^{ 2 } \Pi_\lam(x) \big)^{ \frac{a-b}{a-2} } 
	\big( \abs x^{ a } \Pi_\lam(x) \big)^{ \frac{b-2}{a-2} } .
\end{equation}
Using H\"older's inequality, \eqref{eq:Pi_base}, and \eqref{eq:Pi_good_moment}, we get that
\begin{equation}
\sup_{\lambda \in [\half \lambda_c, \lambda_c ) }
\bignorm{ \abs x^b \Pi_\lam(x) }_{L^2 \cap L^{p_b}}
< \infty
\quad \text{with} \quad
\frac 1 {p_b} = \frac{ a-b }{a-2} + \frac { b-2 }{ (a-2) p_\phi } ,
\end{equation}
where $p_\phi \in [1, p_\phi^*)$.
Taking the Fourier transform and using the definition of $p^*_\phi$, we obtain
\begin{equation} \label{eq:h_gamma_proof}
\sup_{\lambda \in [\half \lambda_c, \lambda_c ) }
\bignorm{
\grad^\gamma \hat \Pi_\lam }_{L^2 \cap L^{q_b \vee 2}} 
< \infty ,
\quad \text{where} \quad
\frac 1 {q_b} = 1 - \frac 1 {p_b}
	= \frac { b-2 }{ a-2 } \Big( 1 - \frac 1 {p_\phi} \Bigr)
	< \frac{ b-2}d. 
\end{equation}
By restricting to the compact subset $B$, the $L^q(B)$ spaces are nested, and we get 
$\grad^\gamma \hat \Pi_\lam \in L^{d/(\abs \gamma - 2)}(B) \subset L^q(B)$ for any $q\inv \ge \abs \gamma / d$, with the norm uniform in $\lambda$.

Exactly the same argument works for $\grad^\gamma \adj$, with the only changes being to use \eqref{eq:adj_hyp} and \eqref{eq:adj_d-2} instead of \eqref{eq:Pi_base} and \eqref{eq:Pi_good_moment}, and to use $a = d-2$ and $p < d / 4 = p^*_{d-2}$ from \eqref{eq:adj_d-2} instead of $a = \phi$ and $p_\phi$.
The desired estimate on derivatives of $\hat J_\lam = \lambda (\hat \adj + \hat \Pi_\lambda)$ follows.

To estimate $\grad^\gamma \hat J_\lam / (1 - \hat J_\lam)$, 
we use the infrared bound \eqref{eq:J_infrared} to bound
\begin{equation} \label{eq:h_loc_pf}
\biggabs{ \frac{ \grad^\gamma \hat J_\lam }{ 1 - \hat J_\lam } (k) }
\lesssim  \frac{ \abs{ \grad^\gamma \hat J_\lam (k)} }{ \abs k^2 } \1_B(k) 
	+ \abs{ \grad^\gamma \hat J_\lam (k)} .
\end{equation}
We have already estimated the second term.
For the first term,
if $\abs \gamma = 1$, 
we use Taylor's theorem, evenness of $J_\lam(x)$, and finiteness of its second moment to bound 
$\abs{ \grad^\gamma \hat J_\lam(k) } \lesssim \abs k$ uniformly in $\lambda$. When divided by $\abs k^2$, the ratio is thus bounded by a multiple of $\abs k \inv \1_B$, which is in $L^q(B)$ for all $q\inv > 1/d$, as desired.
If $\abs \gamma \in [2, \phi]$, 
we apply H\"older's inequality using the previous estimate $\grad^\gamma \hat J_\lam \in L^{d/ (\abs \gamma - 2)}(B)$, in conjunction with that $\abs k^{-2} \in L^p(B)$ for $p\inv > 2/d$.
Their product is thus is in $L^q(B)$ with $q\inv > (\abs \gamma - 2 + 2)/d = \abs \gamma / d$, with the norm uniform in $\lambda$, as desired.

\smallskip\noindent
(ii) This is proved in \cite[Lemma~2.3]{Liu25Rd}, by carefully choosing $L^q$ spaces for $\grad^\gamma \hat \adj$ and $\grad^\gamma \hat J_\lam$ in \eqref{eq:h_gamma_proof}.
\end{proof}

\subsection{Proof of Lemma~\ref{lem:simple}} 

We can now prove Lemma~\ref{lem:simple}. 
We prove in the order (i), (iii), (ii), and then (iv).

\subsubsection{Edge}

\begin{proof}[Proof of Lemma~\ref{lem:simple}(i)]
Let $a$ be an integer in $[0, \phi \vee 2]$.
We want to prove that, if $a < d-2$ and $s \in [2,\infty]$ is such that $s\inv < (d-2-a)/d$, then
\begin{equation}
\bar E \supa_s=  \sup_{\lambda \in [\half \lambda_c, \lambda_c ) }
	\bignorm{ \tilde\tau_\lambda\supa }_{L^s \cap L^\infty} 
	< \infty .
\end{equation}
We use \eqref{eq:tau_fourier};
since $\abs x^a \adj(x) \in L^2 \cap L^\infty \subset L^s \cap L^\infty$ by Assumption~\ref{ass:adj}(i), we only need to estimate the norm of $x_j^a \tau_\plus(x)$ for all $j$.
By the boundedness of the inverse Fourier transform (the Hausdorff--Young inequality), we have
\begin{equation} \label{eq:edge_pf0}
\bignorm{ x_j^a \tau_\plus(x) }_{s \vee 2 }
\lesssim \bignorm{  \grad_j^a \hat \tau_\plus }_{r\wedge 2} ,
\end{equation}
where $s, r \in [1,\infty]$ are related by $s\inv + r\inv = 1$.
This reduces our task to showing that $\hat \tau_\plus$ is $a$ times weakly differentiable on $\Rd$, and to getting a uniform estimate for the right-hand side of \eqref{eq:edge_pf0} when $r\inv > (2 + a) /d$. 
We will estimate all derivatives of $\hat \tau_\plus$ together, using the product and quotient rules of weak derivatives.\footnote{
See \cite[Appendix~A]{LS24a} for the case of $\T^d = (\R / 2\pi \Z)^d$. The proofs for $\Rd$ are the same, by replacing all $L^p(\T^d)$ spaces by local $L^p(\Rd)$ spaces.}

Let $\alpha$ be a multi-index with $\abs \alpha \le a$.
For the differentiability of $\hat \tau_\plus$, we need to show that all terms produced by the usual product and quotient rules are locally integrable. That is, we need to show that all terms of the form
\begin{equation} \label{eq:tau_alpha_decomp}
\frac{ (\grad^{\alpha_1} \hat \adj) (\grad^{\alpha_2} \hat J_\lam) }{ 1- \hat J_\lam}
\prod_{i=1}^n \frac{  - \grad^{\gamma_ i} \hat J_\lam  }{ 1 - \hat J_\lam } ,
\end{equation}
where $\alpha = \alpha_1 + \alpha_2 + \alpha_3$, 
$0 \le n \le \abs {\alpha_3 }$, $\abs {\gamma_i} \ge 1$, 
and $\sum_{i=1}^n \gamma_i = \alpha_3$, 
are locally integrable.
Then the weak derivative $\grad^\alpha \hat \tau_\plus$ would be given by a linear combination of terms like \eqref{eq:tau_alpha_decomp}, via the usual calculus rules.

Indeed, on the ball $B = \{ \abs k \le 1 \}$, H\"older's inequality and the infrared bound $ [1 - \hat J_\lam(k)]\inv \lesssim \abs k^{-2} \vee 1 = \abs k^{-2}$ imply that the $L^r(B)$ norm of \eqref{eq:tau_alpha_decomp} is bounded by a constant multiple of
\begin{equation} \label{eq:tau_alpha_decomp_norm}
\biggnorm{ \frac 1 {\abs k^2 } }_{L^q(B)}
\bignorm{ \grad^{\alpha_1} \hat \adj }_{L^{r_1}(B)}
\bignorm{ \grad^{\alpha_2} \hat J_\lam }_{L^{r_2}(B)}
\prod_{i=1}^n \biggnorm{ \frac{  - \grad^{\gamma_ i} \hat J_\lam  }{ 1 - \hat J_\lam } }_{L^{q_i}(B)}  ,
\end{equation}
where $r$ satisfies $r\inv = q\inv + r_1\inv + r_2 \inv  + \sum_{i=1}^n q_i\inv$.
The norm of $\abs k^{-2} $ is finite when $q\inv > 2/d$.
Also, by Lemma~\ref{lem:fourier}(i), the remaining norms in \eqref{eq:tau_alpha_decomp_norm} are finite and uniform in $\lambda$ when $r_i\inv > \abs {\alpha_i} /d$ and $q_i\inv > \abs{\gamma_i} / d$.
Therefore, 
\begin{equation} \label{eq:edge_pf1}
\sup_{\lambda \in [\half \lambda_c, \lambda_c ) }
\bignorm{ \eqref{eq:tau_alpha_decomp} }_{L^r(B)}  < \infty
\quad \text{when} \quad
\frac 1 r > \frac { 2 + \abs{\alpha_1} + \abs{\alpha_2} 
		+  \sum_{i=1}^n \abs{\gamma_i}  } d
= \frac { 2+ \abs \alpha } d .
\end{equation}
In particular, we can take $r=1$, since $\abs \alpha \le a < d-2$ by assumption.

On the complement $B^c = \{ \abs k >1 \}$, 
the other alternative of the infrared bound gives $ [1 - \hat J_\lam(k)]\inv \lesssim \abs  k^{-2} \vee 1 = 1$, so 
\begin{equation}
\bigabs{ \eqref{eq:tau_alpha_decomp} }
\lesssim 
\bigabs{ \grad^{\alpha_1} \hat \adj }
\bigabs{ \grad^{\alpha_2} \hat J_\lam }
\prod_{i=1}^n \bigabs{ \grad^{\gamma_ i} \hat J_\lam  } .
\end{equation}
Lemma~\ref{lem:fourier}(ii) then immediately implies that
\begin{equation}
\sup_{\lambda \in [\half \lambda_c, \lambda_c ) }
\bignorm{ \eqref{eq:tau_alpha_decomp} }_{L^r(B^c)}  < \infty
\quad \text{when} \quad
\Big( \frac{ -2 + \abs \alpha  } d \Big) \vee 0
\le \frac 1 r \le 1 .
\end{equation}
Combined with \eqref{eq:edge_pf1}, we get
\begin{equation} \label{eq:edge_pf2}
\sup_{\lambda \in [\half \lambda_c, \lambda_c ) }
\bignorm{ \eqref{eq:tau_alpha_decomp} }_{L^r(\Rd)}  < \infty
\quad \text{when} \quad
\frac{ 2 + \abs \alpha  } d 
< \frac 1 r \le 1 ,
\end{equation}
which, in particular, implies the local integrability of \eqref{eq:tau_alpha_decomp} and the differerentiability of $\hat \tau_\plus$.
The derivative $\grad^\alpha \hat \tau_\plus$ inherits the same estimates of $L^r(\Rd)$ norms as in \eqref{eq:edge_pf2}, and the $\abs \alpha = a$ cases provide what we need to conclude the proof. 
\end{proof}

\subsubsection{Triangle}

\begin{proof} [Proof of Lemma~\ref{lem:simple}(iii)]
Let $b = 2$. It is convenient to bound one of the $\tau_\lam$ by $\tilde \tau_\lam$, so we have
\begin{equation}
T \supb (u) \le \lambda^2 \Big( \big[ \abs x^b \varphi(x) \big] * \taul * \taul \Big)(u)
	+ \lambda^2 ( \tau_\plus \supb * \taul * \taul )(u) .
\end{equation}
We want to estimate its sup norm.
By the Cauchy--Schwarz inequality, the first term is bounded by $ \lam_c^2 \norm{ \abs x^b \varphi(x) }_1 \norm{ \tau_\lam }_2^2$, which is finite and uniform in $\lambda$ by \eqref{eq:adj_hyp} in Assumption~\ref{ass:adj} and Lemma~\ref{lem:base}.
For the second term, we use \eqref{eq:triangle_fourier} and estimate the $L^1$ norm of the integrand of the Fourier integral. 
It suffices to prove
\begin{equation} \label{eq:triangle_goal}
\sup_{\lambda \in [\half \lambda_c, \lambda_c ) }
\biggnorm{
\Big( \grad^\alpha \hat \tau_\plus  \Big)
\Big( \frac {\hat J_\lambda}{ 1 - \hat J_\lambda} \Big)^2
}_1 < \infty
\end{equation}
for every multi-index $\alpha$ with $\abs \alpha \le 2$.

Recall the estimate on $\grad^\alpha \hat \tau_\plus$ obtained in \eqref{eq:edge_pf2} and the infrared bound $[1 - \hat J_\lam(k)]\inv \lesssim \abs k^{-2} \vee 1$. On the ball $B = \{ \abs k \le 1 \}$, we have 
\begin{equation}
\biggnorm{
\Big( \grad^\alpha \hat \tau_\plus  \Big)
\Big( \frac {\hat J_\lambda}{ 1 - \hat J_\lambda} \Big)^2
}_{L^p(B)}
\lesssim \bignorm{ \grad^\alpha \hat \tau_\plus }_{L^r(B) }
	\biggnorm{ \frac 1 {\abs k^4 } }_{L^q(B)} \norm{ \hat J_\lambda }_\infty^2,
\end{equation}
where $p$ satisfies $p\inv = r\inv + q \inv$.
By \eqref{eq:edge_pf2} and \eqref{eq:Jhat_small}, this is finite and uniform in $\lambda$ when $r\inv > (2 + \abs \alpha) /d$ and $q \inv > 4/d$. We can take $p=1$ because $\abs \alpha \le 2 < d- 6$ (we assume $d > 8$).
On the complement $B^c = \{ \abs k > 1 \}$, we simply use the infrared bound to get
\begin{equation}
\biggnorm{
\Big( \grad^\alpha \hat \tau_\plus  \Big)
\Big( \frac {\hat J_\lambda}{ 1 - \hat J_\lambda} \Big)^2
}_{L^1(B^c)}
\lesssim \bignorm{ \grad^\alpha \hat \tau_\plus }_{L^1(B^c) }
	\norm{ \hat J_\lambda }_\infty^2 ,
\end{equation}
which is again finite and uniform in $\lambda$ by \eqref{eq:edge_pf2} and \eqref{eq:Jhat_small}.
Together, we obtain \eqref{eq:triangle_goal}, and the proof is complete.
\end{proof}

\begin{remark} [Square]
\label{remark:square}
For later use, we note that a slight modification of the above proof (set $b=0$) yields that, in dimensions $d > 8$,
\begin{equation}
\bar S := \sup_{\lambda \in [\half \lambda_c, \lambda_c ) }
	\lambda^3 \norm{ \tau_\lam * \tau_\lam * \tau_\lam * \tau_\lam}_\infty < \infty .
\end{equation}
\end{remark}

\subsubsection{Bubble}

\begin{proof} [Proof of Lemma~\ref{lem:simple}(ii)]
Let $a,b < d-2$. We want to get an estimate on $W_p \supab(u)$ that is uniform in $u$ and $\lambda$.
By H\"older's inequality, for any $u$ we have
\begin{equation}
W_p \supab (u)  
= \bignorm{  \tilde \tau_\lambda\supa (x) 
	\tilde \tau_\lambda\supb ( u + x) }_{ L^p_x }
\le \bignorm{  \tilde \tau_\lambda\supa}_s
	\bignorm{  \tilde \tau_\lambda\supb}_r ,
\end{equation}
where $p\inv = s\inv + r\inv$.
We use part~(i) to estimate the norms on the right-hand side.
By taking $s=r=\infty$, we get $W_\infty \supab (u) \le \bar E_\infty\supa \bar E_\infty \supb < \infty$, so we only need to work on the lower boundary of $p$. 

If both $a, b < \half d - 2$, we can take $s=r=2$ to get that $W_1 \supab (u) \le \bar E_2\supa \bar E_2 \supb < \infty$.
This gives both \eqref{eq:bubble_small} and \eqref{eq:bubble_large} for these values of $a,b$, since $p^*_\ab > 1$ when $a+b > 2$.
If one of $a$ and $b$, say $a$, is larger than $\half d - 2$, 
due to conditions on $s$ and $r$ imposed by part~(i), we can only estimate for $p$ that satisfies
\begin{equation} \label{eq:Lp_bubble_range}
\frac 1 p <  \frac{ d - 2 - a } d 
		+  \biggl( \half \wedge \frac{ d - 2 - b } d \bigg). 
\end{equation}
The upper limit $a + b < \frac 3 2 d - 4$ in \eqref{eq:bubble_large} is precisely to ensure that we get an estimate on the desired range of $p$.
Indeed, since $a+b < \frac 3 2 d -4$, the right-hand side of \eqref{eq:Lp_bubble_range} is strictly larger than $\half$, so we can take $p=2$ and obtain $W_2 \supab (u) \le \bar E_s\supa \bar E_r \supb < \infty$ for some admissible $s$ and $r$.
Also, using $d > 8$ and $b \ge 0 > 4 - \half d$, the right-hand side of \eqref{eq:Lp_bubble_range} satisfies 
\begin{equation} \label{eq: bubble_pf}
\frac{ d - 2 - a } d  +  \bigg( \half \wedge \frac{ d - 2 - b } d \bigg)
> \frac{ d - (a + b) + 2 } d = \frac 1 {p^*_\ab} .
\end{equation}
This implies the existence of $p = p_\ab < p^*_\ab$ for which \eqref{eq:Lp_bubble_range} holds, and we get $W_{p_\ab} \supab (u) \le \bar E_s\supa \bar E_r \supb < \infty$ for some other admissible $s$ and $r$.
This concludes the proof.
\end{proof}

\begin{remark} \label{remark:bubble}
In this proof, the condition $d > 8$ is only used to justify $b > 4 - \half d$ near \eqref{eq: bubble_pf}.
One can assume $d \ge 7$ and  $a,b > \half$ instead, to work in dimension $d = 7$ or $8$.
\end{remark}

\subsubsection{Martini}

\begin{proof}[Proof of Lemma~\ref{lem:simple}(iv)]
Let $a < d-2$ and $b = 0 \text{ or }2$.
It suffices to show both $\bar H_1 \supab$ and $\bar H_\infty \supab$ are finite.
For $\bar H_\infty \supab$, we use Lemma~\ref{lem:H_infty} to bound
\begin{equation}
\bar H_\infty \supab
\le \lambda_c^2 \norm{\tau\crit}_2^2 \bar E\supa \bar T\supk 0
	\bar W_1 \supk{b,0} ,
\end{equation}
which is finite by Lemma~\ref{lem:base} and parts (i) and~(ii) of Lemma~\ref{lem:simple}.
For $\bar H_1 \supab$, 
we prove a uniform estimate on $H_1 \supab(u,v)$.
In the definition \eqref{eq:def_Y}--\eqref{eq:def_H}, 
we bound $ \tau_\lam \supa (z_2 - z_1)$ by its sup norm, 
and then we write the integrals over $z_2$ and $x$ as convolutions. 
This gives
\begin{align}
H_1 \supab (u,v) 
&\le  \bignorm{  \tau_\lam \supa }_\infty
	\int_{\R^{3d}}
	(\tau_\lam \supb * \taul * \taul )(y - v - z_3)
\nl &\qquad\qquad\qquad\qquad
\times \tau_\lam (z_1)   \tau_\lam (z_3 - z_1)  
	\tau_\lam (y - z_3) \tau_\lam (y - u)
	\lambda^4 \dd z_1 \dd z_3 \dd y  \nl
& \le \lam^4 \bar E\supa 
	\bignorm{  \tlam \supb * \tlam * \tlam }_\infty
	(\tlam * \tlam * \tlam * \tlam)(u)  \nl
&\le \lam\inv \bar E\supa \bar T\supb \bar S
\end{align}
for all $u,v\in \Rd$.
Taking the supremum then gives
$\bar H_1 \supab \le (\half \lam_c)\inv \bar E\supa \bar T\supb \bar S$,
which is finite by previous parts of the lemma, Lemma~\ref{lem:base}, and Remark~\ref{remark:square}.
This concludes the proof.
\end{proof}

\section{Proof of Proposition~\ref{prop:diagram}}
\label{sec:diagram}

In this section, we prove Proposition~\ref{prop:diagram}.
The proof uses \emph{diagrammatic estimates} for the function $\Pi_\lam(x)$, which is defined (for $\lambda\leq\lambda_c$) by an absolutely convergent series
\begin{equation} \label{eq:def_Pi}
\Pi_\lambda(x) = \sum_{n=0}^\infty (-1)^n \Pi_\lambda\supn(x) ,
\end{equation}
where $\Pi_\lambda\supn(x)$ are non-negative functions \cite[Definition~3.7]{HHLM26}.
Note the superscript $(n)$ on $\Pi_\lambda\supn$ does \emph{not} mean multiplying $\Pi_\lam(x)$ by $\abs x^n$. 
We will estimate each $\Pi_\lambda \supn(x)$ separately.

In our bounds we use a quantity $V\crit$. Following the notation of \cite{DH22}, we define
\begin{equation}
	\triangle = \lambda_c^2\norm*{\tau\crit^{*3}}_\infty, \qquad \triangle^{\circ\circ} =  \lambda_c^2\norm*{\tau\crit^{*3}}_\infty +  \lambda_c\norm*{\tau\crit * \tau\crit}_\infty +1,
\end{equation}
and the composite quantities
\begin{equation} \label{eq:def_UV}
	U\crit= \triangle\triangle^{\circ\circ} + (\triangle^{\circ\circ})^2 + \lambda_c\norm*{\connf}_1,
	\qquad 
	V\crit = 4(\triangle\triangle^{\circ\circ}U\crit)^\frac{1}{2}.
\end{equation}
The conditions in Assumption~\ref{ass:adj} imply the conditions in \cite{DH22}, and it is proven there (Proposition~7.1) that $U\crit = O(1)$ and $V\crit\to 0$ as $d\to\infty$. In particular, for sufficiently large $d$ we have $V\crit<1$.

\begin{proposition} \label{prop:Pi_n}
Under the hypotheses of Proposition~\ref{prop:diagram}, 
there is a constant $C < \infty$ (depending only on the diagrams in the hypothesis) such that, for all $n \ge 0$,
\begin{equation}
\sup_{\lambda \in [\half \lambda_c, \lambda_c ) }
\bignorm{ \abs x^{\phi+\theta} \Pi_\lam \supn(x) }_{p} 
\le C (2n+1)^{2+\phi+\theta} \times 
\begin{cases}
1 					&(n=0,1,2,3) \\
V_{\lambda_c}^{n-4}	&(n\ge 4).
\end{cases}
\end{equation}
\end{proposition}

\begin{proof}[Proof of Proposition~\ref{prop:diagram}]
Using \eqref{eq:def_Pi} and Proposition~\ref{prop:Pi_n}, for any $\lambda \in [\half \lam_c, \lam_c)$ we have
\begin{equation}
\bignorm{ \abs x^{\phi+\theta} \Pi_\lam(x) }_{p} 
\le \sum_{n=0}^\infty \bignorm{ \abs x^{\phi+\theta} \Pi_\lam \supn(x) }_{p} 
\le O(1) + C \sum_{n=4}^\infty (2n+1)^{2+\phitheta} V\crit^{n-4} ,
\end{equation}
which is finite because the series converges when $V\crit < 1$.
\end{proof}

It remains to prove Proposition~\ref{prop:Pi_n}. 
In Section~\ref{subsec:diagram}, 
we recall the diagrammatic estimates and introduce our strategy to bound diagrams weighted by $\abs x^\phitheta$, using $\Pi_\lam \supk 0$ and (one diagram of) $\Pi_\lam \supk 1$ as examples.
In Section~\ref{subsec:general_n}, 
we use a more involved but systematic way to prove Proposition~\ref{prop:Pi_n} for all $n\ge 1$.

\subsection{Diagrammatic estimates and strategy}
\label{subsec:diagram}

Recall the function $\tau_\plus(x)$ defined in \eqref{eq:def_tilde_tau}.
The simplest case of the diagrammatic estimates is for $\Pi_\lam \supk 0$, which is given in \cite[(4.3)]{HHLM26} as
\begin{equation} \label{eq:Pi0_bound}
\Pi_\lambda\supzero (x)  \le \half \tau_\plus(x)^2  .
\end{equation}
In general, in the upper bound of $\Pi_\lam \supk n(x)$,
there are always two edge-disjoint ``paths'' connecting $0$ and~$x$.
For $n=0$, the two paths are just the two copies of $\tau_\plus(x)$ in \eqref{eq:Pi0_bound}.
Our strategy is to distribute $\abs x^\phi$ and $\abs x^\theta$ along the two paths, and then to reduce the resultant diagrams into the diagrams assumed finite in the hypotheses.

\begin{proof}[Proof of Proposition~\ref{prop:Pi_n} for $n=0$]
We use \eqref{eq:Pi0_bound}.
Since $\tau_\plus \le \tilde \tau_\lambda$ by definition, 
for any $\lambda \in [\half \lam_c, \lam_c)$ we have
\begin{equation}
\bignorm{ \abs x^\phitheta  \Pi_\lambda\supzero(x) }_p
\le \half \bignorm{ \abs x^\phi \tilde \tau_\lambda(x) 
	\abs x^\theta \tilde \tau_\lambda(x) }_p
= \half W_p \supphitheta(0)
\le \bar W_p \supphitheta ,
\end{equation}
which is assumed finite in the hypotheses of Proposition~\ref{prop:diagram}.
\end{proof}

Diagrammatic estimates for $\Pi_\lam \supn$, $n\ge 1$, are more involved to write, and we use a graphical notation as in Figure~\ref{fig:psi}. 
The basic objects in the graphical notation are vertices $\blackcirc, \circ$ and edges like $\tauLine, \adjLine, \tauCircleLine$.
An edge represents a function evaluated at the displacement between its two endpoints, \eg, 
$ 0 \tauLine x = \tau_\lambda(x)$ and $ 0 \adjLine x = \adj(x)$.
The vertex $\blackcirc$ indicates that that vertex is integrated over $\Rd$ \emph{with a vertex factor $\lambda$}, 
while the vertex $\circ$ is not integrated and does not carry a vertex factor. 

The bounds on $\Pi_\lambda\supn$ are given in terms of the following $\psi$ functions, whose graphical notation is given in Figure~\ref{fig:psi}. 
Like $\Pi_\lambda\supn$, the superscripts on the $\psi$ functions do not mean multiplying $\psi$ by a power of $\abs x$.

\begin{definition} \label{def:psi}
Let $r,s,u,w,x\in\Rd$. 
\begin{enumerate}
\item
Define $\tlamo(x) = 0 \tauCircleLine x = \lambda\inv\delta(x) + \tlam(x)$,
where $\delta(x)$ is the Dirac delta. 
This means that the edge in the graphical notation can be contracted.

\item
Define $\psi_0 = \psi_0^{(1)}+\psi_0^{(2)}+\psi_0^{(3)}$, where
\begin{equation*} \begin{aligned}
\psi_0^{(1)}(w,u;s) &= \tlam(u-s)\tlam(u-w)\tlam(w-s),	\\
\psi_0^{(2)}(w,u;s) &= \lambda\inv\delta(w-s)\tlam(u-s)\lambda(\tlam*\tlam)(u-s),\\
\psi_0^{(3)}(w,u;s) &= \lambda\inv \delta(w-s) \connf(u-s) .
\end{aligned} \end{equation*}

\item
Define $\psi = \psi^{(1)}+\psi^{(2)}+\psi^{(3)} + \psi^{(4)}$, where
\begin{equation*} \begin{aligned}
\psi^{(1)}(w,u;r,s) &= \tlam(w-u)\int_{\R^{2d}} \tlamo(t-s) \tlam(t-w)\tlam(u-z)\tlam(z-t)\tlam(z-r) \lambda^2 \dd z \dd t, \\
\psi^{(2)}(w,u;r,s) &= \tlamo(w-s)\int_{\R^{2d}} \tlam(t-z)\tlam(z-u)\tlam(u-t)\tlamo(t-w)\tlam(z-r) \lambda^2 \dd z \dd t, \\
\psi^{(3)}(w,u;r,s) &= \tlam(u-w)\tlam(w-s)\tlam(u-r),\\
\psi^{(4)}(w,u;r,s) &= \tlam(u-w) \lambda\inv\delta(w-s)\tlam(u-r) .
\end{aligned} \end{equation*}

\item
Define $\psi_n = \psi_n^{(1)} + \psi_n^{(2)}$, where
\begin{equation*} \begin{aligned}
\psi_n^{(1)} (x;r,s) &= \int_{\R^{2d}} \tlamo(t-s)\tlam(z-r)\tlam(t-z)\tlam(z-x)\tlam(x-t) \lambda^2 \dd z \dd t,\\
\psi_n^{(2)}(x;r,s) &= \tlam(x-s)\tlam(x-r).
\end{aligned} \end{equation*}
\end{enumerate}

\begin{figure} 
    \centering
    \begin{subfigure}[b]{0.3\textwidth}
    \centering 
        \begin{tikzpicture}[scale=0.6]
        \draw (0,0) -- (1,0.6) -- (1,-0.6) -- cycle;
        \filldraw[fill=white] (0,0) circle (2pt) node[left]{$s$};
        \filldraw[fill=white] (1,0.6) circle (2pt) node[above right]{$w$};
        \filldraw[fill=white] (1,-0.6) circle (2pt) node[below right]{$u$};
        \end{tikzpicture}
    \caption{$\psi^{(1)}_0$}
    \end{subfigure}
    \hfill
    \begin{subfigure}[b]{0.3\textwidth}
    \centering
        \begin{tikzpicture}[scale=0.6]
        \draw (0,0) -- (1,0.6) -- (1,-0.6) -- cycle;
        \filldraw (0,0) circle (2pt);
        \filldraw[fill=white] (1,0.6) circle (2pt)  node[above]{$s=w$};
        \filldraw[fill=white] (1,-0.6) circle (2pt) node[below right]{$u$};
       \end{tikzpicture}
    \caption{$\psi^{(2)}_0$}
    \end{subfigure}
    \hfill
    \begin{subfigure}[b]{0.3\textwidth}
    \centering
        \begin{tikzpicture}[scale=0.6]
        \filldraw (-0.4,0) circle (0pt);
        \draw (0,0.6) -- (0,-0.6);
        \draw (0,0) circle (0pt) node[rotate=90,above]{$\sim$};
        \filldraw[fill=white] (0,0.6) circle (2pt)  node[above]{$s=w$};
        \filldraw[fill=white] (0,-0.6) circle (2pt) node[below right]{$u$}; 
        \end{tikzpicture}
    \caption{$\psi^{(3)}_0$}
    \end{subfigure}
    \begin{subfigure}[b]{0.24\textwidth}
    \centering
        \begin{tikzpicture}[scale=0.6]
        \filldraw (-0.3,0) circle (0pt); 
        \draw (0,0.6) -- (1,0.6) -- (1,-0.6) -- (0,-0.6);
        \filldraw[fill=white] (0.5,0.8) circle (2pt);
        \draw (1,0.6) -- (2,0.6) -- (2,-0.6) -- (1,-0.6);
        \draw (0,1.75) circle (0pt);
        \filldraw[fill=white] (0,0.6) circle (2pt) node[above left]{$s$};
        \filldraw[fill=white] (0,-0.6) circle (2pt) node[below left]{$r$};
        \filldraw (1,0.6) circle (2pt);
        \filldraw (1,-0.6) circle (2pt);
        \filldraw[fill=white] (2,0.6) circle (2pt) node[above right]{$w$};
        \filldraw[fill=white] (2,-0.6) circle (2pt) node[below right]{$u$};
        \end{tikzpicture}
    \caption{$\psi^{(1)}$}
    \end{subfigure}
    \hfill
    \begin{subfigure}[b]{0.24\textwidth}
    \centering
        \begin{tikzpicture}[scale=0.6]
        \filldraw (-0.3,0) circle (0pt); 
        \draw (0,0.6) -- (1.5,0.6) -- (1.5,-0.1) -- (1,-0.6);
        \filldraw[fill=white] (0.75,0.8) circle (2pt);
        \filldraw[fill=white] (1.3,0.25) circle (2pt);
        \draw (0,-0.6) -- (1,-0.6) -- (2,-0.6) -- (1.5,-0.1);
        \filldraw[fill=white] (0,0.6) circle (2pt) node[above left]{$s$};
        \filldraw[fill=white] (0,-0.6) circle (2pt) node[below left]{$r$};
        \filldraw[fill=white] (1.5,0.6) circle (2pt) node[above right]{$w$};
        \filldraw (1,-0.6) circle (2pt);
        \filldraw (1.5,-0.1) circle (2pt);
        \filldraw[fill=white] (2,-0.6) circle (2pt) node[below right]{$u$};
        \end{tikzpicture}
    \caption{$\psi^{(2)}$}
    \end{subfigure}
    \hfill
    \begin{subfigure}[b]{0.24\textwidth}
    \centering
        \begin{tikzpicture}[scale=0.6]
        \filldraw (-0.3,0) circle (0pt); 
        \draw (0,0.6) -- (1,0.6) -- (1,-0.6) -- (0,-0.6);
        \filldraw[fill=white] (0,0.6) circle (2pt) node[above left]{$s$};
        \filldraw[fill=white] (0,-0.6) circle (2pt) node[below left]{$r$};
        \filldraw[fill=white] (1,0.6) circle (2pt) node[above right]{$w$};
        \filldraw[fill=white] (1,-0.6) circle (2pt) node[below right]{$u$};
        \end{tikzpicture}
    \caption{$\psi^{(3)}$}
    \end{subfigure}
    \hfill
    \begin{subfigure}[b]{0.24\textwidth}
    \centering
        \begin{tikzpicture}[scale=0.6]
        \filldraw (-0.3,0) circle (0pt); 
        \draw (1,0.6) -- (1,-0.6) -- (0,-0.6);
        \filldraw[fill=white] (0,-0.6) circle (2pt) node[below left]{$r$};
        \filldraw[fill=white] (1,0.6) circle (2pt) node[above]{$s=w$};
        \filldraw[fill=white] (1,-0.6) circle (2pt) node[below right]{$u$};
        \end{tikzpicture}
    \caption{$\psi^{(4)}$}
    \end{subfigure}
    \hspace*{\fill}%
    \begin{subfigure}[b]{0.45\textwidth}
    \centering
        \begin{tikzpicture}[scale=0.6]
        \filldraw (-0.3,0) circle (0pt); 
        \draw (0,0.6) -- (1,0.6) -- (1,-0.6);
        \filldraw[fill=white] (0.5,0.8) circle (2pt);
        \draw (1,0.6) -- (2,0) -- (1,-0.6) -- (0,-0.6);
        \filldraw[fill=white] (0,0.6) circle (2pt) node[above left]{$s$};
        \filldraw[fill=white] (0,-0.6) circle (2pt) node[below left]{$r$};
        \filldraw (1,0.6) circle (2pt);
        \filldraw (1,-0.6) circle (2pt);
        \filldraw[fill=white] (2,0) circle (2pt) node[right]{$x$};
        \end{tikzpicture}
    \caption{$\psi^{(1)}_n$}
    \end{subfigure}
    \hfill
    \begin{subfigure}[b]{0.45\textwidth}
    \centering
        \begin{tikzpicture}[scale=0.6]
        \filldraw (0.7,0) circle (0pt); 
        \draw (1,0.6) -- (2,0) -- (1,-0.6);
        \filldraw[fill=white] (1,0.6) circle (2pt) node[above left]{$s$};
        \filldraw[fill=white] (1,-0.6) circle (2pt) node[below left]{$r$};
        \filldraw[fill=white] (2,0) circle (2pt) node[right]{$x$};
        \end{tikzpicture}
    \caption{$\psi^{(2)}_n$}
    \end{subfigure}
    \hspace*{\fill}
    \caption{Graphical notation for the $\psi$ functions.
    Each $\blackcirc$ has a vertex factor $\lambda$.     }
    \label{fig:psi}
\end{figure} 
\end{definition}

\begin{lemma}[Diagrammatic estimate]
\label{lem:Pi_bound}
Let $n \geq 1$, $x\in\Rd$, and $\lambda\in [0,\lambda_c)$. Then
\begin{equation} \label{eq:Pi_bound}
\Lacelam^{(n)}(x) 
\leq  \int_{\R^{n2d}}  \psi_n(x;w_{n-1},u_{n-1}) 
	\Biggl( \prod_{j=1}^{n-1} \psi(w_j,u_j;w_{j-1},u_{j-1}) \Biggr)
	\psi_0(w_0,u_0;\orig) 	\lambda^{2n} \dd\vec w \dd\vec u ,
\end{equation}
where $\vec w = (w_0, \dots, w_{n-1})$ and
$\vec u = (u_0, \dots, u_{n-1})$.
\end{lemma}

\begin{proof}
This is \cite[Proposition~7.2]{HHLM26},
although we have defined the $\psi$ functions slightly differently.
Compared to the $\psi$ functions in \cite{HHLM26},
our $\psi_0$ contains an extra factor of $\lambda\inv$,
our $\psi$ contains an extra factor of $\lambda\inv$ and has its interior vertices integrated,
and our $\psi_n$ has its interior vertices integrated.
This explains the power $\lambda^{2n}$ in \eqref{eq:Pi_bound}, comparing to the power $\lambda^n$ in \cite[Proposition~7.2]{HHLM26}.
\end{proof}

The $n=1$ case of Lemma~\ref{lem:Pi_bound} can be represented graphically, with the unlabeled vertex $\circ$ being~$0$, as
\begin{multline} \label{eq:Pi1_bound}
\Pi_\lambda\supk1(x) \le 
 \PiOneOneOne + \PiOneOneTwo  + \PiOneTwoOne  
+ \PiOneTwoTwo \\+ \PiOneThreeOne + \PiOneThreeTwo .
\end{multline}
Observe that two edge-disjoint paths from $0$ and $x$ always exist, namely, on the two sides of the diagram.
We explain the proof strategy using the first of the six diagrams. 
As for $\Pi_\lam \supk 0$,
we split $\abs x^\phi$ and $\abs x^\theta$ across two sides, using the following lemma. 

\begin{lemma}[Splitting of power] \label{lem:split_power}
Let $a \ge 0$, $N\in \N$, and $x_0, \dots, x_N \in \Rd$. 
Then
\begin{equation} \label{eq:power_split}
\abs{ x_N - x_0 }^a 
\le N^{a} \sum_{j=1}^N \abs{ x_j - x_{j-1} }^a.
\end{equation}
\end{lemma}

\begin{proof}
Since $x_N - x_0 = \sum_{j=1}^N (x_j - x_{j-1})$, 
by the triangle inequality, 
there must be an index $j$ for which 
\begin{equation}
\abs{ x_j - x_{j-1} } \ge \frac 1 N \abs{ x_N - x_0 }.
\end{equation}
Raising to the power $a$ and then summing over all $j$ gives the desired inequality.
\end{proof}

We also need to use a few more diagrams that are finite under the hypothesis of Proposition~\ref{prop:diagram} by the following lemma.
Note that $\bar B_1 \supk{\theta,0} \le \lambda_c \bar W_1 \supk{\theta,0}$.

\begin{lemma} \label{lem:extra}
Let $\phi, \theta \ge 0$ and $p \in [1,\infty]$.
There are (absolute) constants such that
\begin{align}
\bar W_p \supk{\phi,0}
	&\lesssim \bar W_p \supk{\phi,\theta} + \bar E \supphi ,\\
\bar B_1 \supk{\theta,0} 
	:= \sup_{\lambda \in [\half \lambda_c, \lambda_c ) }
		\lambda \norm{ \tau_\lam\suptheta * \tau_\lam }_\infty 
	&\lesssim \bar E\suptheta + \bar T\suptheta , \\
\label{eq:p-triangle}
\sup_{\lambda \in [\half \lambda_c, \lambda_c ) }
	\Bignorm{ \bignorm{  \tau_\lambda \suptheta (x) 
		\lambda( \tau_\lam * \tlam) ( u + x) }_{ L^p_x } }_{L^\infty_u}
	&\lesssim \bar T\suptheta + \bar E\suptheta \bar B_1\supk{0,0}   .
\end{align}
\end{lemma}

It will also be convenient to write
\begin{equation} \label{eq:contracted_triangle}
\bar T \supthetao 
= \bar T \suptheta + \bar B_1\supk{\theta,0} + \bar B_1\supk{0,0} ,
\end{equation}
which bounds an $\abs x^\theta$-decorated triangle with one possibly contracted edge.
Note $\bar B_1\supk{0,0} = \lambda_c \norm{ \tau\crit }_2^2$ is finite by Lemma~\ref{lem:base}.

\begin{proof}[Proof of Lemma~\ref{lem:extra}]
We start from $\bar W_p \supk{\phi,0}$.
In the definition of $W_p \supk{\phi,0}$ in \eqref{eq:def_Wp}, 
we bound $\tilde \tau_\lam(y) \le \abs y^\theta \tilde \tau_\lam(y)$ if $\abs y \ge 1$, and we bound $\tilde \tau_\lam(y) \le \norm \adj_\infty + \lambda \norm \adj_1 \norm{ \tau_\lam }_\infty \le 1 + \lambda$ otherwise.
This gives
\begin{equation}
W_p \supk{\phi, 0}(u)
\le  W_p \supphitheta(u) 
	+ \bar E \supphi (1+\lam) \norm{ \1_{\abs {u+x} < 1} }_{L^p_x}
\le \bar W_p \supphitheta + c_{d,p} (1+\lam_c) \bar E \supphi ,
\end{equation} 
where $c_{d,p}=1$ if $p=\infty$ and 
$c_{d,p} =  [ \pi^{d/2} / \Gamma( \frac{d+2} 2 )]^{1/p} $ if $p < \infty$. 
We simply bound $c_{d,p} \le 6$ for all $d\ge 1$ and $p\in [1,\infty]$.
Taking the supremum over $u \in \Rd$ then gives the desired result.

For $\bar B_1 \supk{\theta,0}$, 
we use $\tau_\lam \le \tilde \tau_\lam = \adj + \lambda \adj * \tau_\lam
\le  \adj + \lambda \tau_\lam * \tau_\lam$ (since by definition $\adj\leq \tau_\lam$)  to get
\begin{equation}
\lambda \norm{ \tau_\lam\suptheta * \tau_\lam }_\infty
\le \lam \bignorm{ \tau_\lam\suptheta }_\infty \norm \adj_1
	+ \lambda^2 \norm{ \tau_\lam\suptheta * \tau_\lam * \tau_\lam }_\infty
\le \lam_c \bar E\suptheta + \bar T\suptheta ,
\end{equation}
which gives the desired result.

For the last diagram, we use the log-convexity of $L^p$ norms to bound $\norm\cdot_{L^p_x} \le \norm\cdot_{L^1_x} + \norm\cdot_{L^\infty_x}$:
\begin{align}
\Bignorm{ \bignorm{  \tau_\lambda \suptheta (x) 
	\lambda( \tau_\lam * \tlam) ( u + x) }_{ L^p_x } }_{L^\infty_u}
&\le \lam\inv \bignorm{ T \suptheta(u) }_{L^\infty_u}
	+ \lam \bignorm{ \tau_\lam\suptheta }_\infty\norm{ \tlam*\tlam}_\infty\nl
&\le \lam\inv \bar T\suptheta + \bar E\suptheta \bar B_1 \supk{0,0}.
\end{align}
The desired result then follows from $\lambda\inv \le (\half \lambda_c)\inv$.
\end{proof}

\begin{proposition} \label{prop:Pi1}
Under the hypotheses of Proposition~\ref{prop:diagram}, 
there is a constant $C < \infty$ (depending only on the diagrams in the hypothesis) such that
\begin{equation}
\sup_{\lambda \in [\half \lambda_c, \lambda_c ) }
\biggnorm{ \abs x^\phitheta \times \PiOneOneOne }_{p} 
\le C 3^\phitheta .
\end{equation}
\end{proposition}

\begin{proof}
We use graphical notation. 
We use a thick edge $0 \weightedLine x$ to denote an $\abs x^a$-decorated edge ($a$-edge for short) $\tau_\lambda\supa(x) = \abs x^a \tau_\lambda(x)$, and we write $0 \weightedCircleLine x = \lambda\inv \delta(x) + 0 \weightedLine x$. 
We use Lemma~\ref{lem:split_power} to split $\abs x^\phi$ along the top side of a diagram, and to split $\abs x^\theta$ along the bottom side. 
As $\phi$ and $\theta$ have asymmetric roles, 
we distinguish then by using red colour for $\phi$-edge and blue colour for $\theta$-edge. This gives
\begin{multline} \label{eq:Pi1_pf0}
\abs x^\phitheta \times \PiOneOneOne \\
\le
3^\phitheta \(
\PiOneWeightedOne + \PiOneWeightedTwo + \PiOneWeightedThree
+ \text{(6 more diagrams)} \) ,
\end{multline}
where 3 of the 6 other diagrams have the $\theta$-edge at the bottom middle, and the other 3 have the $\theta$-edge at the bottom right. 
(Note there are $3^2 = 9$ diagrams in total.)
We bound the $L^p$ norm of each diagram, 
by separating into two cases according to whether the opposite edge of the $\phi$-edge is contracted. 

\smallskip \noindent 
\emph{Case 1: The opposite edge of the $\phi$-edge is not contracted.}
In this case, we use Minkowski's integral inequality to decompose the diagram into a left part, a right part, and an $L^p$ bubble that can be bounded by either $\bar W_p \supphitheta$ or $\bar W_p \supk {\phi,0}$. 
We then bound the left and right parts using triangles $\bar T\suptheta, \bar T\supk 0$ and the usual bubbles $\bar B_1\supk {\theta,0}, \bar B_1\supk {0,0}$.
For example, by Minkowski's inequality,
\begin{equation} 
\biggnorm{ \PiOneCaseOne }_{L^p_x}
\le \int_{\R^{2d}} \Triangleab  \biggnorm{ \PiOneCaseOneRight }_{L^p_x} \lambda^2 \dd a \dd b.
\end{equation}
And by translation invariance, 
\begin{equation}
\biggnorm{ \PiOneCaseOneRight }_{L^p_x}
= \biggnorm{ \PiOneCaseOneRightShift }_{L^p_x}
\le \int_{\R^{2d}} \biggnorm{ \WeightedBubble }_{L^p_x} \TriangleRightcd \lambda^2 \dd c \dd d.
\end{equation}
We thus have
\begin{equation}
\biggnorm{ \PiOneCaseOne }_{L^p_x}
\le \bar T \supk 0 \bar W_p \supphitheta \bar T \supk 0. 
\end{equation}
Another example is
\begin{align}
\biggnorm{ \PiOneCaseOneSecond }_{L^p_x}
&\le \int_{\R^{2d}} \PiOneCaseOneSecondLeft \biggnorm{ \PiOneCaseOneSecondRight }_{L^p_x} \lambda^2 \dd c \dd d \nl
&\le \bar W_p \supk{\phi,0} \int_{\R^{2d}} \Triangleab  \biggnorm{ \PiOneCaseOneSecondMid }_{L^\infty_{a,b}} \lambda^2 \dd a \dd b 
\le \bar T \supk 0 \bar T \supthetao \bar W_p \supk {\phi,0} ,
\end{align}
where $\bar T \supthetao$ was defined in \eqref{eq:contracted_triangle}

\smallskip \noindent 
\emph{Case 2: The opposite edge of the $\phi$-edge is contracted.}
In this case, we first bound the $\phi$-edge by its $L^\infty$ norm using $\bar E\supphi$, and then break the remaining diagram into a left and a right part using Young's convolution inequality, with $L^p$ norm on the left part and $L^1$ norm on the right part (this breakdown is quite arbitrary).
The left and right parts can then be bounded via usual methods, with the additional input of \eqref{eq:p-triangle} for the left.
For example, we have
\begin{align}
\biggnorm{ \PiOneCaseTwo }_{L^p_x}
&\le \lambda \biggnorm{ \sideTriangleBlue {}{x} }_{L^p_x} 
	\norm{ \weightedLineRed }_\infty 
	\biggnorm{ \sideTriangle{}{x} }_{L^1_x} \nl
&\le C (\bar T\suptheta + \bar E\suptheta \bar B_1\supk{0,0}) \bar E\supphi  \bar T \supk0 
\end{align}
for some absolute constant $C > 0$.

\smallskip \noindent 
\emph{Summary.}
Using the two cases, we can bound all diagrams in \eqref{eq:Pi1_pf0}, to get
\begin{align} \label{eq:Pi1_summary}
\biggnorm{ \abs x^\phitheta \times \PiOneOneOne }_{L^p_x}
\le 3^\phitheta \Big\{
	& 2\big( \bar W_p\supphitheta \bar T \supzeroo \bar T\supzero
	    + \bar W_p\supphizero \bar T \supthetao \bar T\supzero
	    + \bar W_p\supphizero \bar T \supzeroo \bar T\suptheta \big)  \nl
	& + \big( \bar T \suptheta \bar W_p \supphizero \bar T \supzero
	   + \bar T \supzero \bar W_p \supphitheta \bar T \supzero
	   + \bar T \supzero \bar W_p \supphizero \bar T \suptheta	\big) \nl
	& + 2 C (\bar T\suptheta + \bar E\suptheta \bar B_1\supk{0,0}) \bar E\supphi  \bar T \supk0 \Big\} .
\end{align}
This gives the desired result, since all appeared diagrams are finite under the hypotheses of Proposition~\ref{prop:diagram}.
\end{proof}

\subsection{The general bound for $n\ge 1$}
\label{subsec:general_n}

In addition to the $\psi$ functions introduced in Definition~\ref{def:psi}, we use $\bar \psi$ functions which explore the bound \eqref{eq:Pi_bound} ``from the right'' rather than ``from the left.'' 
The graphical notation for these functions is given in Figure~\ref{fig:EndBlockFunctions}.

\begin{definition}
Let $r,s,u,w \in \Rd$.
We define $\EndBlock_0 = \EndBlock^{(1)}_0+\EndBlock^{(2)}_0$, where
\begin{equation*}
\begin{aligned}
    \EndBlock^{(1)}_0(w,u;s) &= \tlam(w-s)\tlam(w-u)\tlam(u-s),\\
    \EndBlock^{(2)}_0(w,u;s) &= \lambda^{-2}\delta(w-s)\delta(u-s) .
\end{aligned}
\end{equation*}
We also define $\EndBlock = \EndBlock^{(1)}+\EndBlock^{(2)}+\EndBlock^{(3)} + \EndBlock^{(4)}$, where
\begin{align*}
    \EndBlock^{(1)}(w,u;r,s) &= \tlam(w-u)\int_{\R^{2d}} \tlam(t-s) \tlam(t-w)\tlam(u-z)\tlam(z-t)\tlamo(z-r) \lam^2 \dd z \dd t, \\
    \EndBlock^{(2)}(w,u;r,s) &= \tlam(w-s)\int_{\R^{2d}} \tlam(t-z)\tlam(z-u)\tlam(u-t)\tlamo(t-w)\tlamo(z-r) \lam^2 \dd z \dd t, \\
	\EndBlock^{(3)}(w,u;r,s) &= \tlam(u-w)\tlam(w-s)\tlam(u-r),\\
	\EndBlock^{(4)}(w,u;r,s) &= \tlam(u-w)\tlam(w-s) \lam\inv \delta(u-r).
\end{align*}

\begin{figure}
    \centering
    \begin{subfigure}[b]{0.45\textwidth}
    \centering
        \begin{tikzpicture}[scale=0.6]
        \draw (0,0) -- (-1,0.6) -- (-1,-0.6) -- cycle;
        \filldraw[fill=white] (0,0) circle (2pt) node[right]{$s$};
        \filldraw[fill=white] (-1,0.6) circle (2pt) node[above left]{$w$};
        \filldraw[fill=white] (-1,-0.6) circle (2pt) node[below left]{$u$};
        \end{tikzpicture}
    \caption{$\EndBlock^{(1)}_0$}
    \end{subfigure}
    \hfill
    \begin{subfigure}[b]{0.45\textwidth}
    \centering
        \begin{tikzpicture}[scale=0.6]
        \filldraw[fill=white] (0,0) circle (2pt) node[above]{$s=w=u$};
        \draw[white] (0,-1) circle (0pt);
       \end{tikzpicture}
    \caption{$\EndBlock^{(2)}_0$}
    \end{subfigure}
    \begin{subfigure}[b]{0.24\textwidth}
    \centering
        \begin{tikzpicture}[scale=0.6]
        \draw (0,0.6) -- (-1,0.6) -- (-1,-0.6) -- (0,-0.6);
        \filldraw[fill=white] (-0.5,-0.4) circle (2pt);
        \draw (-1,0.6) -- (-2,0.6) -- (-2,-0.6) -- (-1,-0.6);
        \draw (0,1.75) circle (0pt);
        \filldraw[fill=white] (0,0.6) circle (2pt) node[above right]{$s$};
        \filldraw[fill=white] (0,-0.6) circle (2pt) node[below right]{$r$};
        \filldraw (-1,0.6) circle (2pt);
        \filldraw (-1,-0.6) circle (2pt);
        \filldraw[fill=white] (-2,0.6) circle (2pt) node[above left]{$w$};
        \filldraw[fill=white] (-2,-0.6) circle (2pt) node[below left]{$u$};
        \end{tikzpicture}
    \caption{$\EndBlock^{(1)}$}
    \end{subfigure}
    \hfill
    \begin{subfigure}[b]{0.24\textwidth}
    \centering
        \begin{tikzpicture}[scale=0.6]
        \draw (0,0.6) -- (-1.5,0.6) -- (-1.5,-0.1) -- (-1,-0.6);
        \filldraw[fill=white] (-0.5,-0.4) circle (2pt);
        \filldraw[fill=white] (-1.3,0.25) circle (2pt);
        \draw (0,-0.6) -- (-1,-0.6) -- (-2,-0.6) -- (-1.5,-0.1);
        \filldraw[fill=white] (0,0.6) circle (2pt) node[above right]{$s$};
        \filldraw[fill=white] (0,-0.6) circle (2pt) node[below right]{$r$};
        \filldraw[fill=white] (-1.5,0.6) circle (2pt) node[above left]{$w$};
        \filldraw (-1,-0.6) circle (2pt);
        \filldraw (-1.5,-0.1) circle (2pt);
        \filldraw[fill=white] (-2,-0.6) circle (2pt) node[below left]{$u$};
        \end{tikzpicture}
    \caption{$\EndBlock^{(2)}$}
    \end{subfigure}
    \hfill
    \begin{subfigure}[b]{0.24\textwidth}
    \centering
        \begin{tikzpicture}[scale=0.6]
        \draw (0,0.6) -- (-1,0.6) -- (-1,-0.6) -- (0,-0.6);
        \filldraw[fill=white] (0,0.6) circle (2pt) node[above right]{$s$};
        \filldraw[fill=white] (0,-0.6) circle (2pt) node[below right]{$r$};
        \filldraw[fill=white] (-1,0.6) circle (2pt) node[above left]{$w$};
        \filldraw[fill=white] (-1,-0.6) circle (2pt) node[below left]{$u$};
        \end{tikzpicture}
    \caption{$\EndBlock^{(3)}$}
    \end{subfigure}
    \hfill
    \begin{subfigure}[b]{0.24\textwidth}
    \centering
        \begin{tikzpicture}[scale=0.6]
        \draw (0,0.6) -- (-1,0.6) -- (-1,-0.6);
        \filldraw[fill=white] (0,0.6) circle (2pt) node[above right]{$s$};
        \filldraw[fill=white] (-1,0.6) circle (2pt) node[above left]{$w$};
        \filldraw[fill=white] (-1,-0.6) circle (2pt) node[below]{$r=u$};
        \end{tikzpicture}
    \caption{$\EndBlock^{(4)}$}
    \end{subfigure}
    \hspace*{\fill}
    \caption{Diagrams of the $\EndBlock_0$, and $\EndBlock$ functions.}
    \label{fig:EndBlockFunctions}
\end{figure}
\end{definition}

It is convenient to use diagrams like 
\begin{equation}
    \LeftBlock{0}{a}{b}{m}
\end{equation}
to denote a composition of $\psi$ functions,
where each composition is a ``convolution'' across $\R^{2d}$.
For example, for any $m\ge1$, the diagram above is defined to be
\begin{equation}
\LeftBlock{0}{a}{b}{m} 
= \int_{\R^{m2d}} \Biggl(\prod^m_{j=1}\psi(w_j,u_j;w_{j-1},u_{j-1})\Biggr)
	\psi_0(w_0,u_0;0)\lambda^{2m}\dd \vec{w}\dd\vec{u},
\end{equation}
where $w_m=a$, $u_m=b$, $\vec{w}=(w_0,\ldots,w_{m-1})$, and $\vec{u}=(u_0,\ldots,u_{m-1})$. 
For $m=0$, we define this diagram to be $\psi_0(a,b;0)$. 
Similarly, for $m\geq 1$ we define
\begin{equation}
\RightBlockBar{a}{b}{c}{m} 
=  \int_{\R^{m2d}} \Biggl(\prod^m_{j=1}\bar\psi(w_j,u_j;w_{j-1},u_{j-1})\Biggr)
	\bar\psi_0(w_0,u_0;c)\lambda^{2m}\dd \vec{w}\dd\vec{u},
\end{equation}
where $w_m=a$, $u_m=b$, $\vec{w}=(w_0,\ldots,w_{m-1})$, and $\vec{u}=(u_0,\ldots,u_{m-1})$; while for $m=0$ we define this diagram to be $\bar\psi_0(a,b;c)$.

The following lemma demonstrates how we use the $\EndBlock$ functions.
When we break up the upper bound \eqref{eq:Pi_bound} for $\Lacelam^{(n)}(x)$, the ``right'' ($j\ge m$) part of the integral can be bounded using the $\EndBlock$ composition diagram above.

\begin{lemma}
Let $n \ge m \ge 1$, and let $x,w_{m-1},u_{m-1}\in\R^d$. Then
\begin{multline}
\int_{\R^{(n-m)2d}} \psi_n(x;w_{n-1},u_{n-1}) 	
	\Biggl( \prod_{j=m}^{n-1} \psi(w_j,u_j;w_{j-1},u_{j-1}) \Biggr)	
	\lambda^{2(n-m)} \dd \vec w_{[m,n)} \dd \vec u_{[m,n)}
 \\
\le  \int_{\R^{2d}} 
	\tau_\lam(w_{m-1}-a)\tau^\circ_\lam(u_{m-1}-b) 	
	\RightBlockBar{b}{a}{x}{n-m} 
	\lambda^2 \dd a \dd b ,
\end{multline}
where $\vec w_{[m,n)} = (w_m, \dots , w_{n-1} )$
and $\vec u_{[m,n)} = (u_m, \dots , u_{n-1} )$.
\end{lemma}

\begin{proof}[Proof sketch]
	This is mostly just a rearrangement of terms, as previously used in \cite[Proof of Proposition~4.10]{HHLM26}. The explicit $\tau_\lambda$ and $\tau^\circ_\lam$ in the bound are extracted first. Then each $\psi$ factor ``borrows'' a $\tau_\lambda$ and $\tau^\circ_\lam$ from one neighbour and ``lends'' them to the other, turning themselves into a $\EndBlock$ factor. The inequality arises solely because some $\tau_\lam$ terms are replaced with $\tau^\circ_\lam$ terms to get the $\EndBlock$ factors.
	
The following diagram illustrates an instance of how this rearrangement of terms can work. Here we look at the case where $m=n-1$, $\psi_n$ factor is a $\psi_n^{(1)}$ factor, and the single $\psi$ factor is a $\psi^{(1)}$ factor.
The rearrangement
\begin{align}
&\int_{\R^{2d}} \psi^{(1)}_n(x;w_{n-1},u_{n-1}) 	\psi^{(1)}(w_{n-1},u_{n-1};w_{n-2},u_{n-2}) \lambda^{2} \dd w_{n-1} \dd u_{n-1}\nonumber\\
 &\hspace{2cm}=
\int_{\R^{2d}}\Biggl(
\raisebox{-26pt}{
		\begin{tikzpicture}[scale=0.8]
			\draw (0,0.6) -- (1,0.6) -- (1,-0.6) -- (0,-0.6);
			\draw (1,0.6) -- (2,0.6) -- (2,-0.6) -- (1,-0.6);
			\draw (2,0.6) -- (3,0.6) -- (3,-0.6) -- (2,-0.6);
			\draw (3,0.6) -- (4,0) -- (3,-0.6);
			\filldraw[fill=white] (0.5,0.4) circle (2pt);
			\filldraw[fill=white] (2.5,-0.4) circle (2pt);
			\filldraw[fill=white] (0,0.6) circle (2pt) node[left]{$u_{n-2}$};
			\filldraw[fill=white] (0,-0.6) circle (2pt) node[left]{$w_{n-2}$};
			\filldraw (1,0.6) circle (2pt);
			\filldraw (1,-0.6) circle (2pt);
			\filldraw[fill=white] (2,0.6) circle (2pt) node[above]{$w_{n-1}$};
			\filldraw[fill=white] (2,-0.6) circle (2pt) node[below]{$u_{n-1}$};
			\filldraw (3,0.6) circle (2pt);
			\filldraw (3,-0.6) circle (2pt);
			\filldraw[fill=white] (4,0) circle (2pt) node[right]{$x$};
		\end{tikzpicture}}\Biggl)
		\lambda^2 \dd w_{n-1}\dd u_{n-1}
		\nonumber\\ 
		&\hspace{2cm}= 
\int_{\R^{4d}}\Biggl(
\raisebox{-27pt}{
		\begin{tikzpicture}[scale=0.8]
			\draw (0,0.6) -- (1,0.6) -- (1,-0.6) -- (0,-0.6);
			\draw (1,0.6) -- (2,0.6) -- (2,-0.6) -- (1,-0.6);
			\draw (2,0.6) -- (3,0.6) -- (3,-0.6) -- (2,-0.6);
			\draw (3,0.6) -- (4,0) -- (3,-0.6);
			\filldraw[fill=white] (0.5,0.4) circle (2pt);
			\filldraw[fill=white] (2.5,-0.4) circle (2pt);
			\filldraw[fill=white] (0,0.6) circle (2pt) node[left]{$u_{n-2}$};
			\filldraw[fill=white] (0,-0.6) circle (2pt) node[left]{$w_{n-2}$};
			\filldraw[fill=white] (1,0.6) circle (2pt) node[above]{$b_{n-2}$};
			\filldraw[fill=white] (1,-0.6) circle (2pt) node[below]{$a_{n-2}$};
			\filldraw (2,0.6) circle (2pt);
			\filldraw (2,-0.6) circle (2pt);
			\filldraw[fill=white] (3,0.6) circle (2pt) node[above]{$a_{n-1}$};
			\filldraw[fill=white] (3,-0.6) circle (2pt) node[below]{$b_{n-1}$};
			\filldraw[fill=white] (4,0) circle (2pt) node[right]{$x$};
		\end{tikzpicture}}\Biggl)
		\lambda^4 \dd a_{n-1}\dd a_{n-2}\dd b_{n-1}\dd b_{n-2}\nonumber\\
		&\hspace{2cm} = \int_{\R^{4d}}\tau_\lam(w_{m-1}-a)\tau^\circ_\lam(u_{m-1}-b) \EndBlock^{(1)}\left(b_{n-2},a_{n-2};b_{n-1},a_{n-1}\right)\nonumber\\
		&\hspace{7cm}\times\EndBlock^{(1)}_0(b_{n-1},a_{n-1};x) \lambda^4 \dd a_{n-1}\dd a_{n-2}\dd b_{n-1}\dd a_{n-2},
\end{align}
produces $\EndBlock^{(1)}$ and $\EndBlock_0^{(1)}$ factors, as desired.
	
	To see an instance that is an inequality rather than actually an equality, consider $m=n$ with the $\psi_n$ factor being a $\psi_n^{(2)}$ factor. In this case, we have
	\begin{align}
		\psi^{(2)}_n\left(x;w_{n-1},u_{n-1}\right) &= 
		\raisebox{-19pt}{
		\begin{tikzpicture}[scale=0.8]
			\draw (0,0.6) -- (1,0) -- (0,-0.6);
			\filldraw[fill=white] (0,0.6) circle (2pt) node[left]{$w_{n-1}$};
			\filldraw[fill=white] (0,-0.6) circle (2pt) node[left]{$u_{n-1}$};
			\filldraw[fill=white] (1,0) circle (2pt) node[right]{$x$};
		\end{tikzpicture}}
		\leq 		
		\raisebox{-19pt}{
		\begin{tikzpicture}[scale=0.8]
			\draw (0,0.6) -- (1,0) -- (0,-0.6);
			\filldraw[fill=white] (0.45,-0.15) circle (2pt);
			\filldraw[fill=white] (0,0.6) circle (2pt) node[left]{$w_{n-1}$};
			\filldraw[fill=white] (0,-0.6) circle (2pt) node[left]{$u_{n-1}$};
			\filldraw[fill=white] (1,0) circle (2pt) node[right]{$x$};
		\end{tikzpicture}}\nonumber\\
		 &= \int_{\R^{2d}}\Biggl(
		 \raisebox{-20pt}{
		\begin{tikzpicture}[scale=0.8]
			\draw (0,0.6) -- (1,0.6);
			\draw (1,-0.6) -- (0,-0.6);
			\filldraw[fill=white] (0.5,-0.4) circle (2pt);
			\filldraw[fill=white] (0,0.6) circle (2pt) node[left]{$w_{n-1}$};
			\filldraw[fill=white] (0,-0.6) circle (2pt) node[left]{$u_{n-1}$};
			\filldraw[fill=white] (1,0.6) circle (2pt) node[right]{$a$};
			\filldraw[fill=white] (1,-0.6) circle (2pt) node[right]{$b$};
		\end{tikzpicture}}\Biggl)\lambda^{-2}\delta(x-a)\delta(x-b)\lambda^2\dd a\dd b\nonumber\\
		& = \int_{\R^{2d}} \tau_\lam(w_{n-1}-a)\tau^\circ_\lam(u_{n-1}-b)  \EndBlock_0^{(2)}(b,a;x)
\lambda^{2} \dd a \dd b,
	\end{align}
	as desired.
\end{proof}

The following lemma bounds the composition diagrams.
Recall the definition of $U\crit, V\crit$ from \eqref{eq:def_UV}.

\begin{lemma} \label{lem:left_right}
For any integer $m \ge 0$ and any $\lambda \in [0, \lam_c)$,
we have
\begin{align}
\label{eq:left}
\int_{\R^{2d}} \LeftBlock{\orig}{a}{b}{m} \lambda^2 \dd a\dd b 
&\leq 3 U\crit V\crit^{m}, 
\\
\label{eq:right}
\int_{\R^{2d}} \RightBlockBar{a}{b}{\orig}{m} \lambda^2 \dd a\dd b 
&\leq 2 U\crit V\crit^{m} .
\end{align}
\end{lemma}

\begin{proof}
These bounds follow from \cite[Lemma~5.7]{DH22}. Note that the definition of $V\crit$ used here differs from the $V\crit$ used there in that we have included the factor of $4$ in $V\crit$.
\end{proof}

The proof of Proposition~\ref{prop:Pi_n} for $n\ge 1$ follows the same strategy as the proof of Proposition~\ref{prop:Pi1}. 
Briefly, we split $\abs x^\phitheta$ across the two sides of the diagram for $\Pi_\lam\supn(x)$, and we bound the $L^p_x$ norm of the resultant diagrams by extracting an $L^p$ piece in the middle, and by handling the left and right parts using Lemma~\ref{lem:left_right}. 
We prove the middle $L^p$ piece is finite (in Appendix~\ref{sec:DiagramBounds}), and we get powers of $V\crit$ from the left and right parts.

\begin{proof}[Proof of Proposition~\ref{prop:Pi_n} for $n\ge 1$]
We use the diagrammatic estimate \eqref{eq:Pi_bound},
and we use Lemma~\ref{lem:split_power} to split $\abs x^\phi$ (the red $\phi$ decoration) across the ``top'' of the diagram (from the perspective of the orientation of the $\psi_0$ diagrams) and split $\abs x^\theta$ (the blue $\theta$ decoration) across the other side of the diagram.
Since the side length of the diagram for $\Pi_\lam \supn$ is at most $2n+1$, 
the splitting creates at most $(2n+1)^2$ diagrams (from the sum in \eqref{eq:power_split}), each of which multiplied by a factor less than $(2n+1)^\phitheta$. 
We bound each of these diagrams separately.

In the diagrammatic notation below, when red ($\phi$) and blue ($\theta$) decorations lie on the edge of grey shaded composition diagrams, this indicates that a decoration lies on \emph{some edge} of the diagram (and it is not a decoration bridging the distance between the two endpoints).

Ignoring the $\theta$ decoration for a moment (i.e., for $\theta=0$), 
there are three qualitatively different possibilities for the location of the $\phi$ decoration: 
\begin{enumerate}
\item \label{case:zerothSegment} 
The $\phi$ decoration lies in the $\psi_0$ segment;

\item \label{case:middleSegment} 
The $\phi$ decoration lies in the $j$-th $\psi$ segment, where $j \in \{1, \dots, n-1\}$;

\item \label{case:endSegment} 
The $\phi$ decoration lies in the $\psi_n$ segment.
\end{enumerate}
We treat the three cases differently.
In case \ref{case:zerothSegment}, we bound (using Minkowski's integral inequality)
\begin{multline} \label{eqn:CaseOneSplit}
\norm*{\PiNCaseOnePhi{0}{x}}_{L^p_x}= \norm*{\PiNCaseOnePhi{x}{0}}_{L^p_x}\\
    \leq \int_{\R^{2d}} \Biggnorm{\PiNCaseOnePhiPtOne{x}{a}{b}}_{L^p_x}
    \RightBlockBar{a}{b}{\orig}{n-1} 
    \lambda^2\dd a\dd b.
\end{multline}
(The notation here means that the $\phi$ decoration is on some edge in the $\psi_0$ part of the diagram.)
In case \ref{case:middleSegment}, we bound
\begin{align}
    &\norm*{\PiNCaseTwoPhi}_{L^p_x} \nonumber\\
    &\leq \int_{\R^{2d}} \LeftBlock{0}{a}{b}{j-1}
    \Biggnorm{\PiNCaseTwoPhiPtTwo{a}{b}{x}}_{L^p_x}\lambda^2\dd a \dd b\nonumber\\
    &\leq \int_{\R^{4d}} \LeftBlock{0}{a}{b}{j-1}\Biggnorm{\PiNCaseTwoPhiPtThree{a-x}{b-x}{c}{d}}_{L^p_x}
    \PiNCaseTwoPhiPtFour{c}{d}{0}
    \lambda^4\dd a \dd b \dd c \dd d. \label{eqn:CaseTwoSplit}
\end{align}
Note that, depending on the parity of $j$, the $\phi$ decoration and the $\circ$ edge each can appear on the opposite side of the diagram. 
In case \ref{case:endSegment}, we bound
\begin{equation} 	\label{eqn:CaseThreeSplit}
    \norm*{\PiNCaseThreePhi}_{L^p_x}  
    \leq \int_{\R^{2d}} 
    \LeftBlock 0 a b {n-1}
    \Biggnorm{\PiNCaseThreePhiPartTwo{a}{b}{x}}_{L^p_x}
    \lambda^2\dd a\dd b.
\end{equation}
For each case, we have extracted a segment of the diagram with the $L^p_x$ norm, which we call \emph{special}. (In cases~\ref{case:zerothSegment} and~\ref{case:middleSegment}, the special segment does not coincide with a $\psi_0$ or $\psi$ segment.)

When there is a $\theta$ decoration, we further split into subcases according to the relative position of the $\theta$ decoration and the special segment:
\begin{enumerate}[label=\alph*)]
    \item \label{case:sameTheta} The $\theta$ decoration lies in the special $L^p_x$ segment;
    \item \label{case:earlierTheta} The $\theta$ decoration lies in a $\psi$  or $\psi_0$ segment to the left of the special $L^p_x$ segment;
    \item \label{case:laterTheta} The $\theta$ decoration lies in a $\bar \psi$  or $\bar \psi_0$ segment to the right of the special $L^p_x$ segment.
\end{enumerate}
Note that cases \ref{case:zerothSegment}\ref{case:earlierTheta} and \ref{case:endSegment}\ref{case:laterTheta} are necessarily null.
For the \ref{case:sameTheta} family of cases, 
the only difference to the bounds \eqref{eqn:CaseOneSplit}--\eqref{eqn:CaseThreeSplit} is the $L^p_x$ segment, and we will prove in Appendix~\ref{sec:DiagramBounds}, by expanding out the $\psi$ functions using their definition, that there is a constant $C\spec < \infty$ (depending only on the diagrams in the hypothesis) such that
\begin{multline} \label{eq:special_bound}
	 \max\Biggl\{\Biggnorm{\PiNCaseOnePhiPtOneDouble{x}{a}{b}}_{L^p_x}, \Biggnorm{\PiNCaseTwoPhiPtThreeDouble{a-x}{b-x}{c}{d}}_{L^p_x},\Biggnorm{\PiNCaseTwoPhiPtThreeDoubleVTwo{a-x}{b-x}{c}{d}}_{L^p_x}, \\ 
		\Biggnorm{\PiNCaseThreePhiPartTwoDouble{a}{b}{x}}_{L^p_x},\Biggnorm{\PiNCaseThreePhiPartTwoDoubleVTwo{a}{b}{x}}_{L^p_x}\Biggr\}\leq C\spec 
\end{multline}
uniformly in $a,b,c,d\in\Rd$ and in $\lambda\in [ \half\lambda_c, \lam_c)$.
Using this bound for the special segment, and using Lemma~\ref{lem:left_right} to bound the remaining integrals in \eqref{eqn:CaseOneSplit}--\eqref{eqn:CaseThreeSplit},
all diagrams in the \ref{case:sameTheta} family can be bounded by
\begin{equation}
C\spec \times
\begin{cases}
2U\crit V\crit^{n-1}		 & (\text{case 1a)}) \\
6U\crit^2 V\crit^{n-2}	& (\text{case 2a)}) \\
3U\crit V\crit^{n-1}		& (\text{case 3a)}) .
\end{cases}
\end{equation}
(Note that case \ref{case:middleSegment} only exists if $n\geq 2$.)

Now we consider the \ref{case:earlierTheta} family of cases. 
We still use \eqref{eq:right} in Lemma~\ref{lem:left_right} to bound the right $\bar \psi$ segment. 
We bound the special $L^p_x$ segment by setting $\theta=0$ in the bounds of Appendix~\ref{sec:DiagramBounds}; this gives another constant $C\spec'$ which bounds the same diagrams in \eqref{eq:special_bound} with no $\theta$ decorations. 
For the left $\psi$ segment that now contains the $\theta$ decoration, we only need to make a minor adjustment to the proof of \eqref{eq:left} of Lemma~\ref{lem:left_right}. 
Since there is only one $\theta$ decoration on some edge, 
either one single occurrence of a triangle becomes a $\theta$-decorated triangle ($\bar T^{(0)} \mapsto \bar T^{(\theta)}$), or a single occurrence of a wedge becomes a $\theta$-decorated wedge ($\bar B^{(0,0)}_1 \mapsto \bar B^{(\theta,0)}_1$), or an adjacency becomes a $\theta$-decorated adjacency ($\norm*{\connf}_1\mapsto \norm{ \abs x^\theta \adj(x) }_1$ in $U\crit$). In the worst-case scenario, a single small $V\crit^2$ factor is replaced by an $O(1)$ factor. 
Therefore, there is a constant $C\zuo < \infty$ such that for all integer $m \ge 0$, 
\begin{equation}
\int_{\R^{2d}} \PiNCaseTwoGeneralPtOneBlueM{0}{a}{b}{m} \lambda^2\dd a \dd b 
\le C\zuo  V\crit ^{(m-2) \vee 0} ,
\end{equation}
and all diagrams in the \ref{case:earlierTheta} family can be bounded by
\begin{equation}
C\zuo C\spec' \times
\begin{cases}
2U\crit V\crit^{(n-4) \vee 0}	& (\text{case 2b)}) \\
V\crit^{(n-3)\vee 0}			& (\text{case 3b)}) .
\end{cases}
\end{equation}

The \ref{case:laterTheta} family of cases is very similar. 
We use \eqref{eq:left} in Lemma~\ref{lem:left_right} to bound the left $\psi$ segment, and we bound the right decorated segment by
\begin{equation}
\int_{\R^{2d}} \PiNCaseOneGeneralPtTwoBlueM{a}{b}{0}{m} \lambda^2\dd a\dd b 
\leq C\you V\crit^{(m-2) \vee 0} 
\end{equation}
for all $m\ge 0$. All diagrams in the \ref{case:laterTheta} family can then be bounded by
\begin{equation}
C\spec' C\you \times
\begin{cases}
V\crit^{(n-3)\vee0}		 	& (\text{case 1c)}) \\
3U\crit V\crit^{(n-4)\vee 0}	& (\text{case 2c)}) .
\end{cases}
\end{equation}

Summing over all $(2n+1)^2$ diagrams and including the $(2n+1)^\phitheta$ factor then give the desired \eqref{prop:Pi_n}, since each case contains a power of $V\crit$ that is smaller than $V\crit^{(n-4) \vee 0}$.
\end{proof}

\appendix

\section{Bounds for the $L^p$ segment}
\label{sec:DiagramBounds}

During the proof of Proposition~\ref{prop:Pi_n} (equation  \eqref{eq:special_bound}),
we needed to use that the $L^p$ norm of certain diagrams, which we called the special segment, can be bounded uniformly in $\lambda \in [\half \lam_c, \lam_c)$. 
Here we present the bounds for each of them.

\subsection{Case~1a)}
For this case, we need to bound
\begin{equation}
	\Biggnorm{\PiNCaseOnePhiPtOneDouble{x}{a}{b}}_{L^p_x} ,
\end{equation}
where $\psi_0 = \psi_0 \supk 1 + \psi_0 \supk 2 + \psi_0 \supk 3$.
This diagram does not appear for $\psi_0 \supk 2$ and $\psi_0 \supk 3$, since for them the vertices $x$ and $a$ are adjacent.
We therefore only need to consider $\psi_0 \supk1$,
and we have the following two cases:
\begin{align}
\norm*{\PiNCaseOneJOneStart{x}{a}{b}}_{L^p_x} &\leq \bar W_p^{(\phi,\theta)}(\bar T^{(0)}+\bar B^{(0,0)}_1) ,
\\
\norm*{\PiNCaseTwoGeneralPartOne{x}{a}{b}}_{L^p_x}&\leq \bar W^{(\phi,0)}_p \bar T^{(\theta)} .
\end{align}

\subsection{Case~3a)}
For this case, we need to bound
\begin{equation}
\Biggnorm{\PiNCaseThreePhiPartTwoDouble{a}{b}{x}}_{L^p_x}
\quad\text{and}\quad
\Biggnorm{\PiNCaseThreePhiPartTwoDoubleVTwo{a}{b}{x}}_{L^p_x},
\end{equation}
where $\psi_n = \psi_n^{(1)} + \psi_n^{(2)}$.
For $\psi_n \supk 2$, both give the same diagram which is bounded by $\bar W_p \supphitheta$.
For $\psi_n \supk 1$, 
when the opposite edge of the $\phi$-edge is not contracted, we have
\begin{align}
\norm*{\PiNCaseThreeJOnePtFive{a}{b}{x}}_{L^p_x} ,
\norm*{\PiNCaseThreeJOnePtEight{a}{b}{x}}_{L^p_x} ,
\norm*{\PiNCaseThreeJOnePtEightFlipped{a}{b}{x}}_{L^p_x}
& \leq \bar W^{(\phi,\theta)}_p \bar T^{(0), \circ} ,
\\
\norm*{\PiNCaseThreeJOnePtSeven{a}{b}{x}}_{L^p_x}, \norm*{\PiNCaseThreeJOnePtSix{a}{b}{x}}_{L^p_x} ,
\norm*{\PiNCaseThreeJOnePtSixFlipped{a}{b}{x}}_{L^p_x} 
& \leq \bar W^{(\phi,0)}_p  \bar T^{(\theta), \circ} .
\end{align}
When the opposite edge of the $\phi$-edge is contracted, we have
\begin{align}
\norm*{\PiNCaseThreeJOnePtThreeSwap{a}{b}{x}}_{L^p_x}
&\leq \bar E \supphi \biggnorm{ \sideTriangleBlue a x}_{L^p_x} 
\lesssim \bar E \supphi  	(\bar T\suptheta + \bar E\suptheta \bar B_1\supk{0,0}) ,
\end{align}
by \eqref{eq:p-triangle}.

\subsection{Case~2a)}

For this case, we need to bound
\begin{equation} \label{eq:goal_2a}
\Biggnorm{\PiNCaseTwoPhiPtThreeDouble{a-x}{b-x}{c}{d}}_{L^p_x} 
\quad \text{and}\quad
\Biggnorm{\PiNCaseTwoPhiPtThreeDoubleVTwo{a-x}{b-x}{c}{d}}_{L^p_x} ,
\end{equation}
where $\psi = \psi \supk 1 + \psi \supk 2 + \psi \supk 3 + \psi \supk 4$.

We start from $\psi \supk 4$, for which only the second diagram of \eqref{eq:goal_2a} appears. We have
\begin{align}
\norm*{\PiNDiagramTwentySix{a-x}{b-x}{c}{d}}_{L^p_x} 
& \stackrel{ \eqref{eq:p-triangle} }{\lesssim}
\bar E \supphi  (\bar T\suptheta + \bar E\suptheta \bar B_1\supk{0,0}) ,
\\
\norm*{\PiNDiagramTwentySeven{a-x}{b-x}{c}{d}}_{L^p_x} 
& \leq (1) \bar W^{(\phi,\theta)}_p ,
\end{align}
where we used $\tlam(a-x-d) \le 1$ in the last line.

For $\psi \supk 3$, we have
\begin{align}
\norm*{\PiNDiagramTwentyOne{a-x}{b-x}{c}{d}}_{L^p_x} ,
\norm*{\PiNDiagramTwentyOneFlipped{a-x}{b-x}{c}{d}}_{L^p_x} 
& \leq \bar W^{(\phi,\theta)}_p \bar T^{(0), \circ} , 
\\
\norm*{\PiNDiagramTwentyTwo{a-x}{b-x}{c}{d}}_{L^p_x} ,
\norm*{\PiNDiagramTwentyTwoFlipped{a-x}{b-x}{c}{d}}_{L^p_x} 
& \leq \bar W^{(\phi,0)}_p \bar T^{(\theta), \circ} .
\end{align}

For $\psi \supk 1$, 
when the opposite edge of the $\phi$-edge is not contracted, we have
\begin{align}
\norm*{\PiNCaseTwoJOnePtFiveVTwo{a-x}{b-x}{c}{d}}_{L^p_x} ,
\norm*{\PiNCaseTwoJOnePtEightVTwo{a-x}{b-x}{c}{d}}_{L^p_x} 
& \leq \bar W^{(\phi,\theta)}_p \bar T^{(0), \circ}\bar T^{(0), \circ} ,
\\
\norm*{\PiNCaseTwoJOnePtSevenVTwo{a-x}{b-x}{c}{d}}_{L^p_x},
\norm*{\PiNCaseTwoJOnePtSixVTwo{a-x}{b-x}{c}{d}}_{L^p_x} 
& \leq \bar W^{(\phi,0)}_p \bar T^{(\theta), \circ} \bar T^{(0), \circ} ,
\\
\norm*{\PiNCaseTwoMidPartThree{a-x}{b-x}{c}{d}}_{L^p_x}, \norm*{\PiNCaseTwoMidPartFour{a-x}{b-x}{c}{d}}_{L^p_x}
& \leq \bar W^{(\phi,0)}_p \bar T^{(\theta)} \bar T^{(0), \circ} ,
\end{align}
with similar upper bounds when the contracted edges lie on the opposite side of the diagram (for the second diagram of \eqref{eq:goal_2a}).
When the opposite edge of the $\phi$-edge is contracted, which only happens for the second diagram of \eqref{eq:goal_2a}, we have
\begin{align}
\norm*{\PiNCaseTwoJOnePtThreeVTwoSwapL{a-x}{b-x}{c}{d}}_{L^p_x}
& \stackrel{ \eqref{eq:p-triangle} }{\lesssim}
\bar E \supphi  (\bar T\suptheta + \bar E\suptheta \bar B_1\supk{0,0}) \bar T ^{(0), \circ} ,
\\
\norm*{\PiNCaseTwoJOnePtThreeVTwoSwapR{a-x}{b-x}{c}{d}}_{L^p_x}
& \stackrel{ \eqref{eq:p-triangle} }{\lesssim}
\bar E \supphi  (\bar T\supzero + \bar E\supzero \bar B_1\supk{0,0}) \bar T ^{(\theta), \circ} .
\end{align}

Finally, for $\psi \supk 2$, we start from the first diagram of \eqref{eq:goal_2a}. 
When the vertical edge in the middle of the diagram is contracted, we have
\begin{align}
\label{eq:red_top_left1}
\norm*{\PiNCaseTwoJTwoPtNineVOne{a-x}{b-x}{c}{d}}_{L^p_x} 
& \leq \bar W^{(\phi,\theta)}_p \bar B^{(0,0)}_1 \bar T^{(0),\circ} ,
\\
\norm*{\PiNCaseTwoJTwoPtTenVOne{a-x}{b-x}{c}{d}}_{L^p_x} 
& \leq \bar W^{(\phi,0)}_p \bar B^{(\theta,0)}_1 \bar T^{(0),\circ} ,
\\
\norm*{\PiNDiagramThree{a-x}{b-x}{c}{d}}_{L^p_x} 
& \leq \bar W^{(\phi,0)}_p \bar B_1^{(0,0)} \bar T^{(\theta)} ,
\end{align}
and we note that $\bar B^{(\theta,0)}_1 \lesssim \bar E \suptheta + \bar T \suptheta < \infty$ by Lemma~\ref{lem:extra}.
When the vertical edge is not contracted, we have
\begin{align}
\label{eq:red_top_left2}
\norm*{\PiNCaseTwoJTwoPtElevenVOne{a-x}{b-x}{c}{d}}_{L^p_x} 
& \leq \bar W^{(\phi,\theta)}_p \bar T^{(0)}\bar T^{(0), \circ} ,
\\
\norm*{\PiNCaseTwoJTwoPtTwelveVOne{a-x}{b-x}{c}{d}}_{L^p_x} ,
\norm*{\PiNDiagramFour{a-x}{b-x}{c}{d}}_{L^p_x} 
& \leq \bar W_p \supk{\phi,0} \bar T\supthetao \bar T^{(0), \circ} .
\end{align}

For the second diagram of \eqref{eq:goal_2a}, 
we first consider the case where the $\phi$-edge is at the bottom-left of the diagram.
In this case, if the top-left edge is not contracted, we use
\begin{align}
\norm*{\PiNDiagramFifteen{a-x}{b-x}{c}{d}}_{L^p_x} 
& \leq \bar W^{(\phi,0)}_p \bar B_1^{(0,0)} \bar T^{(\theta), \circ} ,
\\
\norm*{\PiNDiagramNineteen{a-x}{b-x}{c}{d}}_{L^p_x}
& \leq \bar W^{(\phi,0)}_p \bar T^{(0)} \bar T^{(\theta), \circ} ,
\end{align}
and the bounds \eqref{eq:red_top_left1}, \eqref{eq:red_top_left2} with the $\phi$ and $\theta$ decorations switched.
If the top-left edge is contracted, we have
\begin{align}
\norm*{\PiNDiagramFive{a-x}{b-x}{c}{d}}_{L^p_x} 
& \stackrel{ \eqref{eq:p-triangle} }{\lesssim}
	\bar E^{(\phi)} (1) 
	(\bar T\suptheta + \bar E\suptheta \bar B_1\supk{0,0}) ,
\\
\norm*{\PiNDiagramSeven{a-x}{b-x}{c}{d}}_{L^p_x} 
& \stackrel{ \eqref{eq:p-triangle} }{\lesssim}
	\bar E \supphi \bar B_1 \supk{0,0} 
	 (\bar T\suptheta + \bar E\suptheta \bar B_1\supk{0,0}) ,
\\
\norm*{\PiNDiagramEleven{a-x}{b-x}{c}{d}}_{L^p_x} 
& \leq \bar E^{(\phi)} 
	(  \lam\inv \bar T \supthetao \bar T \supzero
	+ \bar E \suptheta \bar B_1 \supk{0,0} \bar T^{(0),\circ} ) ,
\end{align}
where we bounded the $L^p_x$ norm by the sum of $L^1_x$ and $L^\infty_x$ norms in the last line.
For the other case where the $\phi$-edge is at the bottom-middle of the diagram, 
if the vertical edge in the middle of the diagram is contracted, using Young's convolution inequality we have
\begin{align}
\norm*{\PiNCaseTwoJTwoPtTenVOneSwap{a-x}{b-x}{c}{d}}_{L^p_x} ,
\norm*{\PiNDiagramSixteen{a-x}{b-x}{c}{d}}_{L^p_x} 
& \stackrel{ \eqref{eq:p-triangle} }{\lesssim}
	\bar E\supphi (\bar T\suptheta + \bar E\suptheta \bar B_1\supk{0,0})
	\bar T^{(0),\circ} ,
\\
\norm*{\PiNDiagramSixteenFlipped{a-x}{b-x}{c}{d}}_{L^p_x} 
& \stackrel{ \eqref{eq:p-triangle} }{\lesssim}
	\bar E\supphi (\bar T\supzero + \bar E\supzero \bar B_1\supk{0,0})
	\bar T^{(\theta),\circ} ,
\\
\norm*{\PiNDiagramSix{a-x}{b-x}{c}{d}}_{L^p_x} 
= \norm*{\PiNDiagramSix{a}{b}{c+x}{d+x}}_{L^p_x} 
& \leq (1) \bar W^{(\phi,\theta)}_p \bar B^{(0,0)}_1 ,
\end{align}
where we used $\tlam(a-d-x) \le 1$ in the last line. 
If the vertical edge is not contracted, we have
\begin{align}
\norm*{\PiNCaseTwoJTwoPtTwelveVOneSwap{a-x}{b-x}{c}{d}}_{L^p_x} ,
\norm*{\PiNDiagramTwenty{a-x}{b-x}{c}{d}}_{L^p_x} 
& \leq \bar H_p^{(\phi,\theta)} ,
\\
\norm*{\PiNCaseTwoJTwoPtTwelveSwap{a-x}{b-x}{c}{d}}_{L^p_x} ,
\norm*{\PiNDiagramTwelve{a-x}{b-x}{c}{d}}_{L^p_x} 
& \stackrel{ \eqref{eq:p-triangle} }{\lesssim}
	\bar E\supphi (\bar T\suptheta + \bar E\suptheta \bar B_1\supk{0,0})
	\bar T \supzero ,
\\
\norm*{\PiNDiagramEighteen{a-x}{b-x}{c}{d}}_{L^p_x}
& \stackrel{ \eqref{eq:p-triangle} }{\lesssim}
	\bar E\supphi (\bar T\supzero + \bar E\supzero \bar B_1\supk{0,0})
	\bar T \suptheta ,
\\
\norm*{\PiNDiagramTen{a-x}{b-x}{c}{d}}_{L^p_x}
= \norm*{\PiNDiagramTen{a}{b}{c+x}{d+x}}_{L^p_x}
& \leq (1) \bar W^{(\phi,\theta)}_p \bar T^{(0)} .
\end{align}
This covers all possible cases and concludes the proof of \eqref{eq:special_bound}.

\section*{Acknowledgements}
We thank Gordon Slade for comments on a preliminary version.
The work was supported in part by NSERC of Canada.
The work of YL was supported in part by 
the National Natural Science Foundation of China (Grant No.~12595284, 12595280).

\end{document}